\documentclass{tac}

\author{Soichiro Fujii}

\date{\today}

\thanks{The author is supported by JST ERATO HASUO Metamathematics for Systems Design Project (No.~JPMJER1603) and the JSPS-INRIA Bilateral Joint Research Project ``CRECOGI''.
}

\address{Research Institute for Mathematical Sciences, Kyoto University\\
 Kyoto 606-8502, Japan}

\title {A unified framework\\ for notions of algebraic theory}

\copyrightyear{2019}


\eaddress{s.fujii.math@gmail.com}

\usepackage{mathrsfs}
\usepackage{amsfonts, amssymb}
\usepackage{mathtools}
\usepackage{comment}
\usepackage{thmtools}
\usepackage{tikz}
\usetikzlibrary{arrows.meta,decorations.pathreplacing,decorations.markings,shapes,calc}
\usetikzlibrary{matrix,arrows,decorations.pathmorphing,backgrounds,decorations.markings,positioning}
\tikzset{2cell/.style={-implies,double,double equal sign distance,shorten 
>=2pt, shorten <=3pt}}
\tikzset{2cellshort/.style={-implies,double,double equal sign distance,shorten 
>=4pt, shorten <=5pt}}
\tikzset{2cellr/.style={implies-,double,double equal sign distance,shorten 
>=3pt, shorten <=2pt}}
\tikzset{3cell/.style={-implies,double,double distance=2.5pt,shorten >=2pt, 
shorten <=3pt}}
\tikzset{labelsize/.style={font=\scriptsize}}
\tikzset{string/.style={very thick}}
\tikzset{
  pto/.style={->,postaction={decorate},
    decoration={
        markings,
        mark=at position 0.5 with {\arrow{|}}}
  },
}
\newcommand{\tzsquare}[3]{
\draw[2cell] (#1,#2)++(0,0.25) to node[auto,labelsize]{#3} ++(0,-0.5);}

\newcommand{\tzsquareup}[3]{
\draw[2cell] (#1,#2)++(0,-0.25) to node[auto,swap,labelsize]{#3} ++(0,0.5);}
\newcommand{\tzsquareupswap}[3]{
\draw[2cell] (#1,#2)++(0,-0.25) to node[auto,labelsize]{#3} ++(0,0.5);}

\mathchardef\mhyphen="2D

\newcommand{\CAT}{{\mathbf{CAT}}}
\newcommand{\Cat}{\mathbf{Cat}}
\newcommand{\Set}{\mathbf{Set}}
\newcommand{\SET}{\mathbf{SET}}
\newcommand{\tCAT}{\mathscr{C\!A\!T}}
\newcommand{\tCat}{{\mathscr{C}\!at}}

\newcommand{\MonCATls}{{\mathscr{M}\!on\mathscr{C\!A\!T}^{\mathrm{ls}}_{\mathrm{lax}}}}
\newcommand{\SymMonCATls}{{\mathscr{S}\!ym\mathscr{M}\!on
\mathscr{C\!A\!T}^{\mathrm{ls}}_{\mathrm{lax}}}}
\newcommand{\SymMonCATstls}{{\mathscr{S}\!ym\mathscr{M}\!on
\mathscr{C\!A\!T}^{\mathrm{ls}}_{\mathrm{strong}}}}
\newcommand{\SymMonCAT}{{\mathscr{S}\!ym\mathscr{M}\!on
\mathscr{C\!A\!T}_{\mathrm{lax}}}}
\newcommand{\MonCAT}{{\mathscr{M}\!on\mathscr{C\!A\!T}_{\mathrm{lax}}}}

\newcommand{\MonCATstls}{{\mathscr{M}\!on\mathscr{C\!A\!T}^{\mathrm{ls}}_{\mathrm{strong}}}}
\newcommand{\MonCATol}{{\mathscr{M}\!on\mathscr{C\!A\!T}_{\mathrm{oplax}}}}
\newcommand{\twoCAT}{\underline{2\mhyphen\mathscr{C\!A\!T}}}

\newcommand{\Str}{\mathrm{Str}}
\newcommand{\Sem}{\mathrm{Sem}}

\newcommand{\univ}[1]{\mathcal{#1}}

\newcommand{\PROF}{{\mathscr{P\!R\!O\!F}}}

\newcommand{\PPROF}{{\mathbb{PROF}}}

\newcommand{\FPow}{{\mathscr{F\!P}\!ow}}
\newcommand{\FProd}{{\mathscr{F\!P}\!rod}}

\newcommand{\tcat}[1]{\mathscr{#1}}

\newcommand{\monoid}[1]{\mathsf{#1}}
\newcommand{\End}{\monoid{End}}

\newcommand{\Th}[1]{{\mathbf{Th}(#1)}}
\newcommand{\Mon}[1]{{\mathbf{Mon}(#1)}}
\newcommand{\FinSet}{\mathbf{FinSet}}

\newcommand{\enrich}[2]{{\langle #1, #2\rangle}}

\newcommand{\Simp}{\Delta_{a}}
\newcommand{\ord}[1]{\mathbf{#1}}
\newcommand{\dcat}[1]{\mathbb{#1}}

\newcommand{\F}{\mathbf{F}}
\newcommand{\Pcat}{\mathbf{P}}
\newcommand{\Ncat}{\mathbf{N}}
\newcommand{\NN}{\mathbb{N}}
\newcommand{\pres}[2]{{{\langle\, #1\, |\, #2 \,\rangle}}}
\newcommand{\interp}[1]{{[\![#1]\!]}}
\newcommand{\interpp}[1]{{[\![#1]\!]'}}
\newcommand{\group}{\mathrm{Grp}}

\newcommand{\Mod}[2]{{\mathbf{Mod}(#1,#2)}}

\newcommand{\olAct}[1]{{\mathscr{A}\!ct_{\mathrm{oplax}}(#1)}}
\newcommand{\olActl}[1]{{\mathscr{A}\!ct^{\mathrm{l}}_{\mathrm{oplax}}(#1)}}
\newcommand{\Enrich}[1]{{\mathscr{E}\!nrich(#1)}}

\newcommand{\Enrichr}[1]{{\mathscr{E}\!nrich^{\mathrm{r}}(#1)}}
\newcommand{\MtMod}[1]{{\mathscr{M}\!\mathscr{M}\!od(#1)}}
\newcommand{\MtModfib}[2]{{\mathbf{MMod}_{#1}(#2)}}

\newcommand{\MtTH}{{\mathscr{M}\!\mathscr{T\!H}}}

\makeatletter
\DeclareRobustCommand{\rvdots}{%
  \vbox{
    \baselineskip4\p@\lineskiplimit\z@
    \kern-\p@
    \hbox{.}\hbox{.}\hbox{.}
  }}
\makeatother

\makeatletter
\newcommand{\pto}{}
\newcommand{\pgets}{}
\DeclareRobustCommand{\pto}{\mathrel{\mathpalette\p@to@gets\to}}
\DeclareRobustCommand{\pgets}{\mathrel{\mathpalette\p@to@gets\gets}}
\newcommand{\p@to@gets}[2]{%
  \ooalign{\hidewidth$\m@th#1\mapstochar\mkern5mu$\hidewidth\cr$\m@th#1\longrightarrow$\cr}%
}
\makeatother

\newcommand{\ptensor}{\odot}
\newcommand{\ptensorrev}{\mathbin{\ptensor^\mathrm{rev}}}

\newcommand{\N}{\mathbb{N}}

\newcommand{\id}[1]{{\mathrm{id}_{#1}}}
\newcommand{\ob}[1]{{\mathrm{ob}(#1)}}

\newcommand{\op}{{\mathrm{op}}}
\newcommand{\co}{{\mathrm{co}}}
\newcommand{\coop}{{\mathrm{coop}}}

\newcommand{\ar}[1]{{\mathrm{ar}_{#1}}}

\newcommand{\name}[1]{{\lceil #1\rceil}}

\newcommand{\cat}[1]{\mathcal{#1}}

\newcommand{\DayI}{\widehat{I}}
\newcommand{\Dayo}{\mathbin{\widehat{\otimes}}}
\newcommand{\Lan}{\mathrm{Lan}}
\newcommand{\Ran}{\mathrm{Ran}}

\newcommand{\enGph}[1]{{{#1}\mhyphen\mathbf{Gph}}}

\newcommand{\Mnd}[1]{{\mathbf{Mnd}(#1)}}

\newcommand{\defemph}[1]{\textbf{#1}}

\newcommand{\Operad}[1]{{{#1}\mhyphen\mathbf{Opd}}}
\newcommand{\arity}[1]{{\mathrm{ar}_{#1}}}

\begin{document}
\maketitle
\begin{abstract}
Universal algebra uniformly captures various algebraic structures, 
by expressing them  
as equational theories or abstract clones. The ubiquity of 
algebraic structures in mathematics and related fields 
has given rise to several variants of 
universal algebra, such as theories of symmetric operads, 
non-symmetric operads, generalised operads, PROPs, PROs, and monads.
These variants of universal algebra are called
{notions of algebraic theory}.
In this paper, we develop a unified framework 
for them.
The key observation is that each notion of algebraic theory can be 
identified with a monoidal category, in such a way that algebraic theories 
correspond to monoid objects therein. 
To incorporate semantics, we 
introduce a categorical structure called {metamodel}, which 
formalises a definition of models of algebraic theories.
We also define morphisms between notions of algebraic theory, which are
a monoidal version of profunctors.
Every strong monoidal functor gives rise to an adjoint pair of such morphisms,
and provides a uniform method to establish isomorphisms between categories of 
models in different notions of algebraic theory.
A general structure-semantics adjointness result and a double categorical
universal property of categories of models are also shown.
\end{abstract}
\tableofcontents

\pagebreak
\section{Introduction}

Algebras permeate both pure and applied mathematics.
Important types of algebras, such as vector spaces, groups and rings, arise 
naturally in many branches of mathematical sciences and 
it would hardly be an exaggeration to say that algebraic structures 
are one of the most universal and fundamental structures in mathematics.

A type of algebras, such as groups, is normally specified by a family of 
operations and a family of equational axioms. 
We call such a specification of a type of algebras an \emph{algebraic theory},
and call a background theory for a type of algebraic theories a \emph{notion of 
algebraic theory}. 
In order to capture various types of algebras, a variety of notions of 
algebraic theory have been introduced.
Examples include {universal algebra} \cite{Birkhoff_abst_alg},
symmetric and non-symmetric operads \cite{May_loop},
generalised operads (also called clubs)
\cite{Burroni_T_cats,Kelly_club_data_type,Hermida_representable,Leinster_book}, 
PROPs  and PROs \cite{MacLane_cat_alg},
and monads \cite{Eilenberg_Moore,Linton_equational};
we shall review these notions of algebraic theory in 
Section~\ref{sec:notions_of_alg_thy}.

Notions of algebraic theory all aim to provide a means to define algebras, but 
they attain this goal in quite distinct manners.
The diversity of the existing notions of algebraic theory 
leaves one to wonder what, if any, is a formal core or essence shared by them.
Our main aim in this paper is to provide an answer to this 
question, by developing a unified 
framework for notions of algebraic theory.

\medskip

The starting point of our approach is quite simple. 
We identify a notion of algebraic theory with an (arbitrary) monoidal category,
and algebraic theories in a notion of algebraic theory with monoid objects in
the corresponding monoidal category.
As we shall see in Section~\ref{sec:metatheories_theories},
it has been observed (or easily follows from known observations)
that each type of algebraic theories we have 
listed above can be characterised as monoid objects in a suitable monoidal 
category.
From now on let us adopt the terminology to be introduced in 
Section~\ref{sec:metatheories_theories}:
we call a monoidal category a \emph{metatheory} and a monoid object therein
a \emph{theory}, to remind ourselves of our intention. 

In order to formalise the semantical aspect of notions of algebraic theory---by
which we mean definitions of \emph{models} (= {algebras}) of an algebraic 
theory, their homomorphisms, and so on---we introduce the concept of 
\emph{metamodel}.
Metamodels are a certain categorical structure defined relative to   
a metatheory $\cat{M}$ and a category $\cat{C}$, 
and are meant to capture a 
\emph{notion of model} of an algebraic theory, i.e., what it means to take a 
model of a theory in $\cat{M}$ 
in the category $\cat{C}$.
A model of an algebraic theory is always given relative to some notion of 
model, even though usually it is not recognised explicitly.
We shall say more about the idea of notions of model at the beginning of 
Section~\ref{sec:enrichment}.
A metamodel of a metatheory $\cat{M}$ in a category $\cat{C}$
generalises both an $\cat{M}$-category (as in enriched category theory) whose underlying category is $\cat{C}$,
and a {(left) oplax action} of $\cat{M}$ on $\cat{C}$.
Indeed, as we shall see in Sections~\ref{sec:enrichment} and 
\ref{sec:oplax_action}, it has been observed that enrichments
(which we introduce as a slight variant of $\cat{M}$-categories)
and oplax actions can account for the standard semantics of the known
notions of algebraic theory.
Our concept of metamodel provides a unified account of the semantical aspects 
of notions of algebraic theory.

Metamodels of a fixed metatheory $\cat{M}$ naturally form a 
2-category $\MtMod{\cat{M}}$,
and we shall see that theories in $\cat{M}$ can be identified with
certain metamodels of $\cat{M}$ in the terminal category $1$.
This way we obtain a fully faithful 2-functor 
from the category $\Th{\cat{M}}$ of theories in $\cat{M}$ (which is just the 
category of monoid objects in $\cat{M}$)
to $\MtMod{\cat{M}}$.
A metamodel $\Phi$ of $\cat{M}$ in $\cat{C}$ provides a definition of model of 
a theory in $\cat{M}$ as an object of $\cat{C}$ with additional 
structure,
hence if we fix a metamodel $(\cat{C},\Phi)$ and a theory $\monoid{T}$,
we obtain the 
category of models $\Mod{\monoid{T}}{(\cat{C},\Phi)}$ equipped with 
the forgetful functor 
$U\colon\Mod{\monoid{T}}{(\cat{C},\Phi)}\longrightarrow\cat{C}$.
By exploiting the 2-category $\MtMod{\cat{M}}$, the 
construction $\Mod{-}{-}$ of categories of models may be expressed as 
the following composition 
\begin{equation}
\label{eqn:Mod_as_hom}
\begin{tikzpicture}[baseline=-\the\dimexpr\fontdimen22\textfont2\relax ]
      \node (1) at (0,1.5)  {$\Th{\cat{M}}^\op\times \MtMod{\cat{M}}$};
      \node (2) at (0,0)  {$\MtMod{\cat{M}}^\op\times \MtMod{\cat{M}}$};
      \node (3) at (0,-1.5) {$\tCAT$,};
      \draw[->] (1) to node[auto,labelsize]{inclusion} (2);
      \draw[->] (2) to node[auto,labelsize] {$\MtMod{\cat{M}}(-,-)$} (3);
\end{tikzpicture}
\end{equation}
where $\MtMod{\cat{M}}(-,-)$ is the hom-2-functor
and $\tCAT$ is a 2-category of categories.

We also introduce morphisms (and 2-cells) between metatheories
(Section \ref{sec:morphism_metatheory}).
Such morphisms are a monoidal version of {profunctors}.
The principal motivation of the introduction of morphisms of metatheories is 
to compare different notions of algebraic theory,
and indeed our morphisms of metatheories induce 2-functors between the 
corresponding 2-categories of metamodels.
Analogously to the well-known fact for profunctors
that any functor induces an adjoint pair of profunctors,
we see that any strong monoidal functor $F$ induces an adjoint pair 
$F_\ast\dashv F^\ast$ of morphisms of metatheories.
Therefore, whenever we have a strong monoidal functor 
$F\colon \cat{M}\longrightarrow\cat{N}$ between metamodels,
we obtain a 2-adjunction
\begin{equation}
\label{eqn:2-adjunction}
\begin{tikzpicture}[baseline=-\the\dimexpr\fontdimen22\textfont2\relax ]
      \node (L) at (0,0)  {$\MtMod{\cat{M}}$};
      \node (R) at (4.5,0)  {$\MtMod{\cat{N}}$.};
      \draw[->,transform canvas={yshift=5pt}]  (L) to node[auto,labelsize] 
      {$\MtMod{F_\ast}$} (R);
      \draw[<-,transform canvas={yshift=-5pt}]  (L) to 
      node[auto,swap,labelsize] {$\MtMod{F^\ast}$} 
      (R);
      \node[rotate=-90,labelsize] at (2.25,0)  {$\dashv$};
\end{tikzpicture} 
\end{equation}
Now, the strong monoidal functor $F$ also induces a functor 
\[
\Th{F}\colon\Th{\cat{M}}\longrightarrow\Th{\cat{N}},
\]
which is in fact a restriction of $\MtMod{F_\ast}$.
This implies that, immediately from the description (\ref{eqn:Mod_as_hom})
of categories of models and the 2-adjointness (\ref{eqn:2-adjunction}),
for any $\monoid{T}\in\Th{\cat{M}}$ and $(\cat{C},\Phi)\in\MtMod{\cat{N}}$,
we have a canonical isomorphism of categories
\begin{equation}
\label{eqn:iso_models}
\Mod{\Th{F}(\monoid{T})}{(\cat{C},\Phi)}\cong
\Mod{\monoid{T}}{\MtMod{F^\ast}(\cat{C},\Phi)}.
\end{equation}
In fact, as we shall see, the action of $\MtMod{-}$ on morphisms of metatheories
preserves the ``underlying categories'' of metamodels.
So $\MtMod{F^\ast}(\cat{C},\Phi)$ is a metamodel of $\cat{M}$ in $\cat{C}$,
and we have an isomorphism of categories \emph{over} $\cat{C}$
(that is, the isomorphism (\ref{eqn:iso_models}) commutes with the forgetful 
functors).

The above argument gives a unified conceptual account for a range of known 
results on the compatibility of semantics of notions of algebraic theory.
For example, it is known that any clone (or Lawvere theory) $\monoid{T}$
induces a monad $\monoid{T}'$ on $\Set$
in such a way that the models of $\monoid{T}$ and $\monoid{T'}$ in $\Set$
(with respect to the standard notions of model) coincide;
this result follows from the existence of a natural strong monoidal functor 
between the metatheories corresponding to clones and monads on $\Set$,
together with a simple observation that the induced 2-functor between
the 2-categories of metamodels preserves the standard metamodel.
This and other examples will be treated in 
Section~\ref{sec:morphism_metatheory}.

\medskip

In Section~\ref{sec:str-sem} we study \emph{structure-semantics adjunctions} 
within our framework. 
If we fix a metatheory $\cat{M}$ and a metamodel $(\cat{C},\Phi)$ of $\cat{M}$,
we obtain a functor 
\begin{equation}
\label{eqn:semantics_from_Th}
\Th{\cat{M}}^\op\longrightarrow \CAT/\cat{C}
\end{equation}
by mapping a theory $\monoid{T}$ in $\cat{M}$  to the category of models 
$\Mod{\monoid{T}}{(\cat{C},\Phi)}$ equipped with the forgetful functor 
into $\cat{C}$.
The functor (\ref{eqn:semantics_from_Th}) is sometimes called the 
\emph{semantics functor},
and it has been observed for many notions of algebraic theory that 
this functor (or an appropriate variant of it) admits a left adjoint
called the \emph{structure functor} 
\cite{Lawvere_thesis,Linton_equational,Linton_outline,Dubuc_Kan,Street_FTM,Avery_thesis}.
The idea behind the structure functor is as follows.
One can regard a functor $V\colon\cat{A}\longrightarrow\cat{C}$
into $\cat{C}$ as specifying an additional structure (in a very broad sense) on 
objects in $\cat{C}$, by viewing $\cat{A}$ as the category of $\cat{C}$-objects 
equipped with that structure, and $V$ as the forgetful functor.
The structure functor then maps $V$ to the best approximation of that structure
by theories in $\cat{M}$.
Indeed, if (\ref{eqn:semantics_from_Th}) is fully faithful (though
this is not always the case), then the structure functor reconstructs the 
theory from
its category of models.

We cannot get a left adjoint to the functor 
(\ref{eqn:semantics_from_Th}) for an arbitrary metatheory $\cat{M}$ and its 
metamodel $(\cat{C},\Phi)$.
In order to get general structure-semantics adjunctions,
we extend the category $\Th{\cat{M}}$ of theories in $\cat{M}$
to the category $\Th{\widehat{\cat{M}}}$ of theories in the metatheory 
$\widehat{\cat{M}}=[\cat{M}^\op,\SET]$ equipped with the {convolution 
monoidal structure}~\cite{Day_thesis}.
We show in Theorem~\ref{thm:str_sem_small} that 
the structure-semantics adjunction
\[
\begin{tikzpicture}[baseline=-\the\dimexpr\fontdimen22\textfont2\relax ]
      \node(11) at (0,0) 
      {$\Th{\widehat{\cat{M}}}^\op$};
      \node(22) at (4,0) {$\CAT/\cat{C}$};
  
      \draw [->,transform canvas={yshift=5pt}]  (22) to node 
      [auto,swap,labelsize]{$\Str$} (11);
      \draw [->,transform canvas={yshift=-5pt}]  (11) to node 
      [auto,swap,labelsize]{$\Sem$} (22);
      \path (11) to node[midway](m){} (22); 

      \node at (m) [labelsize,rotate=90] {$\vdash$};
\end{tikzpicture}
\] 
exists for any metatheory $\cat{M}$ and its metamodel $(\cat{C},\Phi)$.

\medskip

We conclude the paper in Section~\ref{sec:double_lim},
by giving a universal characterisation of categories of models in our 
framework.
It is well-known that the Eilenberg--Moore categories (= categories of models) 
of monads can be characterised by a 2-categorical universal property in the 
2-category $\tCAT$ of categories~\cite{Street_FTM}.
We show in Theorem~\ref{thm:double_categorical_univ_property}
that our category of models admit a similar universal characterisation,
but instead of inside the 2-category $\tCAT$, inside the \emph{pseudo double 
category} $\PPROF$ of categories, functors, profunctors and natural 
transformations.
The notion of pseudo double category, as well as $\PPROF$ itself, 
was introduced by Grandis and Par{\'e} \cite{GP1}.
In the same paper they also introduced the notion of \emph{double limit},
a suitable limit notion in (pseudo) double categories.
The double categorical universal property that our categories of models 
enjoy can also be formulated in terms of double limits;
see Corollary~\ref{cor:Mod_as_dbl_lim}.

\medskip

During the investigation of our framework, we have encountered a number of 
problems, several of which are still open.
A major open problem is that of characterising via intrinsic properties 
the forgetful functors from the categories of models
arising in our framework.
In the case of monads, the corresponding result is the various 
\emph{monadicity theorems}, such as Beck's theorem 
\cite[Section~VI.7]{MacLane_CWM}.
We shall discuss this problem further at the end of 
Section~\ref{sec:morphism_metatheory}.

\subsection{Acknowledgements}
We are grateful to  
Martin Hyland,
Pierre-Alain Jacqmin,
Shin-ya Katsumata,
Kenji Maillard,
Paul-Andr{\'e} Melli{\`e}s,
John Power and
Exequiel Rivas
for stimulating discussions and 
helpful comments.

\section{Set theoretic conventions}
\label{sec:foundational_convention}
Due to the metatheoretical nature of the subject matter, in this paper
we will use multiple (Grothendieck) \emph{universes};
see e.g., \cite[Section~I.6]{MacLane_CWM} or \cite[Definition~1.1.1]{KS:CS}
for definitions of universes.


For the purpose of this paper, we assume 
the existence of three universes $\univ{U}_1$,
$\univ{U}_2$ and $\univ{U}_3$ with $\univ{U}_1\in\univ{U}_2\in\univ{U}_3$. 
We now fix these universes once and for all.

Let $\univ{U}$ be a universe.
We define size-regulating conditions 
on sets and other mathematical structures
in reference to $\univ{U}$.
\begin{itemize}
\item A set is said to be 
\defemph{in $\univ{U}$} if 
it is an element of $\univ{U}$.
\end{itemize}
In this paper, a category is always assumed to have sets of objects
and of morphisms (rather than \emph{proper classes} of them).
We say that a category $\cat{C}$ is
\begin{itemize}
\item \defemph{in $\univ{U}$}
if the tuple $(\ob{\cat{C}},(\cat{C}(A,B))_{A,B\in\ob{\cat{C}}},
(\id{C}\in\cat{C}(C,C))_{C\in\ob{\cat{C}}},(\circ_{A,B,C}\colon$
$\cat{C}(B,C)\times\cat{C}(A,B)\longrightarrow\cat{C}(A,C))_{A,B,C\in
\ob{\cat{C}}})$, consisting of the data for $\cat{C}$,
is an element of $\univ{U}$;
\item \defemph{locally in $\univ{U}$}
if for each $A,B\in\ob{\cat{C}}$, 
the hom-set $\cat{C}(A,B)$ is in $\cat{U}$.
\end{itemize}
We also write $C\in\cat{C}$ for $C\in\ob{\cat{C}}$.

We extend these definitions to other mathematical structures.
For example, a group is said to be \defemph{in $\univ{U}$} if it
(i.e.,  the tuple consisting of its data) is an element
of $\univ{U}$, 
a 2-category is \defemph{locally in $\univ{U}$}  
if all its hom-categories are in $\univ{U}$, and so on.

\medskip

Recall the universes $\univ{U}_1$, $\univ{U}_2$ and $\univ{U}_3$
we have fixed above.
\begin{definition}
\label{conv:size}
A set or other mathematical structure (group, category, etc.) 
is said to be:
\begin{itemize}
\item \defemph{small} if it is in $\univ{U}_1$;
\item \defemph{large} if it is in $\univ{U}_2$;
\item \defemph{huge} if it is in $\univ{U}_3$.
\end{itemize}
Sets and other mathematical structures are often assumed to be small by 
default, even when we do not say so explicitly.

A category (or a 2-category) is said to be:
\begin{itemize}
\item \defemph{locally small} if it is large and locally in $\univ{U}_1$;
\item \defemph{locally large} if it is huge and locally in $\univ{U}_2$.
\end{itemize}
\end{definition}
In the following, we mainly talk about the 
size-regulating conditions using the terms \emph{small}, \emph{large} 
and \emph{huge},
avoiding direct references to the universes $\univ{U}_1$, $\univ{U}_2$
and $\univ{U}_3$.

We shall use the following basic (2-)categories throughout this paper.
\begin{itemize}
\item $\Set$, the (large) category of all small sets and functions.
\item $\SET$, the (huge) category of all large sets and functions.
\item $\Cat$, the (large) category of all small categories and functors.
\item $\CAT$, the (huge) category of all large categories and functors.
\item $\tCat$, the (large) 2-category of all small categories, functors and 
natural 
transformations.
\item $\tCAT$, the (huge) 2-category of all large categories, functors and 
natural 
transformations.
\item $\MonCAT$, the (huge) 2-category of all large monoidal categories, lax 
monoidal functors\footnote{Also called \emph{monoidal functors} in e.g., 
\cite{MacLane_CWM}.} and monoidal natural transformations.
We also use several variants of it.
\item $\twoCAT$, the 2-category of all huge 2-categories,
2-functors and 2-natural transformations.
\end{itemize}

\section{Notions of algebraic theory}
\label{sec:notions_of_alg_thy}

In this section, we review several known 
notions of algebraic theory.
As we shall see later, they all turn out to be instances of our unified 
framework for notions of algebraic theory developed from 
the next section on.
A more introductory account of these notions of algebraic theory
(except for PROPs and PROs) may be found in \cite[Chapter~2]{Fujii_thesis}.

\subsection{Clones}
\label{sec:clone}
\emph{Abstract clones} (\emph{clones} for short) \cite{Taylor_clone} are a 
presentation independent 
version of equational theories in universal algebra and, as such,
they are more or less equivalent to \emph{Lawvere 
theories}~\cite{Lawvere_thesis}.
Let us begin with the definition.

\begin{definition}
\label{def:clone}
A \defemph{clone} $\monoid{T}$ consists of:
\begin{description}
\item[(CD1)] a family of sets $T=(T_n)_{n\in\NN}$ indexed by natural numbers;
\item[(CD2)] for each $n\in\NN$ and $i\in\{1,\dots,n\}$, an element
\[
p^{(n)}_i\in T_n;
\]
\item[(CD3)] for each $k,n\in\NN$, a function
\[
\circ^{(n)}_k\colon T_k\times (T_n)^k\longrightarrow T_n
\]
whose action on an element $(\phi,\theta_1,\dots,\theta_k)\in T_k\times (T_n)^k$
we write as $\phi\circ^{(n)}_k(\theta_1,\dots,\theta_k)$
or simply $\phi\circ(\theta_1,\dots,\theta_k)$;
\end{description}
satisfying the following equations:
\begin{description}
\item[(CA1)] for each $k,n\in\NN$, $j\in\{1,\dots,k\}$, 
$\theta_1,\dots,\theta_k\in T_n$, 
\[
p^{(k)}_j\circ^{(n)}_k(\theta_1,\dots,\theta_k) = \theta_j;
\]
\item[(CA2)] for each $n\in\NN$, $\theta\in T_n$,
\[
\theta\circ^{(n)}_n (p^{(n)}_1,\dots,p^{(n)}_n) = \theta;
\]
\item[(CA3)] for each $l,k,n\in\NN$, $\psi\in T_l$, $\phi_1,\dots,\phi_l\in
T_k$, $\theta_1,\dots,\theta_k\in T_n$, 
\begin{multline*}
\psi\circ^{(k)}_l\big(\phi_1\circ^{(n)}_k(\theta_1,\dots,\theta_k),\ 
\dots,\ \phi_l\circ^{(n)}_k(\theta_1,\dots,\theta_k)\big)
\\=
\big(\psi\circ^{(k)}_l(\phi_1,\dots,\phi_l)\big)\circ^{(n)}_k 
(\theta_1,\dots,\theta_k).
\end{multline*}
\end{description}
Such a clone is written as 
$\monoid{T}=(T,(p^{(i)}_n)_{n\in\NN,i\in\{1,\dots,n\}},
(\circ^{(n)}_k)_{k,n\in\NN})$
or simply $(T,p,\circ)$.
\end{definition}

The following example shows a typical way in which clones arise.

\begin{example}
\label{ex:endoclone}
Let $\cat{C}$ be a locally small category with all finite products, 
and $C$ be an object of $\cat{C}$.
Then we obtain the clone $\End_\cat{C}(C)=(\enrich{C}{C},p,\circ)$ defined as 
follows:
\begin{description}
\item[(CD1)] for each $n\in\NN$, let $\enrich{C}{C}_n$ be the hom-set 
$\cat{C}(C^n,C)$, where $C^n$ is the product of $n$ copies of $C$ (the 
$n$-th power of $C$);
\item[(CD2)] for each $n\in\NN$ and $i\in \{1,\dots,n\}$, let 
$p^{(n)}_i\colon C^n\longrightarrow C$ be the $i$-th projection;
\item[(CD3)] for each $k,n\in\NN$, $g\colon C^k\longrightarrow C$
and $f_1,\dots,f_k\colon C^n\longrightarrow C$, let 
$g\circ^{(n)}_k (f_1,\dots,f_k)\colon$ $C^n\longrightarrow C$ be the following 
composite in $\cat{C}$:
\[
\begin{tikzpicture}[baseline=-\the\dimexpr\fontdimen22\textfont2\relax ]
      \node (1) at (0,0)  {$C^n$};
      \node (2) at (3,0)  {$C^k$};
      \node (3) at (5,0) {$C.$};
      \draw[->] (1) to node[auto,labelsize]{$\langle f_1,\dots,f_k\rangle$} (2);
      \draw[->] (2) to node[auto,labelsize] {$g$} (3);
\end{tikzpicture} 
\]
\end{description}
It is straightforward to check the axioms (CA1)--(CA3).
\end{example}

In fact, \emph{every} clone arises in the above manner.
To see this, let $\monoid{T}=(T,p,\circ)$ be an arbitrary clone.
We can then define the category $\cat{C}_\monoid{T}$,
whose set of objects is the set $\NN$ of natural numbers and 
whose hom-sets are given by $\cat{C}_\monoid{T}(n,m) = (T_n)^m$.
We may routinely define identity morphisms and composition in 
$\cat{C}_\monoid{T}$ from the structure of $\monoid{T}$,
and it turns out that the object $n\in\cat{C}_\monoid{T}$ 
is the $n$-th power of $1\in \cat{C}_\monoid{T}$ in $\cat{C}_\monoid{T}$;
thus we recover $\monoid{T}$ as $\End_{\cat{C}_\monoid{T}}(1)$.
This construction also shows how clones and Lawvere theories are 
related.
See \cite{Taylor_clone} for details.

Another source of clones is provided by 
\emph{presentations of equational theories}
in universal algebra.
A presentation of an equational theory $\pres{\Sigma}{E}$ 
is given by a family $\Sigma$ of basic operations (with designated arities)
and a family $E$ of equational axioms between $\Sigma$-terms.
Given such a data, we can define a clone 
$\monoid{T}^\pres{\Sigma}{E}=(T^\pres{\Sigma}{E},p,\circ)$
by setting $(T^\pres{\Sigma}{E})_n$ to be the set of all $\Sigma$-terms 
in which at most $n$ different variables occur, modulo \emph{equational 
theorems} derivable from $E$, $p$ to be the equivalence classes
of variables, and $\circ$ to be the
simultaneous substitution; see \cite[Sections~2.1 and 2.2]{Fujii_thesis} for 
details.
It is well-known that various types of algebras, such as groups, monoids and 
rings, admit presentations of equational theories, and hence by the above 
construction we obtain the \emph{clone of groups}, etc. 
Such are the motivating examples of clones  \textit{qua} algebraic theories.

\medskip

Homomorphisms of clones may be defined routinely.
\begin{definition}
\label{def:clone_hom}
Let $\monoid{T}=(T,p,\circ)$ and $\monoid{T'}=(T',p',\circ')$ be 
clones.
A \defemph{clone homomorphism} from $\monoid{T}$ to $\monoid{T'}$
is a family of functions
$h=(h_n\colon T_n\longrightarrow T'_n)_{n\in\NN}$ which preserves the structure 
of clones;
precisely,
\begin{itemize}
\item for each $n\in\NN$ and $i\in\{1,\dots,n\}$, $h_n(p^{(n)}_i)=p'^{(n)}_i$;
\item for each $k,n\in\NN$, $\phi\in T_k$ and $\theta_1,\dots,\theta_k\in T_n$,
\[
h_n\big(\phi\circ^{(n)}_k(\theta_1,\dots,\theta_k)\big)=
h_k(\phi)\circ'^{(n)}_k \big(h_n(\theta_1),\dots,h_n(\theta_k)\big).
\]
\end{itemize}
\end{definition}

Let us turn to the definition of models of a clone.
One can consider models of a clone in any locally small category with finite 
products.
\begin{definition}
\label{def:clone_model}
Let $\monoid{T}$ be a clone and $\cat{C}$ be a locally small category with all
finite products.
A \defemph{model of $\monoid{T}$ in $\cat{C}$} consists of an object 
$C\in\cat{C}$ together with 
a clone homomorphism $\chi\colon \monoid{T}\longrightarrow \End_\cat{C}(C)$.
\end{definition}
For example, models of the clone of groups in $\cat{C}$ are precisely
\emph{group objects} (or \emph{internal groups}) in $\cat{C}$. 

We then define the notion of homomorphism between models (of a clone).
First we note the functoriality of the $\enrich{-}{-}$ construction, 
already appeared (partly) in Example~\ref{ex:endoclone}.
For the current purpose, it suffices to remark that for each locally small 
category $\cat{C}$ with finite products and each $n\in\NN$,
the assignment $\enrich{A}{B}_n=\cat{C}(A^n,B)$ canonically extends to a 
functor 
\[
\enrich{-}{-}_n\colon \cat{C}^\op\times\cat{C}\longrightarrow \Set.
\]

\begin{definition}
\label{def:clone_mod_hom}
Let $\monoid{T}$ be a clone, $\cat{C}$ be a locally small category with all 
finite products, and $(A,\alpha)$ and $(B,\beta)$ be models of 
$\monoid{T}$ in $\cat{C}$.
A \defemph{homomorphism} from $(A,\alpha)$ to $(B,\beta)$ is a morphism 
$f\colon A\longrightarrow B$ in $\cat{C}$ such that for each $n\in\NN$,
the following diagram commutes:
\[
\begin{tikzpicture}[baseline=-\the\dimexpr\fontdimen22\textfont2\relax ]
      \node (TL) at (0,2)  {$T_n$};
      \node (TR) at (3,2)  {$\enrich{A}{A}_n$};
      \node (BL) at (0,0) {$\enrich{B}{B}_n$};
      \node (BR) at (3,0) {$\enrich{A}{B}_n.$};
      \draw[->] (TL) to node[auto,labelsize](T) {$\alpha_n$} (TR);
      \draw[->]  (TR) to node[auto,labelsize] {$\enrich{A}{f}_n$} (BR);
      \draw[->]  (TL) to node[auto,swap,labelsize] {$\beta_n$} (BL);
      \draw[->] (BL) to node[auto,swap,labelsize](B) {$\enrich{f}{B}_n$} (BR);
\end{tikzpicture} 
\]
\end{definition}
For each clone $\monoid{T}$ and a locally small 
category $\cat{C}$ with all finite products, we denote the category of all 
models of $\monoid{T}$ in $\cat{C}$ and their homomorphisms by 
$\Mod{\monoid{T}}{\cat{C}}$.
Note that we have the canonical forgetful functor 
$U\colon\Mod{\monoid{T}}{\cat{C}}\longrightarrow\cat{C}$.

\subsection{Non-symmetric operads}
Non-symmetric operads~\cite{May_loop} may be seen as a variant
of clones.
Compared to clones, non-symmetric operads are less expressive (for example,
groups cannot be captured by non-symmetric operads), but 
their models can be taken in wider contexts than for clones (whereas models of 
clones are taken in a locally small category with finite products,
models of non-symmetric operads can be taken in any locally small monoidal 
category).

\begin{definition}
\label{def:non_symm_op}
A \defemph{non-symmetric operad} $\monoid{T}$ consists of:
\begin{description}
\item[(ND1)] a family of sets $T=(T_n)_{n\in\NN}$ indexed by natural numbers;
\item[(ND2)] an element $\id{}\in T_1$;
\item[(ND3)] for each $k,n_1,\dots,n_k\in \NN$, a function
(we omit the sub- and superscripts)
\[
\circ\colon T_k\times T_{n_1}\times\dots\times T_{n_k}\longrightarrow 
T_{n_1+\dots+ n_k}
\]
whose action we write as $(\phi,\theta_1,\dots,\theta_k)\longmapsto \phi\circ 
(\theta_1,\dots,\theta_k)$
\end{description}
satisfying the following equations:
\begin{description}
\item[(NA1)] for each $n\in\NN$ and $\theta\in T_n$,
\[
\id{}\circ (\theta)=\theta;
\]
\item[(NA2)] for each $n\in\NN$ and $\theta\in T_n$,
\[
\theta\circ (\id{},\dots,\id{}) = \theta;
\] 
\item[(NA3)] for each $l,\ k_1,\dots,k_l,\ 
n_{1,1},\dots,n_{1,k_1},\dots,n_{l,1},
\dots,n_{l,k_l}\in\NN$, \ $\psi\in T_l$, \ $\phi_1\in T_{k_1}, \dots, \phi_l\in 
T_{k_l}$, \ $\theta_{1,1}\in T_{n_{1,1}},\dots,\theta_{1,k_1}\in T_{n_{1,k_1}},
\dots, \theta_{l,1}\in T_{n_{l,1}},\dots,$ $\theta_{l,k_l}\in T_{n_{l, k_l}}$,
\begin{multline*}
\psi\circ\big(\phi_1\circ (\theta_{1,1},\dots,\theta_{1,k_1}),\ \dots\ ,
\phi_l\circ (\theta_{l,1},\dots,\theta_{l,k_l}) \big)\\
=
\big(\psi\circ 
(\phi_1,\dots,\phi_l)\big)\circ(\theta_{1,1},\dots,\theta_{1,k_1},\ \dots,\ 
\theta_{l,1},\dots,\theta_{l,k_l}).
\end{multline*}
\end{description}
Such a non-symmetric operad is written as $\monoid{T}=(T,\id{},\circ)$.
\end{definition}

\begin{example}
\label{ex:endo_ns_operad}
Let $\cat{C}=(\cat{C},I,\otimes)$ be a locally small monoidal category, 
and $C$ be an object of $\cat{C}$.
Define the non-symmetric operad $\End_\cat{C}(C)=(\enrich{C}{C},\id{},\circ)$ 
as follows:
\begin{description}
\item[(ND1)] for each $n\in\NN$, let $\enrich{C}{C}_n$ be the hom-set 
$\cat{C}(C^{\otimes n},C)$, where $C^{\otimes n}$ is the monoidal product of 
$n$ copies of $C$;
\item[(ND2)] let 
$\id{}\colon C\longrightarrow C$ be the identity morphism on $\cat{C}$;
\item[(ND3)] for each $k,n_1,\dots,n_k\in\NN$, $g\colon C^{\otimes 
k}\longrightarrow C$
and $f_1\colon C^{\otimes n_1}\longrightarrow C,\dots,f_k\colon C^{\otimes 
n_k}\longrightarrow C$, let 
$g\circ (f_1,\dots,f_k)\colon$ $C^{\otimes (n_1+\dots+n_k)}\longrightarrow C$ 
be 
the following 
composite in $\cat{C}$:
\[
\begin{tikzpicture}[baseline=-\the\dimexpr\fontdimen22\textfont2\relax ]
      \node (1) at (-3,0)  {$C^{\otimes (n_1+\dots+n_k)}\cong C^{\otimes 
      n_1}\otimes \dots \otimes C^{\otimes n_k}$};
      \node (2) at (3,0)  {$C^{\otimes k}$};
      \node (3) at (5,0) {$C.$};
      \draw[->] (1) to node[auto,labelsize]{$f_1\otimes \dots\otimes f_k$} (2);
      \draw[->] (2) to node[auto,labelsize] {$g$} (3);
\end{tikzpicture} 
\]
\end{description}
It is straightforward to check the axioms (NA1)--(NA3).
\end{example}

Just like the case of clones, it can be shown that {every} non-symmetric operad 
arises as above.
Given an arbitrary non-symmetric operad $\monoid{T}=(T,\id{},\circ)$,
we define a strict monoidal category $\cat{C}_\monoid{T}$
whose set of objects is $\NN$ and whose hom-sets are defined as 
\[
\cat{C}_\monoid{T}(n,m)=
\coprod_{\substack{n_1,\dots,n_m\in\NN\\n_1+\dots+n_m=n}}
T_{n_1}\times\dots\times T_{n_m},
\]
with sum of natural numbers as the monoidal product;
cf.~\cite[Section~2.3]{Leinster_book}.
We then obtain $\monoid{T}$ as $\End_{\cat{C}_\monoid{T}}(1)$.

As for examples of non-symmetric operads naturally seen as algebraic theories,
we may obtain a non-symmetric operad 
from a {strongly regular} presentation of an equational theory.
Here, a presentation of an equational theory is called \defemph{strongly 
regular} if
each of its equational axioms satisfies the condition that, on each side of the 
equation, exactly same variables appear without repetition and in the same 
order.
For example, assuming that $e$ is a nullary operation and $\cdot$ is a binary 
operation, the equations
\[
(x\cdot y)\cdot z = x\cdot (y\cdot z), \quad x \cdot e = x 
\]
are strongly regular, whereas the equations
\[
x\cdot x = x,\quad x\cdot y = x,\quad x\cdot y = y\cdot x
\]
are not.
The structure of monoids can be expressed by a non-symmetric operad.

\medskip

We may define the notion of \defemph{non-symmetric operad homomorphism} 
between non-symmetric operads 
just in the same way as that of clone homomorphism 
(Definition~\ref{def:clone_hom}).

The definition of models of a non-symmetric operad is also similar to that of 
a clone (Definition~\ref{def:clone_model}).
\begin{definition}
\label{def:ns_operad_model}
Let $\monoid{T}$ be a non-symmetric operad and $\cat{C}$ be a locally small 
monoidal category.
A \defemph{model of $\monoid{T}$ in $\cat{C}$} consists of an object 
$C\in\cat{C}$ together with a non-symmetric operad homomorphism
$\chi\colon \monoid{T}\longrightarrow \End_\cat{C}(C)$.
\end{definition}

With the definition of homomorphisms between models
completely parallel to the case of clones (Definition~\ref{def:clone_mod_hom}),
we obtain the category of models $\Mod{\monoid{T}}{\cat{C}}$
(as well as the associated forgetful functor)
for each non-symmetric operad $\monoid{T}$ and a locally small monoidal
category $\cat{C}$.

\subsection{Symmetric operads}
\label{sec:sym_operad}
Symmetric operads~\cite{May_loop} are an intermediate notion of algebraic theory
which lie between clones and 
non-symmetric operads, both in terms of expressive power and 
in terms of range of notions of models.

In terms of expressive power, symmetric operads correspond to 
regular presentations of equational theories,
where a presentation of an equational theory is called \defemph{regular}
if each of its equational axioms satisfies the following condition:
on each side of the equation, exactly same variables appear without repetition
(but possibly in different order).
So we may express \emph{commutative} monoids in addition to monoids
via symmetric operads.

In terms of range of notions of models,
one can take models of a symmetric operad in any locally small symmetric 
monoidal category.

We omit the definitions of symmetric operads, their models and so on, 
which are analogous to the cases of clones and non-symmetric operads.
See e.g.~\cite[Section~2.4]{Fujii_thesis} for details.

\subsection{PROs}
PROs~\cite{MacLane_cat_alg} are a certain generalisation of
non-symmetric operads.

\begin{definition}
\label{def:PRO}
A \defemph{PRO} is a small strict monoidal category
whose set of objects is $\NN$ and whose monoidal product
on objects is the sum of natural numbers.
\end{definition}
Homomorphisms of PROs are identity-on-objects strict monoidal functors. 

The relationship of PROs to non-symmetric operads is as follows.
Given a non-symmetric operad, one can define a PRO
by the construction sketched just after Example~\ref{ex:endo_ns_operad}.
PROs are more general than non-symmetric operads in that algebraic structures 
described by them allow many-in, many-out operations
(operations of type $C^{\otimes n}\longrightarrow C^{\otimes m}$), whereas 
those described by non-symmetric operads
allow only many-in, single-out operations (those of type $C^{\otimes 
n}\longrightarrow C$).
Therefore via PROs, one can capture the structure of 
\emph{comonoids} in addition to monoids.

\begin{example}
\label{ex:endo_PRO}
Let $\cat{C}=(\cat{C},I,\otimes)$ be a locally small monoidal category, 
and $C$ be an object of $\cat{C}$.
Then we obtain the PRO $\End_\cat{C}(C)$ defined as follows:
for each $n,m\in\NN$, let the hom-set $\End_\cat{C}(C)(n,m)$ of 
the PRO be the hom-set 
$\cat{C}(C^{\otimes n},C^{\otimes m})$ of $\cat{C}$, with the rest of the 
structure
induced from $\cat{C}$ in the evident manner.
\end{example}

Models of a PRO $\monoid{T}$
can be taken in any locally small monoidal category $\cat{C}$,
and are sometimes defined as strong monoidal functors of type 
$\monoid{T}\longrightarrow\cat{C}$,
following the doctrine of \emph{functorial semantics}~\cite{Lawvere_thesis}.
We shall adopt the following alternative, more or less equivalent definition
(appearing e.g., in \cite{MacLane_cat_alg}).

\begin{definition}
\label{def:PRO_model}
Let $\monoid{T}$ be a PRO and $\cat{C}$ be a locally small monoidal category.
A \defemph{model of $\monoid{T}$ in $\cat{C}$} consists of 
an object $C\in\cat{C}$ together with a PRO homomorphism
$\chi\colon \monoid{T}\longrightarrow\End_\cat{C}(C)$.
\end{definition}

With the straightforward definition of homomorphisms of models
(just like the case of clones), we obtain 
the category of models $\Mod{\monoid{T}}{\cat{C}}$ and 
the associated forgetful functor.

\medskip

The above definition of models not only looks similar to the corresponding 
definitions for clones, symmetric operads and non-symmetric operads,
but it also eliminates certain redundancy involved in
the definition of models as certain functors.
That is, if we define a model to be a strong monoidal functor 
$\monoid{T}\longrightarrow\cat{C}$,
to give such a data
we have to specify the image of each object in $\monoid{T}$,
and this involves a choice, for each natural number $n$, 
of a specific $n$-fold monoidal product of the 
underlying object; but
such choices are usually not included 
in the data of an algebraic structure.
In more precise terms,
such a definition of models 
(in the spirit of functorial semantics) in general prevents 
the forgetful functor out of the category of models from being 
\emph{amnestic}~\cite[Definition~3.27]{AHS:joy-of-cats}.
On the other hand, it will turn out that 
all forgetful functors arising in our unified framework for notions of 
algebraic theory are amnestic.

\subsection{PROPs}
PROPs \cite{MacLane_cat_alg} are the {symmetric} version of PROs.

\begin{definition}
\label{def:PROP}
A \defemph{PROP} is a small symmetric strict monoidal category whose set of 
objects is $\NN$ and whose monoidal product on objects is the sum 
of natural numbers.
\end{definition}
Here, by a \defemph{symmetric strict monoidal category} we mean 
a symmetric monoidal category
whose structure isomorphisms \emph{except for the symmetry} are 
identities.
Similarly to the case of PROs, we can
take models of a PROP in any locally small symmetric monoidal category.

\subsection{Monads}
The definition of monads on a category is well-known.
For any large category $\cat{C}$, we shall denote the category of
all monads on $\cat{C}$ and monad morphisms by $\Mnd{\cat{C}}$,
which we define to be the category of monoid objects 
in the monoidal category $[\cat{C},\cat{C}]$ of all endofunctors on $\cat{C}$
(with composition of endofunctors as monoidal product);
note that our definition of monad morphism is {opposite}
in the direction to the one defined in \cite{Street_FTM}.

The following definition of
models of a monad, usually called \emph{Eilenberg--Moore 
algebras}, is also well-known.

\begin{definition}[\cite{Eilenberg_Moore}]
\label{def:EM_alg}
Let $\cat{C}$ be a large category and $\monoid{T}=(T,\eta,\mu)$ be a monad on 
$\cat{C}$.
\begin{enumerate}
\item 
An \defemph{Eilenberg--Moore algebra of $\monoid{T}$}
consists of:
\begin{itemize}
\item an object $C\in\cat{C}$;
\item a morphism $\gamma\colon TC\longrightarrow C$ in $\cat{C}$,
\end{itemize}
making the following diagrams commute:
\begin{equation*}
\begin{tikzpicture}[baseline=-\the\dimexpr\fontdimen22\textfont2\relax ]
      \node (TL) at (0,2)  {$C$};
      \node (TR) at (2,2)  {$TC$};
      \node (BR) at (2,0)  {$C$};
      \draw[->] (TL) to node[auto,labelsize](T) {$\eta_C$} (TR);
      \draw[->] (TR) to node[auto,labelsize](R) {$\gamma$} (BR);
      \draw[->]  (TL) to node[auto,labelsize,swap] (L) {$\id{C}$} (BR);
\end{tikzpicture} 
\qquad
\begin{tikzpicture}[baseline=-\the\dimexpr\fontdimen22\textfont2\relax ]
      \node (TL) at (0,2)  {$TTC$};
      \node (TR) at (2,2)  {$TC$};
      \node (BL) at (0,0) {$TC$};
      \node (BR) at (2,0) {$C.$};
      \draw[->] (TL) to node[auto,labelsize](T) {$\mu_C$} (TR);
      \draw[->]  (TR) to node[auto,labelsize] {$\gamma$} (BR);
      \draw[->]  (TL) to node[auto,swap,labelsize] {$T\gamma$} (BL);
      \draw[->] (BL) to node[auto,labelsize](B) {$\gamma$} (BR);
\end{tikzpicture} 
\end{equation*}
\item Let $(C,\gamma)$ and $(C',\gamma')$ be Eilenberg--Moore algebras of 
$\monoid{T}$.
A \defemph{homomorphism} from $(C,\gamma)$ to $(C',\gamma')$ 
is a morphism $f\colon C\longrightarrow C'$ in $\cat{C}$
making the following diagram commute:
\[
\begin{tikzpicture}[baseline=-\the\dimexpr\fontdimen22\textfont2\relax ]
      \node (TL) at (0,2)  {$TC$};
      \node (TR) at (2,2)  {$TC'$};
      \node (BL) at (0,0) {$C$};
      \node (BR) at (2,0) {$C'.$};
      \draw[->] (TL) to node[auto,labelsize](T) {$Tf$} (TR);
      \draw[->]  (TR) to node[auto,labelsize] {$\gamma'$} (BR);
      \draw[->]  (TL) to node[auto,swap,labelsize] {$\gamma$} (BL);
      \draw[->] (BL) to node[auto,labelsize](B) {$f$} (BR);
\end{tikzpicture}
\]
\end{enumerate}
The category of all Eilenberg--Moore algebras of $\monoid{T}$ and their
homomorphisms is called the \defemph{Eilenberg--Moore category of $\monoid{T}$},
and is denoted by $\cat{C}^\monoid{T}$.
\end{definition}
Excellent introductions to monads abound (see e.g., 
\cite[Chapter~VI]{MacLane_CWM}). 
Here we simply remark that, seen as a notion of algebraic theory, monads on 
$\Set$ have stronger expressive power than 
clones, and can express such infinitary structures as 
compact Hausdorff spaces and complete semilattices~\cite{Manes}. 

\medskip

Let us make clear our convention as to which class of monads counts as a 
single notion of algebraic theory.
We understand that for each large category $\cat{C}$,
the monads on $\cat{C}$ form a single notion of algebraic theory.
So actually we have just introduced 
an infinite family of notions of algebraic theory, one notion 
for each large category. 

\subsection{Generalised operads}
\label{sec:S-operad}
Just like monads are a family of notions of algebraic theory 
parameterised by a large category, the term
\emph{generalised operads} 
\cite{Burroni_T_cats,Kelly_club_data_type,Hermida_representable,Leinster_book}
also refer to a family of notions of algebraic 
theory, this time parameterised by a large category with finite limits and a 
\emph{cartesian monad} thereon.
We start with the definition of cartesian monad.

\begin{definition}
\begin{enumerate}
\item Let $\cat{C}$ and $\cat{D}$ be categories, and $F,G\colon 
\cat{C}\longrightarrow \cat{D}$ be functors.
A natural transformation $\alpha\colon F\longrightarrow G$
is called \defemph{cartesian} if and only if all naturality squares of $\alpha$
are pullback squares;
that is, if and only if 
for any morphism $f\colon C\longrightarrow C'$ in $\cat{C}$, 
the square 
\[
\begin{tikzpicture}[baseline=-\the\dimexpr\fontdimen22\textfont2\relax ]
      \node (TL) at (0,2)  {$FC$};
      \node (TR) at (2,2)  {$GC$};
      \node (BL) at (0,0) {$FC'$};
      \node (BR) at (2,0) {$GC'$};
      \draw[->] (TL) to node[auto,labelsize](T) {$\alpha_C$} (TR);
      \draw[->]  (TR) to node[auto,labelsize] {$Gf$} (BR);
      \draw[->]  (TL) to node[auto,swap,labelsize] {$Ff$} (BL);
      \draw[->] (BL) to node[auto,labelsize](B) {$\alpha_{C'}$} (BR);
\end{tikzpicture}
\]
is a pullback of $Gf$ and $\alpha_{C'}$.
\item Let $\cat{C}$ be a category with all pullbacks.
A monad $\monoid{S}=(S,\eta,\mu)$ on $\cat{C}$ is called 
\defemph{cartesian} if and only if the functor $S$ preserves pullbacks,
and $\eta$ and $\mu$ are cartesian.
\end{enumerate}
\end{definition}

For each cartesian monad $\monoid{S}=(S,\eta,\mu)$ on a large category $\cat{C}$
with all finite limits
we now introduce \emph{$\monoid{S}$-operads}, which 
form a single notion of algebraic theory.

The crucial observation is that under this assumption,
the slice category $\cat{C}/S1$ (where $1$ is the terminal object of $\cat{C}$) 
acquires a canonical monoidal structure, given as follows.
(We write an object of $\cat{C}/S1$ either as $p\colon P\longrightarrow S1$
or $(P,p)$.)
\begin{itemize}
\item The unit object is $I=(\eta_1\colon 1\longrightarrow S1)$.
\item Given a pair of objects $p\colon P\longrightarrow S1$ and $q\colon 
Q\longrightarrow S1$ in $\cat{C}/S1$,
first form the pullback
\begin{equation}
\label{eqn:pb_over_S1}
\begin{tikzpicture}[baseline=-\the\dimexpr\fontdimen22\textfont2\relax ]
      \node(0) at (0,1) {$(Q,q)\ast P$};
      \node(1) at (2,1) {$SP$};
      \node(2) at (0,-1) {$Q$};
      \node(3) at (2,-1) {$S1,$};
      
      \draw [->] 
            (0) to node (t)[auto,labelsize] {$\pi_2$} 
            (1);
      \draw [->] 
            (1) to node (r)[auto,labelsize] {$S!$} 
            (3);
      \draw [->] 
            (0) to node (l)[auto,swap,labelsize] {$\pi_1$} 
            (2);
      \draw [->] 
            (2) to node (b)[auto,swap,labelsize] {$q$} 
            (3);  

     \draw (0.2,0.4) -- (0.6,0.4) -- (0.6,0.8);
\end{tikzpicture}
\end{equation}
where $!\colon P\longrightarrow 1$ is the unique morphism to the terminal 
object.
The monoidal product $(Q,q)\otimes(P,p)\in\cat{C}/S1$
is $((Q,q)\ast P, \mu_1\circ Sp\circ \pi_2)$:
\[
\begin{tikzpicture}[baseline=-\the\dimexpr\fontdimen22\textfont2\relax ]
      \node(0) at (-0.5,0) {$(Q,q)\ast P$};
      \node(1) at (2,0) {$SP$};
      \node(2) at (4,0) {$SS1$};
      \node(3) at (6,0) {$S1.$};
      
      \draw [->] 
            (0) to node (t)[auto,labelsize] {$\pi_2$} 
            (1);
      \draw [->] 
            (1) to node (r)[auto,labelsize] {$Sp$} 
            (2);
      \draw [->] 
            (2) to node (l)[auto,labelsize] {$\mu_1$} 
            (3); 
\end{tikzpicture}
\]
\end{itemize}
We remark that this monoidal category arises as a restriction of Burroni's 
bicategory of $\monoid{S}$-spans~\cite{Burroni_T_cats}.

\begin{definition}
\label{def:S_operad}
Let $\cat{C}$ be a large category with all finite limits 
and $\monoid{S}=(S,\eta,\mu)$ a cartesian monad on $\cat{C}$.
\begin{enumerate}
\item An \defemph{$\monoid{S}$-operad} is a monoid object in the monoidal 
category
$(\cat{C}/S1,I,\otimes)$ introduced above.
\item A \defemph{morphism of $\monoid{S}$-operads} is a homomorphism 
of monoid objects in $(\cat{C}/S1,I,\otimes)$.
\end{enumerate}
We denote the category of $\monoid{S}$-operads and their homomorphisms by 
$\Operad{\monoid{S}}$;
by definition it is identical to the category of monoid
objects in $\cat{C}/S1$.
\end{definition}

Given an $\monoid{S}$-operad $\monoid{T}=((\arity{T}\colon 
T\longrightarrow S1),e,m)$, we may interpret
the object $S1$ in $\cat{C}$
as the object of arities, $T$ as the object of all
(derived) operations of the algebraic theory expressed by $\monoid{T}$, 
and $\ar{T}$ as assigning an arity to each operation.

\begin{example}[{\cite[Example~4.2.7]{Leinster_book}}]
\label{ex:non_symm_op_as_gen_op}
If we let $\cat{C}=\Set$ and $\monoid{S}$ be the free monoid monad
(which is cartesian), 
then $\monoid{S}$-operads are equivalent to {non-symmetric operads}.
The arities are the natural numbers: $S1\cong\mathbb{N}$.

In more detail, the data of an $\monoid{S}$-operad in this case
consist of a set $T$, and functions 
$\arity{T}\colon T\longrightarrow \N$, $e\colon 1\longrightarrow T$
and $m\colon (T,\arity{T})\ast T\longrightarrow T$. 
Unravelling this, we obtain a family of sets $(T_n)_{n\in\N}$, 
an element $\id{}\in T_1$ and 
a family of functions $(m_{k,n_1,\dots n_k}\colon T_k\times T_{n_1}\times\dots
\times T_{n_k}
\longrightarrow 
T_{n_1+\dots +n_k})_{k,n_1,\dots,n_k\in\N}$,
agreeing with Definition~\ref{def:non_symm_op}.
Note that indeed $T_n$ may be interpreted as the set of all (derived) 
operations of arity $n$.
\end{example}

\begin{example}
If we set $\cat{C}=\enGph{n}$, the category of $n$-graphs for 
$n\in\NN\cup\{\omega\}$ 
and $\monoid{S}$ be the free strict $n$-category monad, 
then 
$\monoid{S}$-operads are called \emph{$n$-globular operads};
see \cite[Chapter~8]{Leinster_book} for details.
These generalised operads have been used to give a definition of 
weak $n$-categories \cite{Batanin_98,Leinster_book}.
\end{example}

Next we define models of an $\monoid{S}$-operad.
For this, we first note that the monoidal category $\cat{C}/S1$ has a canonical
pseudo action on $\cat{C}$.
(The precise definition of pseudo action is a variant of 
Definition~\ref{def:oplax_action}, obtained by replacing the term ``natural 
transformation'' there by ``natural isomorphism''.)
The functor 
\[
\ast\colon(\cat{C}/S1)\times\cat{C}\longrightarrow\cat{C}
\]
defining this pseudo action is given by mapping $((Q,q),P)\in 
(\cat{C}/S1)\times\cat{C}$ to
$(Q,q)\ast P\in\cat{C}$ defined as the pullback (\ref{eqn:pb_over_S1}).
\begin{definition}
\label{def:S-operad_model}
Let $\cat{C}$ be a large category with finite limits, $\monoid{S}=(S,\eta,\mu)$
be a cartesian monad on $\cat{C}$, and $\monoid{T}=(T,e,m)$ be an 
$\monoid{S}$-operad.
\begin{enumerate}
\item A \defemph{model of $\monoid{T}$} consists of:
\begin{itemize}
\item an object $C\in\cat{C}$;
\item a morphism $\gamma\colon T\ast C\longrightarrow C$ in $\cat{C}$,
\end{itemize}
making the following diagrams commute:
\begin{equation*}
\begin{tikzpicture}[baseline=-\the\dimexpr\fontdimen22\textfont2\relax ]
      \node (TL) at (0,2)  {$I\ast C$};
      \node (TR) at (2,2)  {$T\ast C$};
      \node (BR) at (2,0) {${C}$};
      \draw[->]  (TL) to node[auto,labelsize] {$e\ast C$} (TR);
      \draw[->]  (TR) to node[auto,labelsize] {$\gamma$} (BR);
      \draw[->] (TL) to node[auto,swap,labelsize](B) {$\cong$} (BR);
\end{tikzpicture} 
\qquad
\begin{tikzpicture}[baseline=-\the\dimexpr\fontdimen22\textfont2\relax ]
      \node (TL) at (0,2)  {$(T\otimes T)\ast C$};
      \node (TR) at (5,2)  {$T\ast C$};
      \node (BL) at (0,0) {$T\ast (T\ast C)$};
      \node (BM) at (3,0) {$T\ast C$};
      \node (BR) at (5,0) {${C},$};
      \draw[->]  (TL) to node[auto,swap,labelsize] {$\cong$} (BL);
      \draw[->]  (TL) to node[auto,labelsize] {$m\ast C$} (TR);
      \draw[->]  (BL) to node[auto,labelsize] {$T\ast \gamma$} (BM);
      \draw[->]  (BM) to node[auto,labelsize] {$\gamma$} (BR);
      \draw[->]  (TR) to node[auto,labelsize](B) {$\gamma$} (BR);
\end{tikzpicture} 
\end{equation*}
where the arrows labelled with $\cong$ refer to the isomorphisms 
provided by the pseudo action.
\item Let $(C,\gamma)$ and $(C',\gamma')$ be models of $\monoid{T}$. A 
\defemph{homomorphism from $(C,\gamma)$ to 
$(C',\gamma')$}
is a morphism $f\colon C\longrightarrow C'$ in $\cat{C}$ making the 
following diagram commute:
\[
\begin{tikzpicture}[baseline=-\the\dimexpr\fontdimen22\textfont2\relax ]
      \node (TL) at (0,2)  {$T\ast C$};
      \node (TR) at (3,2)  {$T\ast C'$};
      \node (BL) at (0,0) {$C$};
      \node (BR) at (3,0) {$C'.$};
      \draw[->]  (TL) to node[auto,swap,labelsize] {$\gamma$} (BL);
      \draw[->]  (TL) to node[auto,labelsize] {$T\ast f$} (TR);
      \draw[->]  (BL) to node[auto,labelsize] {${f}$} (BR);
      \draw[->]  (TR) to node[auto,labelsize](B) {$\gamma'$} (BR);
\end{tikzpicture} 
\] 
\end{enumerate}
\end{definition}

We saw in Example~\ref{ex:non_symm_op_as_gen_op}
that non-symmetric operads arise as a special case of generalised operads.
In fact the definitions of models also agree, 
the models defined by the above definition (where $\monoid{S}$ is taken as in 
Example~\ref{ex:non_symm_op_as_gen_op}) being equivalent to the models as in  
Definition~\ref{def:ns_operad_model}, with $\cat{C}=\Set$ 
equipped with the cartesian monoidal structure.

\section{Metatheories and theories}
\label{sec:metatheories_theories}
In the previous section
we have seen several examples of notions of algebraic theory,
in which the corresponding types of {algebraic theories} are called
under various names, such as \emph{clones}, \emph{non-symmetric operads} and 
\emph{monads (on $\cat{C}$)}.
Being a background theory for a type of algebraic theories,
each notion of algebraic theory has definitions of algebraic theory, 
of model of an algebraic theory,
and of homomorphism between models.
Nevertheless, different notions of algebraic theory take 
different approaches to define these concepts, and the 
resulting definitions (say, of algebraic theory) can look quite remote.

As we have already mentioned, this paper introduces a 
\emph{unified framework for notions of algebraic theory}
which includes all of the notions
of algebraic theory reviewed in the previous section
as instances.
To the best of our knowledge, this is the first framework for notions of 
algebraic theory attaining such generality. 
Due to the diversity of notions of algebraic theory
we aim to capture, we take a very simple approach:
the basic idea is that we identify \emph{notions of algebraic theory}
with (large) \emph{monoidal categories},
and \emph{algebraic theories} with \emph{monoid objects} therein.


\begin{definition}
A \defemph{metatheory} is a large
monoidal category $\cat{M}=(\cat{M},I,\otimes)$.
%
\end{definition}

Metatheories are intended to formalise notions of algebraic theory.
We remark that, in this paper,
we leave the term \emph{notion of algebraic theory} informal
and will not give any mathematical definitions to it.

\begin{definition}
\label{def:theory}
Let $\cat{M}$ be a metatheory. 
A \defemph{theory in $\cat{M}$} is a monoid object
$\monoid{T}=(T,e,m)$ in $\cat{M}$.

We denote the category of theories in $\cat{M}$ by $\Th{\cat{M}}$,
which we define to be the same as
the category of monoid objects in $\cat{M}$ and homomorphisms between them. 
\end{definition}

Theories formalise what we have been calling \emph{algebraic theories}.

\medskip

The above definitions simply introduce aliases to well-known concepts.
Our hope is that, by using the terms which reflect our intention,
statements and discussions become easier to follow;
think of the terms such as \emph{generalised element} (which is synonymous
to morphism in a category) or \emph{map} (used by some authors to mean 
left adjoint in a bicategory)
which have been used with great benefit in the literature.

\subsection{The metatheory $[\F,\Set]$ for clones}
As a first example, we treat clones (Definition~\ref{def:clone}) within 
our framework.
That is, we seek a suitable metatheory 
for which clones are theories therein.

\begin{definition}
Let $\F$ be the category defined as follows:
\begin{itemize}
\item The set of objects is $\ob{\F}=\{\,[n]\mid n\in\NN\,\}$,
where for each natural number $n\in\NN$, $[n]$ is 
defined to be the $n$-element set $\{1,\dots,n\}$.
\item A morphism is any function between these sets.
\end{itemize}
\end{definition}
So the category $\F$ is a skeleton of the category $\FinSet$
of all (small) finite sets and functions.
The underlying category of the metatheory for clones 
is the category $[\F,\Set]$ of 
all functors from $\F$ to $\Set$ and natural transformations.
For $X\in[\F,\Set]$ and $[n]\in\F$,
we write the set $X([n])$ as $X_n$.

The monoidal structure on the category $[\F,\Set]$
we shall consider is known as the \emph{substitution monoidal 
structure}.
\begin{definition}[\cite{Kelly_Power_coeq,FPT99}]
The \defemph{substitution monoidal structure} on $[\F,\Set]$ is defined as 
follows:
\begin{itemize}
\item The unit object is $I=\F([1],-)\colon \F\longrightarrow\Set$.
\item Given $X,Y\colon \F\longrightarrow\Set$, their monoidal product 
$Y\otimes X\colon \F\longrightarrow\Set$ maps $[n]\in\F$ to 
\begin{equation}
\label{eqn:subst_tensor_F}
(Y\otimes X)_n=\int^{[k]\in\F} Y_k\times (X_n)^k.
\end{equation}
\end{itemize}
\end{definition}
The integral sign with a superscript in 
(\ref{eqn:subst_tensor_F}) stands for a \emph{coend}
(dually, we will denote an \emph{end} by the integral sign with a subscript);
see \cite[Section~IX.~6]{MacLane_CWM}.

\medskip

This defines a metatheory $[\F,\Set]$.
We claim that clones correspond to
theories in $[\F,\Set]$.
A theory in $[\F,\Set]$ consists of
\begin{itemize}
\item a functor $T\colon \F\longrightarrow \Set$;
\item a natural transformation $e\colon I\longrightarrow T$;
\item a natural transformation $m\colon T\otimes T\longrightarrow T$
\end{itemize}
satisfying the monoid axioms.
By the Yoneda lemma, $e$ corresponds to an element $\overline{e}\in T_1$,
and by the universality of coends, $m$ corresponds to
a natural transformation\footnote{In more detail, the relevant naturality
here may also be phrased as 
``natural in $[n]\in\F$ and \emph{extranatural} in $[k]\in\F$''; 
see \cite[Section~IX.~4]{MacLane_CWM}.
Following \cite{Kelly:enriched}, in this paper we shall not distinguish 
(terminologically)
extranaturality from naturality, using the latter term for both.}
\[
(\overline{m}_{n,k}\colon T_k\times (T_n)^k\longrightarrow T_n)_{n,k\in \NN}.
\]
Hence given a theory $(T,e,m)$ in $[\F,\Set]$, we can construct a 
clone with the underlying family of sets $(T_n)_{n\in\NN}$ by setting 
$p^{(n)}_i=T_\name{i}(\overline{e})$ (here, 
$\name{i}\colon[1]\longrightarrow[n]$
is the morphism in $\F$ defined as $\name{i}(1)=i$) and 
$\circ^{(n)}_k=\overline{m}_{n,k}$.
Conversely, given a clone $(T,p,\circ)$,
we can construct a theory in $[\F,\Set]$ as follows.
First we extend the family of sets $T$ 
to a functor $T\colon \F\longrightarrow\Set$
by setting, for any $u\colon [m]\longrightarrow [n]$ in $\F$,
\[
T_u (\theta)= \theta \circ^{(n)}_m (p^{(m)}_{u(1)},\dots, p^{(m)}_{u(m)}).
\]
Then we may set $\overline{e}=p^{(1)}_1$ and $\overline{m}_{n,k}=\circ^{(n)}_k$.

\begin{proposition}[Cf.~\cite{Curien_operad,Hyland_elts}]
The above constructions establish an isomorphism of categories 
between the category of clones and $\Th{[\F,\Set]}$.
\end{proposition}

\subsection{The metatheory $[\Pcat,\Set]$ for symmetric operads}
Symmetric operads can be similarly
seen as theories in a suitable metatheory.
The main difference from the case of clones is that,
instead of the category $\F$, we use the following category.

\begin{definition}
\label{def:Pcat}
Let $\Pcat$ be the category defined as follows:
\begin{itemize}
\item The set of objects is the same as $\F$: $\ob{\Pcat}=\{\,[n]\mid 
n\in\NN\,\}$ where $[n]=\{1,\dots,n\}$.
\item A morphism is any \emph{bijective} function.
\end{itemize}
\end{definition}
So $\Pcat$ is the subcategory of $\F$ consisting of all isomorphisms.

Symmetric operads are
theories in a metatheory whose underlying category is 
the functor category $[\Pcat,\Set]$.
The monoidal structure on $[\Pcat,\Set]$ we shall use is 
also called the substitution monoidal structure.
\begin{definition}[\cite{Kelly_operad}]
\label{def:subst_monoidal_P}
The \defemph{substitution monoidal structure} on $[\Pcat,\Set]$ is defined as 
follows:
\begin{itemize}
\item The unit object is $I=\Pcat([1],-)\colon\Pcat\longrightarrow\Set$.
\item Given $X,Y\colon\Pcat\longrightarrow\Set$, their monoidal product 
$Y\otimes X\colon \Pcat\longrightarrow\Set$ maps $[n]\in\Pcat$ to
\begin{equation*}
(Y\otimes X)_n
=\int^{[k]\in\Pcat} 
Y_k\times 
(X^{\circledast k})_n,
\end{equation*}
where 
\[
(X^{\circledast k})_n=\int^{[n_1],\dots,[n_k]\in\Pcat} 
\Pcat([n_1+\dots+n_k],[n])\times 
X_{n_1}\times\dots\times X_{n_k}.
\]
\end{itemize}
\end{definition}

\begin{proposition}[Cf.~\cite{Curien_operad,Hyland_elts}]
The category of symmetric operads is isomorphic to $\Th{[\Pcat,\Set]}$.
\end{proposition}

\subsection{The metatheory $[\Ncat,\Set]$ for non-symmetric operads}
For non-symmetric operads (Definition~\ref{def:non_symm_op}), 
we use the following category.

\begin{definition}
\label{def:Ncat}
Let $\Ncat$ be the category defined as follows:
\begin{itemize}
\item The set of objects is the same as $\F$ and $\Pcat$.
\item There are only identity morphisms in $\Ncat$.
\end{itemize}
\end{definition}
$\Ncat$ is the discrete category with the same objects as $\F$ and $\Pcat$.
We again consider the functor category $[\Ncat,\Set]$.

\begin{definition}
The \defemph{substitution monoidal structure} on $[\Ncat,\Set]$ is defined as 
follows:
\begin{itemize}
\item The unit object is $I=\Ncat([1],-)\colon \Ncat\longrightarrow \Set$.
\item Given $X,Y\colon \Ncat\longrightarrow\Set$, their monoidal product
$Y\otimes X\colon\Ncat\longrightarrow\Set$ maps $[n]\in\Ncat$ to 
\[
(Y\otimes X)_n=\coprod_{[k]\in\Ncat} 
Y_k\times\bigg(\coprod_{\substack{[n_1],\dots,[n_k]\in\Ncat\\n_1+\dots+n_k=n}}
X_{n_1}\times\dots\times X_{n_k}\bigg).
\]
\end{itemize}
\end{definition}

\begin{proposition}[Cf.~\cite{Curien_operad,Hyland_elts}]
The category of non-symmetric operads is isomorphic to $\Th{[\Ncat,\Set]}$.
\end{proposition}

\medskip

For unified studies of various substitution monoidal structures,
see \cite{Tanaka_Power_psd_dist,Fiore_et_al_rel_psdmnd}, as well as the 
aforementioned \cite{Curien_operad,Hyland_elts}.

\subsection{The metatheory $\MonCAT(\Ncat^\op\times\Ncat,\Set)$ for PROs}
In order to characterise PROs as theories in a metatheory,
we combine the following known facts.
\begin{itemize}
\item A PRO (Definition~\ref{def:PRO}) can be 
given equivalently as a small monoidal category $\monoid{T}$ equipped with
an identity-on-objects strict monoidal functor from $\Ncat$   
(Definition~\ref{def:Ncat}; the monoidal structure on $\Ncat$ is given by 
sum of natural numbers)
to $\monoid{T}$.
\item Fix a small category $\cat{C}$. 
A pair $(\cat{D},J)$ consisting of a small category 
$\cat{D}$ (having the same set of objects as $\cat{C}$) 
and an identity-on-objects functor 
$J\colon\cat{C}\longrightarrow\cat{D}$
is given equivalently as a monad on $\cat{C}$ in the bicategory of 
($\Set$-valued) profunctors,
that is to say as a monoid in the monoidal category of endo-profunctors on 
$\cat{C}$.
\end{itemize}
Hence it is natural to consider a monoidal version of endo-profunctors on
$\Ncat$.
Since the notion of profunctor (and its monoidal variant) will recur in 
this paper, let us fix our notation at this point.
As we shall later be more concerned with $\SET$-valued profunctors 
than $\Set$-valued ones (cf.~Section~\ref{sec:foundational_convention}),
we consider the former as default.
\begin{definition}[\cite{Benabou:dist-at-work,Lawvere_metric}]
\label{def:profunctor}
We define the bicategory $\PROF$ as follows.
\begin{itemize}
\item An object is a large category.
\item A 1-cell from $\cat{A}$ to $\cat{B}$ is a \defemph{profunctor from 
$\cat{A}$ to $\cat{B}$}, which we define to be 
a functor 
\[
H\colon \cat{B}^\op\times\cat{A}\longrightarrow\SET.
\]
We write $H\colon \cat{A}\pto\cat{B}$ if $H$ is a profunctor from $\cat{A}$
to $\cat{B}$.
The identity 1-cell on a large category $\cat{C}$
is the hom-functor $\cat{C}(-,-)$.
Given profunctors $H\colon \cat{A}\pto\cat{B}$ and $K\colon \cat{B}\pto\cat{C}$,
their composite 
$K\ptensor H\colon\cat{A}\pto\cat{C}$ maps $(C,A)\in\cat{C}^\op\times\cat{A}$
to 
\begin{equation}
\label{eqn:profunctor_composition}
(K\ptensor H)(C,A)=\int^{B\in\cat{B}} K(C,B)\times H(B,A).
\end{equation}
\item A 2-cell from $H$ to $H'$, both from $\cat{A}$ to $\cat{B}$,
is a natural transformation $\alpha\colon H\Longrightarrow H'
\colon\cat{B}^\op\times\cat{A}\longrightarrow\SET$. 
\end{itemize}
\end{definition}

Between small categories we also consider \defemph{$\Set$-valued 
profunctors}, defined by replacing $\SET$ by $\Set$ in the above definition of
profunctors.

\begin{definition}[Cf.~\cite{Im_Kelly}]
\label{def:monoidal_profunctor}
Let $\cat{M}=(\cat{M},I_\cat{M},\otimes_\cat{M})$ and 
$\cat{N}=(\cat{N},I_\cat{N},\otimes_\cat{N})$ be large monoidal categories.
A \defemph{monoidal profunctor} from $\cat{M}$ to 
$\cat{N}$ is a lax monoidal functor
\[
H=(H,h_\cdot,h)\colon \cat{N}^\op\times \cat{M}\longrightarrow\SET,
\]
where $\SET$ is regarded as a monoidal category 
via the cartesian monoidal structure.
In detail, it consists of:
\begin{itemize}
\item a functor $H\colon \cat{N}^\op\times\cat{M}\longrightarrow\SET$;
\item a function $h_\cdot\colon 1\longrightarrow H(I_\cat{N},I_\cat{M})$,
where $1$ is a singleton;
\item a natural transformation 
$
h=(h_{N,N',M,M'}\colon H(N',M')\times H(N,M)\longrightarrow H(N'\otimes_\cat{N}
N,M'\otimes_\cat{M} M))_{N,N'\in\cat{N},M,M'\in\cat{M}}
$
\end{itemize}
satisfying the suitable coherence axioms.
\end{definition}

One can compose monoidal profunctors by the formula 
(\ref{eqn:profunctor_composition}), giving rise to the bicategory of 
large monoidal categories, monoidal profunctors, and monoidal natural 
transformations; 
we shall use this in Section~\ref{sec:morphism_metatheory}.
Between small monoidal categories, one can define 
\defemph{$\Set$-valued monoidal profunctors} and their compositions.
Corresponding to the aforementioned fact about profunctors,
it can be shown that for any small monoidal category $\cat{M}$,
to give a monad on $\cat{M}$ in the bicategory of $\Set$-valued monoidal 
profunctors
is equivalent to give a small monoidal category $\cat{N}$ (having the same set
of objects as $\cat{M}$) together with an identity-on-objects
\emph{strict} monoidal functor $J\colon \cat{M}\longrightarrow\cat{N}$.
It follows that PROs correspond to theories in the metatheory of 
$\Set$-valued monoidal endo-profunctors on $\Ncat$, whose underlying category
is $\MonCAT(\Ncat^\op\times\Ncat,\Set)$, the category of lax monoidal 
functors of type $\Ncat^\op\times \Ncat\longrightarrow\Set$ and monoidal 
natural transformations.

\subsection{The metatheory $\SymMonCAT(\Pcat^\op\times\Pcat,\Set)$ for PROPs}
The case of PROPs is a straightforward adaptation of that of PROs.
For small symmetric monoidal categories $\cat{M}$ and $\cat{N}$,
we define a \defemph{$\Set$-valued 
symmetric monoidal profunctor} from $\cat{M}$ to $\cat{N}$
to be a lax symmetric monoidal functor of type 
$
\cat{N}^\op\times \cat{M}\longrightarrow\Set.
$
Similarly to $\Set$-valued monoidal profunctors,
$\Set$-valued symmetric monoidal profunctors compose and 
form a bicategory, a monad therein (say, on $\cat{M}$) being equivalent to 
a small symmetric monoidal category
equipped with an identity-on-objects strict symmetric monoidal functor
from $\cat{M}$.

We can endow the category $\Pcat$ (Definition~\ref{def:Pcat})
with the natural structure of a symmetric strict monoidal category.
As a PROP (Definition~\ref{def:PROP}) can be defined equivalently as a small 
symmetric monoidal category equipped with an identity-on-objects 
strict symmetric monoidal functor from $\Pcat$, 
we conclude that PROPs correspond to theories in 
the metatheory of $\Set$-valued symmetric monoidal endo-profunctors
on $\Pcat$, whose underlying category is 
$\SymMonCAT(\Pcat^\op\times\Pcat,\Set)$,
the category of lax symmetric monoidal functors of type 
$\Pcat^\op\times\Pcat\longrightarrow\Set$ and monoidal natural transformations.

\subsection{The metatheories for monads and generalised operads}
Finally, we recall that monads 
and generalised operads were introduced as monoid objects in the first place.
Hence for a large category $\cat{C}$, 
the metatheory for monads on $\cat{C}$ is $[\cat{C},\cat{C}]$ (with the 
composition as monoidal product), whereas 
for a large category $\cat{C}$ with finite limits and a cartesian monad 
$\monoid{S}$ thereon, the metatheory for $\monoid{S}$-operads is $\cat{C}/S1$
(with the monoidal structure defined earlier).

\section{Notions of model as enrichments}
\label{sec:enrichment}
Let us proceed to incorporate semantical aspects of notions of algebraic theory
into our framework.
We start with a discussion on \emph{notions of model}.
An important feature of several notions of algebraic 
theory---especially clones, symmetric operads, non-symmetric 
operads, PROPs and PROs---is that we may consider 
models of an algebraic theory in more than one category.
For example, 
models of a clone can be taken in any locally small category with 
finite products (Definition~\ref{def:clone_model}).
We may phrase this fact by saying that clones admit multiple {notions of 
model}, one for each locally small category with finite products.

Informally, a \emph{notion of model} for a notion of algebraic theory 
is a definition of model of an algebraic theory in that notion of algebraic 
theory.
Hence whenever we consider actual \emph{models} of an algebraic theory, we 
must specify in advance a notion of model with respect to which the models are 
taken.
Our framework emphasises the inevitable fact that
\emph{models are always relative to notions of model},
by treating notions of model explicitly as mathematical structures.

But how can we formalise such notions of model?
Below we show that the standard notions of model for clones, symmetric operads, 
non-symmetric operads, PROPs and PROs can be captured by a categorical structure
which we call \emph{enrichment}.

\begin{definition}[Cf.~{\cite[Definition 2.1]{Campbell_skew}}]
\label{def:enrichment}
Let $\cat{M}=(\cat{M},I,\otimes)$ be a metatheory.
An \defemph{enrichment over $\cat{M}$} consists of:
\begin{itemize}
\item  a large category $\cat{C}$;
\item a functor 
$\enrich{-}{-}\colon\cat{C}^\op\times\cat{C}\longrightarrow\cat{M}$;
\item a natural transformation $(j_C\colon 
I\longrightarrow\enrich{C}{C})_{C\in\cat{C}}$;
\item a natural transformation 
$(M_{A,B,C}\colon\enrich{B}{C}\otimes\enrich{A}{B}
\longrightarrow\enrich{A}{C})_{A,B,C\in\cat{C}}$,
\end{itemize}
making the following diagrams commute for all $A,B,C,D\in\cat{C}$:
\begin{equation*}
\begin{tikzpicture}[baseline=-\the\dimexpr\fontdimen22\textfont2\relax ]
      \node (TL) at (0,1)  {$I\otimes\enrich{A}{B}$};
      \node (TR) at (4,1)  {$\enrich{B}{B}\otimes\enrich{A}{B}$};
      \node (BR) at (4,-1) {$\enrich{A}{B}$};
      \draw[->] (TL) to node[auto,labelsize](T) {$j_B\otimes\enrich{A}{B}$} 
      (TR);
      \draw[->] (TR) to node[auto,labelsize](R) {$M_{A,B,B}$} (BR);
      \draw[->]  (TL) to node[auto,labelsize,swap] (L) {$\cong$} (BR);
\end{tikzpicture} 
\quad
\begin{tikzpicture}[baseline=-\the\dimexpr\fontdimen22\textfont2\relax ]
      \node (TL) at (0,1)  {$\enrich{A}{B}\otimes I$};
      \node (TR) at (4,1)  {$\enrich{A}{B}\otimes\enrich{A}{A}$};
      \node (BR) at (4,-1) {$\enrich{A}{B}$};
      \draw[->] (TL) to node[auto,labelsize](T) {$\enrich{A}{B}\otimes j_A$} 
      (TR);
      \draw[->] (TR) to node[auto,labelsize](R) {$M_{A,A,B}$} (BR);
      \draw[->]  (TL) to node[auto,labelsize,swap] (L) {$\cong$} (BR);
\end{tikzpicture} 
\end{equation*}
\begin{equation*}
\begin{tikzpicture}[baseline=-\the\dimexpr\fontdimen22\textfont2\relax ]
      \node (TL) at (0,1)  {$(\enrich{C}{D}\otimes \enrich{B}{C})\otimes 
      \enrich{A}{B}$};
      \node (TR) at (9,1)  {$\enrich{B}{D}\otimes\enrich{A}{B}$};
      \node (BL) at (0,-1) 
      {$\enrich{C}{D}\otimes(\enrich{B}{C}\otimes\enrich{A}{B})$};
      \node (BM) at (6,-1) {$\enrich{C}{D}\otimes\enrich{A}{C}$};
      \node (BR) at (9,-1) {$\enrich{A}{D}.$};
      \draw[->] (TL) to node[auto,labelsize](T) 
      {$M_{B,C,D}\otimes\enrich{A}{B}$} 
      (TR);
      \draw[->]  (TR) to node[auto,labelsize] {$M_{A,B,D}$} (BR);
      \draw[->]  (TL) to node[auto,swap,labelsize] {$\cong$} (BL);
      \draw[->]  (BL) to node[auto,labelsize] {$\enrich{C}{D}\otimes 
      M_{A,B,C}$} (BM);
      \draw[->] (BM) to node[auto,labelsize](B) {$M_{A,C,D}$} (BR);
\end{tikzpicture} 
\end{equation*}
We say that $(\cat{C},\enrich{-}{-},j,M)$ is an enrichment over $\cat{M}$,
or that $(\enrich{-}{-},j,M)$ is an enrichment of $\cat{C}$ over $\cat{M}$.
\end{definition}

An enrichment over $\cat{M}$ is 
not the same as a (large) \emph{$\cat{M}$-category} in enriched category 
theory~\cite{Kelly:enriched}.
It is rather a triple consisting of a large category $\cat{C}$, 
a large $\cat{M}$-category $\cat{D}$
and an identity-on-objects functor $J\colon \cat{C}\longrightarrow\cat{D}_0$,
where $\cat{D}_0$ is the underlying category of $\cat{D}$.

In more detail, given an enrichment $(\enrich{-}{-},j,M)$ of $\cat{C}$ in 
$\cat{M}$,
we may define the $\cat{M}$-category $\cat{D}$ with $\ob{\cat{D}}=\ob{\cat{C}}$
using the data $(\enrich{-}{-},j,M)$ of the enrichment, by setting
$\cat{D}(A,B)=\enrich{A}{B}$ and so on.
The identity-on-objects functor $J\colon \cat{C}\longrightarrow\cat{D}_0$
may be defined by mapping a morphism $f\colon A\longrightarrow B$
in $\cat{C}$ to
the composite $\enrich{A}{f}\circ j_A$, or equivalently, $\enrich{f}{B}\circ 
j_B$:
\[
\begin{tikzpicture}[baseline=-\the\dimexpr\fontdimen22\textfont2\relax ]
      \node (TL) at (0,2)  {$I$};
      \node (TR) at (3,2)  {$\enrich{B}{B}$};
      \node (BL) at (0,0) {$\enrich{A}{A}$};
      \node (BR) at (3,0) {$\enrich{A}{B}.$};
      \draw[->]  (TL) to node[auto,swap,labelsize] {$j_A$} (BL);
      \draw[->]  (TL) to node[auto,labelsize] {$j_B$} (TR);
      \draw[->]  (BL) to node[auto,labelsize] {$\enrich{A}{f}$} (BR);
      \draw[->]  (TR) to node[auto,labelsize](B) {$\enrich{f}{B}$} (BR);
\end{tikzpicture} 
\]
See \cite[Section 2]{Campbell_skew} and \cite[Section~3.1.2]{Fujii_thesis} for more about
the relationship to enriched category theory.

\medskip

From an enrichment, we now derive a definition of model 
of a theory.
First observe that, given an enrichment $\enrich{-}{-}=(\enrich{-}{-},j,M)$ of 
a large category $\cat{C}$ over a metatheory $\cat{M}$
and an object $C\in\cat{C}$,
we have a theory $\End_{\enrich{-}{-}}(C)=(\enrich{C}{C},j_C,M_{C,C,C})$
in $\cat{M}$.

\begin{definition}
\label{def:enrich_model}
Let $\cat{M}=(\cat{M},I,\otimes)$ be a metatheory, 
$\monoid{T}=(T,e,m)$ be a theory in $\cat{M}$, 
$\cat{C}$ be a large category, and $\enrich{-}{-}=(\enrich{-}{-},j,M)$ be an
enrichment of $\cat{C}$ over $\cat{M}$.
\begin{enumerate}
\item A \defemph{model of $\monoid{T}$ in $\cat{C}$ with respect to 
$\enrich{-}{-}$}
is a pair $(C,\chi)$ consisting of an object $C$ of $\cat{C}$
and a theory homomorphism (= monoid homomorphism) $\chi\colon 
\monoid{T}\longrightarrow\End_\enrich{-}{-}(C)$; that is, 
a morphism $\chi\colon T\longrightarrow \enrich{C}{C}$
in $\cat{M}$ making the following diagrams commute:
\begin{equation*}
\begin{tikzpicture}[baseline=-\the\dimexpr\fontdimen22\textfont2\relax ]
      \node (TL) at (0,2)  {$I$};
      \node (TR) at (2,2)  {$T$};
      \node (BR) at (2,0) {$\enrich{C}{C}$};
      \draw[->]  (TL) to node[auto,labelsize] {$e$} (TR);
      \draw[->]  (TR) to node[auto,labelsize] {$\chi$} (BR);
      \draw[->] (TL) to node[auto,swap,labelsize](B) {$j_C$} (BR);
\end{tikzpicture} 
\qquad
\begin{tikzpicture}[baseline=-\the\dimexpr\fontdimen22\textfont2\relax ]
      \node (TL) at (0,2)  {$T\otimes T$};
      \node (TR) at (4,2)  {$T$};
      \node (BL) at (0,0) {$\enrich{C}{C}\otimes\enrich{C}{C}$};
      \node (BR) at (4,0) {$\enrich{C}{C}.$};
      \draw[->]  (TL) to node[auto,swap,labelsize] {$\chi\otimes\chi$} (BL);
      \draw[->]  (TL) to node[auto,labelsize] {$m$} (TR);
      \draw[->]  (BL) to node[auto,labelsize] {$M_{C,C,C}$} (BR);
      \draw[->]  (TR) to node[auto,labelsize](B) {$\chi$} (BR);
\end{tikzpicture} 
\end{equation*}
\item Let $(C,\chi)$ and $(C',\chi')$ be models of $\monoid{T}$ in $\cat{C}$
with respect to $\enrich{-}{-}$. A \defemph{homomorphism from $(C,\chi)$ to 
$(C',\chi')$}
is a morphism $f\colon C\longrightarrow C'$ in $\cat{C}$ making the 
following diagram commute:
\[
\begin{tikzpicture}[baseline=-\the\dimexpr\fontdimen22\textfont2\relax ]
      \node (TL) at (0,2)  {$T$};
      \node (TR) at (3,2)  {$\enrich{C'}{C'}$};
      \node (BL) at (0,0) {$\enrich{C}{C}$};
      \node (BR) at (3,0) {$\enrich{C}{C'}.$};
      \draw[->]  (TL) to node[auto,swap,labelsize] {$\chi$} (BL);
      \draw[->]  (TL) to node[auto,labelsize] {$\chi'$} (TR);
      \draw[->]  (BL) to node[auto,labelsize] {$\enrich{C}{f}$} (BR);
      \draw[->]  (TR) to node[auto,labelsize](B) {$\enrich{f}{C'}$} (BR);
\end{tikzpicture} 
\] 
\end{enumerate}
We denote the (large) category of models of $\monoid{T}$ in $\cat{C}$ with 
respect to 
$\enrich{-}{-}$ by $\Mod{\monoid{T}}{(\cat{C},\enrich{-}{-})}$.
\end{definition}

The above definitions of model and homomorphism are reminiscent of 
ones for clones (Definitions~\ref{def:clone_model} and \ref{def:clone_mod_hom}),
symmetric operads, non-symmetric operads
(Definition~\ref{def:ns_operad_model}),
PROPs and PROs (Definition~\ref{def:PRO_model}).
Indeed, we can recover the standard notions of model
for these notions of algebraic theory via suitable enrichments.

\begin{example}
\label{ex:clone_enrichment}
Recall that the metatheory for clones is $[\F,\Set]$ with the substitution 
monoidal structure.
Let $\cat{C}$ be a locally small\footnote{Note that according to
our convention (Definition~\ref{conv:size}), ``locally small'' implies 
``large''.} 
category with all finite products
(or more generally, with all finite powers).
We have an enrichment of $\cat{C}$ over $[\F,\Set]$ defined as follows:
\begin{itemize}
\item The functor $\enrich{-}{-}\colon \cat{C}^\op\times\cat{C}
\longrightarrow[\F,\Set]$
maps $A,B\in\cat{C}$ and $[n]\in\F$ to the set 
\[
\enrich{A}{B}_n=\cat{C}(A^n,B).
\]
\item The natural transformation $(j_C\colon I\longrightarrow 
\enrich{C}{C})_{C\in\cat{C}}$ corresponds by the Yoneda lemma
(recall that $I=\F([1],-)$) to the family 
\[
(\overline{j}_C=\id{C}\in \enrich{C}{C}_1)_{C\in\cat{C}}.
\] 
\item The natural transformation $(M_{A,B,C}\colon \enrich{B}{C}\otimes
\enrich{A}{B}\longrightarrow\enrich{A}{C})_{A,B,C\in\cat{C}}$
corresponds by the universality of coends (recall
that $(Y\otimes X)_{n}=\int^{[k]\in\F} Y_k\times (X_n)^k$) to 
the family
\[
(\overline{M}_{A,B,C})_{n,k}\colon \enrich{B}{C}_k\times (\enrich{A}{B}_n)^k
\longrightarrow\enrich{A}{B}_n
\]
mapping $(g,f_1,\dots,f_k)\in \cat{C}(B^k,C)\times \cat{C}(A^n,B)^k$
to $g\circ\langle f_1,\dots,f_k\rangle\in\cat{C}(A^n,B)$.
\end{itemize}

Clearly the definition of the clone $\End_\cat{C}(C)$ from an object
$C\in\cat{C}$ 
(Definition~\ref{ex:endoclone}) is derived from the above enrichment.
Consequently, we recover the classical definitions of model 
(Definition~\ref{def:clone_model})
and homomorphism between models (Definition~\ref{def:clone_mod_hom}) for clones,
as instances of Definition~\ref{def:enrich_model}.
\end{example}

\begin{example}
The metatheory for symmetric operads is $[\Pcat,\Set]$ with the substitution 
monoidal structure.
Let $\cat{C}=(\cat{C},I',\otimes')$ 
be a locally small symmetric monoidal category.
We have an enrichment of $\cat{C}$ over $[\Pcat,\Set]$ defined as follows:
\begin{itemize}
\item The functor $\enrich{-}{-}\colon \cat{C}^\op\times\cat{C}
\longrightarrow[\Pcat,\Set]$
maps $A,B\in\cat{C}$ and $[n]\in\Pcat$ to the set 
\[
\enrich{A}{B}_n=\cat{C}(A^{\otimes' n},B),
\]
where $A^{\otimes' n}$ is the monoidal product of $n$ many copies of $A$.
\item The natural transformation $(j_C\colon I\longrightarrow 
\enrich{C}{C})_{C\in\cat{C}}$ corresponds by the Yoneda lemma
(recall that $I=\Pcat([1],-)$) to the family 
\[
(\overline{j}_C=\id{C}\in \enrich{C}{C}_1)_{C\in\cat{C}}.
\] 
\item The natural transformation $(M_{A,B,C}\colon \enrich{B}{C}\otimes
\enrich{A}{B}\longrightarrow\enrich{A}{C})_{A,B,C\in\cat{C}}$
corresponds by the universality of coends (recall 
Definition~\ref{def:subst_monoidal_P}) to 
the family 
\begin{multline*}
(\overline{M}_{A,B,C})_{n,k,n_1,\dots,n_k}\colon 
\enrich{B}{C}_k\times \Pcat([n_1+\dots+n_k],[n])\\ \times
\enrich{A}{B}_{n_1}\times\dots\times\enrich{A}{B}_{n_k}
\longrightarrow\enrich{A}{B}_n,
\end{multline*}
which is the unique function from the empty set if $n\neq n_1+\dots+n_k$
and, if $n= n_1+\dots+n_k$, maps 
$(g,u,f_1,\dots,f_k)\in \cat{C}(B^{\otimes' k},C)\times 
\Pcat([n_1+\dots+n_k],[n])\times \cat{C}(A^{\otimes' n_1},B)
\times\dots\times \cat{C}(A^{\otimes' n_k},B)$
to $g\circ (f_1\otimes'\dots\otimes' f_k)\circ A^{\otimes' u}$.
\end{itemize}

Via the above enrichment, we recover the classical definitions of model and 
homomorphism between models for symmetric operads.
\end{example}

\begin{example}
The metatheory for non-symmetric operads is $[\Ncat,\Set]$ with the 
substitution monoidal structure.
Let $\cat{C}=(\cat{C},I',\otimes')$ 
be a locally small monoidal category.
We have an enrichment of $\cat{C}$ over $[\Ncat,\Set]$ which is
similar to, and simpler than, the one in the previous example.
\end{example}

\begin{example}
The metatheory for PROPs is $\SymMonCAT(\Pcat^\op\times\Pcat,\Set)$
with the monoidal structure given by composition of $\Set$-valued 
symmetric monoidal profunctors.
Let $\cat{C}=(\cat{C},I',\otimes')$ be a locally small symmetric monoidal 
category.
We have an enrichment of $\cat{C}$ over $\SymMonCAT(\Pcat^\op\times\Pcat,\Set)$
defined as follows:
\begin{itemize}
\item The functor $\enrich{-}{-}\colon \cat{C}^\op\times\cat{C}
\longrightarrow\SymMonCAT(\Pcat^\op\times\Pcat,\Set)$
maps $A,B\in\cat{C}$ and $[m],[n]\in\Pcat$ to the set
\[
\enrich{A}{B}([m],[n])=\cat{C}(A^{\otimes' n},B^{\otimes' m}),
\]
with the evident lax symmetric monoidal structure
on $\enrich{A}{B}$.
\item The natural transformation $(j_C\colon \Pcat(-,-)\longrightarrow
\enrich{C}{C})$ is determined by the functoriality of $\enrich{C}{C}$.
\item The natural transformation $(M_{A,B,C}\colon \enrich{B}{C}\ptensor
\enrich{A}{B}\longrightarrow\enrich{A}{C})_{A,B,C\in\cat{C}}$
corresponds by the universality of coends to the family
\[
(\overline{M}_{A,B,C})_{l,m,n}\colon \enrich{B}{C}([l],[m])\times
\enrich{A}{B}([m],[n])\longrightarrow\enrich{A}{C}([l],[n])
\]
defined by composition in $\cat{C}$.
\end{itemize}

This enrichment captures the classical definition of models of PROPs.
\end{example}

\begin{example}
The metatheory for PROs is $\MonCAT(\Ncat^\op\times\Ncat,\Set)$ with the 
monoidal structure given by composition of $\Set$-valued monoidal profunctors.
Let $\cat{C}=(\cat{C},I',\otimes')$ 
be a locally small monoidal category.
We have an enrichment of $\cat{C}$ over $\MonCAT(\Ncat^\op\times\Ncat,\Set)$ 
which is similar to, and simpler than, the one in the previous example.
\end{example}

\begin{example}
\label{ex:relative_alg}
We may also consider infinitary variants of Example~\ref{ex:clone_enrichment}.
Here we take an extreme. 
Let $\cat{C}$ be a locally small category with all small powers. 
Then we obtain an enrichment of $\cat{C}$ over $[\Set,\Set]$,
the metatheory for monads on $\Set$.
\begin{itemize}
\item The functor $\enrich{-}{-}\colon \cat{C}^\op\times 
\cat{C}\longrightarrow[\Set,\Set]$ maps $A,B\in\cat{C}$
and $X\in \Set$ to the set 
\[
\enrich{A}{B}(X)=\cat{C}(A^X,B),
\]
where $A^X$ is the $X$-th power of $A$.
\item The natural transformation $(j_C\colon \id{\Set}\longrightarrow 
\enrich{C}{C})_{C\in\cat{C}}$ corresponds by the Yoneda lemma
(note that $\id{\Set}\cong\Set(1,-)$, where $1$ is a singleton)
to the family
\[
(\overline{j}_C=\id{C}\in\enrich{C}{C}(1))_{C\in\cat{C}}.
\]
\item The natural transformation $(M_{A,B,C}\colon 
\enrich{B}{C}\circ\enrich{A}{B}\longrightarrow\enrich{A}{C})_{A,B,C\in\cat{C}}$
has the $X$-th component ($X\in\Set$)
\[
\enrich{B}{C}\circ\enrich{A}{B}(X)=\cat{C}(B^{\cat{C}(A^X,B)},C)\longrightarrow
\cat{C}(A^X,C)=\enrich{A}{C}(X)
\]
the function induced from the canonical morphism $A^X\longrightarrow 
B^{\cat{C}(A^X,B)}$ in $\cat{C}$.
\end{itemize}

This enrichment gives us a definition of model of a monad $\monoid{T}$
on $\Set$ in $\cat{C}$. 
To spell this out, first note that 
for any object $C\in\cat{C}$, the functor $\enrich{C}{C}\colon 
\Set\longrightarrow\Set$ which maps $X\in\Set$ to $\cat{C}(C^X,C)$
acquires the monad structure, giving rise to the monad $\End_{\enrich{-}{-}}(C)$
on $\Set$.
A model of $\monoid{T}$ is then an object $C\in\cat{C}$ together with
a monad morphism $\monoid{T}\longrightarrow\End_{\enrich{-}{-}}(C)$.
This is the definition of \emph{relative algebra} of a monad on $\Set$ by 
Hino, Kobayashi, Hasuo and Jacobs~\cite{Hino}.
As noted in \cite{Hino}, in the case where $\cat{C}=\Set$,
relative algebras of a monad $\monoid{T}$ on $\Set$ agree with Eilenberg--Moore 
algebras of $\monoid{T}$; we shall later show this fact
in Example~\ref{ex:monad_on_cat_w_powers}.
\end{example}

\begin{example}
Let $\cat{S}$ be a large category and consider the metatheory 
$[\cat{S},\cat{S}]$ of  monads on $\cat{S}$.
Then an enrichment over $[\cat{S},\cat{S}]$ is the same thing as 
an \emph{$\cat{S}$-parameterised monad} (without strength) in the sense of 
Atkey~\cite[Definition~1]{Atkey_parameterised}, introduced
in the study of computational effects.
\end{example}

\medskip

Having reformulated semantics of notions of algebraic theory
in terms of enrichments, let us 
investigate some of its immediate consequences.

\subsection{$\Mod{-}{-}$ as a 2-functor}
It is well-known that given clones $\monoid{T}$ and $\monoid{T'}$,
a clone homomorphism $f\colon \monoid{T}\longrightarrow\monoid{T'}$,
and a locally small category $\cat{C}$ with finite products, 
we have the induced functor 
\[
\Mod{f}{\cat{C}}\colon\Mod{\monoid{T'}}{\cat{C}}\longrightarrow
\Mod{\monoid{T}}{\cat{C}}
\]
between the categories of models.
An equally widely known phenomenon is that
given a clone $\monoid{T}$, locally small 
categories $\cat{C}$ and 
$\cat{C'}$ with finite products, and a functor $G\colon\cat{C}\longrightarrow 
\cat{C'}$ preserving finite products, we obtain a functor 
\[
\Mod{\monoid{T}}{G}\colon\Mod{\monoid{T}}{\cat{C}}\longrightarrow
\Mod{\monoid{T}}{\cat{C'}}.
\]
In order to capture such functoriality of $\Mod{-}{-}$,
we introduce a 2-category of enrichments.

\begin{definition}[{Cf.~\cite[Definitions 2.9 and 2.10]{Campbell_skew}}]
\label{def:2-cat_of_enrichments}
Let $\cat{M}=(\cat{M},I,\otimes)$ be a metatheory.
The (locally large) 2-category $\Enrich{\cat{M}}$ of enrichments over $\cat{M}$ 
is defined as follows:
\begin{itemize}
\item An object is an enrichment $(\cat{C},\enrich{-}{-},j,M)$ over $\cat{M}$.
\item A 1-cell from $(\cat{C},\enrich{-}{-},j,M)$ to
$(\cat{C'},\enrich{-}{-}',j',M')$ is a functor 
$G\colon\cat{C}\longrightarrow\cat{C'}$ together with a natural transformation
$(g_{A,B}\colon\enrich{A}{B}\longrightarrow\enrich{GA}{GB}')_{A,B\in\cat{C}}$ 
making the 
following diagrams commute for all $A,B,C\in\cat{C}$:
\begin{equation*}
\begin{tikzpicture}[baseline=-\the\dimexpr\fontdimen22\textfont2\relax ]
      \node (TL) at (0,1)  {$I$};
      \node (TR) at (4,1)  {$\enrich{C}{C}$};
      \node (BR) at (4,-1) {$\enrich{GC}{GC}'$};
      \draw[->] (TL) to node[auto,labelsize](T) {$j_C$} (TR);
      \draw[->] (TR) to node[auto,labelsize](R) {$g_{C,C}$} (BR);
      \draw[->]  (TL) to node[auto,labelsize,swap] (L) {$j'_{GC}$} 
      (BR);
\end{tikzpicture} 
\end{equation*}
\begin{equation*}
\begin{tikzpicture}[baseline=-\the\dimexpr\fontdimen22\textfont2\relax ]
      \node (TL) at (0,1)  {$\enrich{B}{C}\otimes\enrich{A}{B}$};
      \node (TR) at (6,1)  {$\enrich{A}{C}$};
      \node (BL) at (0,-1) {$\enrich{GB}{GC}'\otimes\enrich{GA}{GB}'$};
      \node (BR) at (6,-1) {$\enrich{GA}{GC}'.$};
      \draw[->] (TL) to node[auto,labelsize](T) {$M_{A,B,C}$} 
      (TR);
      \draw[->]  (TR) to node[auto,labelsize] {$g_{A,C}$} (BR);
      \draw[->]  (TL) to node[auto,swap,labelsize] {$g_{B,C}\otimes g_{A,B}$} 
      (BL);
      \draw[->]  (BL) to node[auto,labelsize] {$M'_{GA,GB,GC}$} (BR);
\end{tikzpicture} 
\end{equation*}
\item A 2-cell from $(G,g)$ to $(G',g')$, both from 
$(\cat{C},\enrich{-}{-},j,M)$ to $(\cat{C'},\enrich{-}{-}',j',$ $M')$, is a 
natural transformation $\theta\colon G\Longrightarrow G'$
making the following diagram commute for all $A,B\in\cat{C}$:
\begin{equation*}
\begin{tikzpicture}[baseline=-\the\dimexpr\fontdimen22\textfont2\relax ]
      \node (TL) at (0,2)  {$\enrich{A}{B}$};
      \node (TR) at (4,2)  {$\enrich{GA}{GB}'$};
      \node (BL) at (0,0) {$\enrich{G'A}{G'B}'$};
      \node (BR) at (4,0) {$\enrich{GA}{G'B}'.$};
      \draw[->] (TL) to node[auto,labelsize](T) {$g_{A,B}$} 
      (TR);
      \draw[->]  (TR) to node[auto,labelsize] {$\enrich{GA}{\theta_B}'$} (BR);
      \draw[->]  (TL) to node[auto,swap,labelsize] {$g'_{A,B}$} 
      (BL);
      \draw[->]  (BL) to node[auto,labelsize] {$\enrich{\theta_A}{G'B}'$} (BR);
\end{tikzpicture} 
\end{equation*}
\end{itemize}
\end{definition}

\begin{example}
Let $\FPow$ be the 2-category of locally small categories with chosen finite 
powers,
functors preserving finite powers (in the usual sense\footnote{That is,
we do not require these functors to preserve the chosen finite powers on the 
nose.}) 
and all natural transformations.
We have a canonical 2-functor 
\[
\FPow\longrightarrow\Enrich{[\F,\Set]}
\]
which is fully faithful.

Let $\FProd$ be the 2-category of locally small categories with chosen finite 
products,
functors preserving finite products (in the usual sense) 
and all natural transformations.
We have a canonical 2-functor 
\[
\FProd\longrightarrow\Enrich{[\F,\Set]}
\]
which is locally fully faithful.

Hence we may recover the classical functoriality of $\Mod{\monoid{T}}{-}$
for a clone $\monoid{T}$, recalled above, if we could show that it is 
functorial with respect to morphisms in $\Enrich{[\F,\Set]}$.
\end{example}

We also have canonical (locally faithful) 2-functors
\begin{align*}
\SymMonCATls&\longrightarrow\Enrich{[\Pcat,\Set]},\\
\SymMonCATstls&\longrightarrow\Enrich{\SymMonCAT(\Pcat^\op\times\Pcat,\Set)},
\end{align*}
where the domain is the 2-category of locally small symmetric monoidal 
categories, symmetric lax (resp.~strong) monoidal functors  and monoidal 
natural transformations, and
\begin{align*}
\MonCATls&\longrightarrow\Enrich{[\Ncat,\Set]},\\
\MonCATstls&\longrightarrow\Enrich{\MonCAT(\Ncat^\op\times\Ncat,\Set)},
\end{align*}
where the domain is the 2-category of locally small monoidal categories,
lax (resp.~strong) monoidal functors and monoidal natural transformations.

\medskip
Now the functoriality of $\Mod{-}{-}$
may be expressed by saying that it is a 2-functor
\begin{equation}
\label{eqn:enrich_Mod_bifunctorial}
\Mod{-}{-}\colon \Th{\cat{M}}^\op\times\Enrich{\cat{M}}\longrightarrow \tCAT
\end{equation}
(when we say that (\ref{eqn:enrich_Mod_bifunctorial}) is a 
2-functor, we are identifying the category $\Th{\cat{M}}$ with the 
corresponding locally discrete 2-category).
Actually, the 2-functor~(\ref{eqn:enrich_Mod_bifunctorial}) arises
immediately from the structure of
$\Enrich{\cat{M}}$.
Observe that we may identify a theory in $\cat{M}$ with 
an enrichment of the terminal category $1$ over $\cat{M}$.
The full sub-2-category of $\Enrich{\cat{M}}$ consisting of 
all enrichments over the (fixed) terminal category $1$ is in fact 
locally discrete, and is isomorphic to $\Th{\cat{M}}$.
This way we obtain a fully faithful inclusion 2-functor
$\Th{\cat{M}}\longrightarrow\Enrich{\cat{M}}$.
It is straightforward to see that the appropriate 2-functor 
(\ref{eqn:enrich_Mod_bifunctorial}) is given by the composite
\[
\begin{tikzpicture}[baseline=-\the\dimexpr\fontdimen22\textfont2\relax ]
      \node (1) at (0,0)  {$\Th{\cat{M}}^\op\times \Enrich{\cat{M}}$};
      \node (2) at (0,-1.5)  {$\Enrich{\cat{M}}^\op\times \Enrich{\cat{M}}$};
      \node (3) at (0,-3) {$\tCAT$,};
      \draw[->] (1) to node[auto,labelsize]{inclusion} (2);
      \draw[->] (2) to node[auto,labelsize] {$\Enrich{\cat{M}}(-,-)$} (3);
\end{tikzpicture}
\]
where $\Enrich{\cat{M}}(-,-)$ is the hom-2-functor for $\Enrich{\cat{M}}$.

\subsection{Comparing different notions of algebraic theory}
So far we have been working within a fixed notion of algebraic theory.
We now turn to the question of comparing different notions of algebraic 
theory.

By way of illustration,
let us consider the relationship of 
clones, symmetric operads and non-symmetric operads.
On the ``syntactical'' side, we have inclusions
of algebraic theories
\begin{equation}
\label{eqn:chain_alg_thy}
\{\text{non-sym.~operads}\}\subseteq\{\text{sym.~operads}\}\subseteq
\{\text{clones}\},
\end{equation}
in the sense that every symmetric operad may be derived from 
a regular presentation of an equational theory,
which is at the same time a presentation of an equational theory
and therefore defines a clone, etc.
On the ``semantical'' side, in contrast, we have 
inclusions of (standard) notions of models in the opposite direction,
namely:
\begin{equation}
\label{eqn:chain_semantics}
\{\text{mon.~cat.}\}
\supseteq\{\text{sym.~mon.~cat.}\}
\supseteq\{\text{cat.~with fin.~prod.}\}.
\end{equation}

Furthermore,
suppose we take the theory $\monoid{T}$ of monoids 
(which is expressible as a non-symmetric operad)
and the category $\Set$ (which has finite products).
Then we can consider the category of models $\Mod{\monoid{T}}{\Set}$
in three different ways:
either thinking of $\monoid{T}$ as a clone and $\Set$ as a category with finite 
products,
$\monoid{T}$ as a symmetric operad and $\Set$ as a symmetric monoidal category
(via finite products),
or $\monoid{T}$ as a non-symmetric operad and $\Set$ as a monoidal category
(via finite products).
It turns out that the resulting three categories of models are isomorphic to 
each other,
indicating certain compatibility between the three notions of algebraic theory.

\medskip

The key to understand these phenomena in our framework is 
the functoriality of the $\Enrich{-}$ construction. 
That is, we may extend (just like base change of
enriched categories) $\Enrich{-}$ to a 2-functor
\begin{equation}
\label{eqn:Enrich_as_2-functor_1st}
\Enrich{-}\colon \MonCAT\longrightarrow \twoCAT
\end{equation}
from the 2-category $\MonCAT$ 
of large monoidal categories (= metatheories), lax monoidal functors 
and monoidal natural transformations,
to the 2-category $\twoCAT$ of huge 2-categories, 2-functors and 
2-natural transformations.
We just describe the action of a lax monoidal functor on an enrichment,
as the rest of the data for the 2-functor (\ref{eqn:Enrich_as_2-functor_1st}) 
follows routinely.
\begin{definition}
\label{def:action_of_lax_on_enrichment}
Let $\cat{M}=(\cat{M},I_\cat{M},\otimes_\cat{M})$ and 
$\cat{N}=(\cat{N},I_\cat{N},\otimes_\cat{N})$ be metatheories,
$F=(F,f_\cdot,f)\colon\cat{M}\longrightarrow\cat{N}$ be a lax monoidal 
functor,\footnote{Consisting of a functor 
$F\colon\cat{M}\longrightarrow\cat{N}$,
a morphism $f_\cdot\colon I_\cat{N}\longrightarrow FI_\cat{M}$
and a natural transformation $(f_{X,Y}\colon FY\otimes_\cat{N}FX\longrightarrow 
F(Y\otimes_\cat{M} X))_{X,Y\in\cat{M}}$ satisfying the suitable axioms.} 
$\cat{C}$ be a large category and $\enrich{-}{-}=(\enrich{-}{-},j,M)$ be an 
enrichment of $\cat{C}$ over $\cat{M}$.
We define the enrichment $F_\ast(\enrich{-}{-})=(\enrich{-}{-}',j',M')$ 
of $\cat{C}$ over $\cat{N}$ as follows:
\begin{itemize}
\item The functor 
$\enrich{-}{-}'\colon\cat{C}^\op\times\cat{C}\longrightarrow\cat{N}$
maps $(A,B)\in\cat{C}^\op\times\cat{C}$ to $F\enrich{A}{B}$.
\item The natural transformation $(j'_C\colon 
I_\cat{N}\longrightarrow\enrich{C}{C}')_{C\in\cat{C}}$
is defined by $j'_C=Fj_C\circ f_\cdot$:
\[
\begin{tikzpicture}[baseline=-\the\dimexpr\fontdimen22\textfont2\relax ]
      \node (L) at (0,0)  {$I_\cat{N}$};
      \node (M) at (2,0)  {$FI_\cat{M}$};
      \node (R) at (4.5,0)  {$F\enrich{C}{C}$.};
      \draw[->]  (L) to node[auto,labelsize] {$f_\cdot$} (M);
      \draw[->]  (M) to node[auto,labelsize] {$Fj_C$} (R);
\end{tikzpicture} 
\]
\item The natural transformation $(M'_{A,B,C}\colon 
\enrich{B}{C}'\otimes_\cat{N}\enrich{A}{B}'\longrightarrow\enrich{A}{C}')_{A,
B,C\in\cat{C}}$
is defined by $M'_{A,B,C}=FM_{A,B,C}\circ f_{\enrich{A}{B},\enrich{B}{C}}$:
\[
\begin{tikzpicture}[baseline=-\the\dimexpr\fontdimen22\textfont2\relax ]
      \node (L) at (0,0)  {$F\enrich{B}{C}\otimes_\cat{N}F\enrich{A}{B}$};
      \node (M) at (5.5,0)  {$F(\enrich{B}{C}\otimes_\cat{M}\enrich{A}{B})$};
      \node (R) at (9.5,0)  {$F\enrich{A}{C}$.};
      \draw[->]  (L) to node[auto,labelsize] 
      {$f_{\enrich{A}{B},\enrich{B}{C}}$} (M);
      \draw[->]  (M) to node[auto,labelsize] {$FM_{A,B,C}$} (R);
\end{tikzpicture} 
\]
\end{itemize}
\end{definition}
As an immediate consequence of the 2-functoriality 
of $\Enrich{-}$,
it follows that whenever we have a monoidal adjunction (adjunction in 
$\MonCAT$) 
\[
\begin{tikzpicture}[baseline=-\the\dimexpr\fontdimen22\textfont2\relax ]
      \node (L) at (0,0)  {$\cat{M}$};
      \node (R) at (3,0)  {$\cat{N}$,};
      \draw[->,transform canvas={yshift=5pt}]  (L) to node[auto,labelsize] 
      {$L$} (R);
      \draw[<-,transform canvas={yshift=-5pt}]  (L) to 
      node[auto,swap,labelsize] {$R$} (R);
      \node[rotate=-90,labelsize] at (1.5,0)  {$\dashv$};
\end{tikzpicture} 
\] 
we obtain a 2-adjunction
\[
\begin{tikzpicture}[baseline=-\the\dimexpr\fontdimen22\textfont2\relax ]
      \node (L) at (0,0)  {$\Enrich{\cat{M}}$};
      \node (R) at (4.5,0)  {$\Enrich{\cat{N}}$.};
      \draw[->,transform canvas={yshift=5pt}]  (L) to node[auto,labelsize] 
      {$\Enrich{L}$} (R);
      \draw[<-,transform canvas={yshift=-5pt}]  (L) to 
      node[auto,swap,labelsize] {$\Enrich{R}$} 
      (R);
      \node[rotate=-90,labelsize] at (2.25,0)  {$\dashv$};
\end{tikzpicture} 
\] 
Therefore, if we take 
$\monoid{T}\in\Th{\cat{M}}\subseteq\Enrich{\cat{M}}$
and $(\cat{C},\enrich{-}{-})\in \Enrich{\cat{N}}$ in this situation, 
then we get an isomorphism of categories
\begin{equation}
\label{eqn:iso_adj_enrich}
\Enrich{\cat{M}}(\monoid{T},\Enrich{R}(\cat{C},\enrich{-}{-}))
\cong\Enrich{\cat{N}}(\Enrich{L}(\monoid{T}),(\cat{C},\enrich{-}{-})).
\end{equation}
Since the action of $\Enrich{-}$ preserves the underlying categories,
we may regard $\Enrich{L}(\monoid{T})$ as a theory in $\cat{N}$.
Therefore 
(\ref{eqn:iso_adj_enrich}) may be seen as an isomorphism between
the category of models of $\monoid{T}$ in $\cat{C}$ with 
respect to $R_\ast(\enrich{-}{-})$,
and the category of models of $\Enrich{L}(\monoid{T})$
in $\cat{C}$ with respect to $\enrich{-}{-}$.

The relationship between clones, symmetric operads and non-symmetric operads
mentioned above can be explained in this way.
First note that there is a chain of inclusions
\[
\begin{tikzpicture}[baseline=-\the\dimexpr\fontdimen22\textfont2\relax ]
      \node (L) at (0,0)  {$\Ncat$};
      \node (M) at (2,0)  {$\Pcat$};
      \node (R) at (4,0)  {$\F$.};
      \draw[->]  (L) to node[auto,labelsize] {$J$} (M);
      \draw[->]  (M) to node[auto,labelsize] {$J'$} (R);
\end{tikzpicture} 
\]
Therefore, precomposition and left Kan extensions 
induce a chain of adjunctions
\[
\begin{tikzpicture}[baseline=-\the\dimexpr\fontdimen22\textfont2\relax ]
      \node (L) at (0,0)  {$[\Ncat,\Set]$};
      \node (M) at (3.5,0)  {$[\Pcat,\Set]$};
      \node (R) at (7,0)  {$[\F,\Set]$.};
      \draw[->,transform canvas={yshift=5pt}]  (L) to node[auto,labelsize] 
      {$\Lan_J$} (M);
      \draw[<-,transform canvas={yshift=-5pt}] (L) to node[auto,swap,labelsize] 
      {$[J,\Set]$} (M);
      \draw[->,transform canvas={yshift=5pt}]  (M) to node[auto,labelsize] 
      {$\Lan_{J'}$} (R);
      \draw[<-,transform canvas={yshift=-5pt}] (M) to node[auto,swap,labelsize] 
      {$[J',\Set]$} 
      (R);
      \node[rotate=-90,labelsize] at (1.75,0)  {$\dashv$};
      \node[rotate=-90,labelsize] at (5.25,0)  {$\dashv$};
\end{tikzpicture} 
\] 
It turns out that these adjunctions acquire natural structures of 
monoidal adjunctions
(with respect to the substitution monoidal structures).
Hence in our framework, the inclusions 
(\ref{eqn:chain_alg_thy}) are expressed as the functors
\[
\begin{tikzpicture}[baseline=-\the\dimexpr\fontdimen22\textfont2\relax ]
      \node (L) at (0,0)  {$\Th{[\Ncat,\Set]}$};
      \node (M) at (4.5,0)  {$\Th{[\Pcat,\Set]}$};
      \node (R) at (9,0)  {$\Th{[\F,\Set]}$};
      \draw[->]  (L) to node[auto,labelsize] {$\Th{\Lan_J}$} (M);
      \draw[->]  (M) to node[auto,labelsize] {$\Th{\Lan_{J'}}$} (R);
\end{tikzpicture} 
\]
between the categories of theories,
whereas the inclusions (\ref{eqn:chain_semantics}) are restrictions of 
the 2-functors 
\[
\begin{tikzpicture}[baseline=-\the\dimexpr\fontdimen22\textfont2\relax ]
      \node (L) at (0,0)  {$\Enrich{[\Ncat,\Set]}$};
      \node (M) at (5,0)  {$\Enrich{[\Pcat,\Set]}$};
      \node (R) at (10,0)  {$\Enrich{[\F,\Set]}$};
      \draw[<-]  (L) to node[auto,labelsize] {$\Enrich{[J,\Set]}$} (M);
      \draw[<-]  (M) to node[auto,labelsize] {$\Enrich{[{J'},\Set]}$} (R);
\end{tikzpicture} 
\]
between the 2-categories of enrichments.

\section{Notions of model as oplax actions}
\label{sec:oplax_action}
In order to capture the standard notions of model for monads and generalised 
operads, 
enrichments do not suffice in general.
A suitable structure is \emph{oplax action}, defined as follows.

\begin{definition}
\label{def:oplax_action}
Let $\cat{M}=(\cat{M},I,\otimes)$ be a metatheory.
An \defemph{oplax action of $\cat{M}$} consists of:
\begin{itemize}
\item a large category $\cat{C}$;
\item a functor $\ast \colon \cat{M}\times\cat{C}\longrightarrow\cat{C}$;
\item a natural transformation $(\varepsilon_C\colon I\ast C\longrightarrow 
C)_{C\in\cat{C}}$;
\item a natural transformation 
$(\delta_{X,Y,C}\colon (Y\otimes X)\ast C\longrightarrow Y\ast (X\ast 
C))_{X,Y\in\cat{M},C\in\cat{C}},$\footnote{We have chosen to set 
$\delta_{X,Y,C}\colon (Y\otimes X)\ast C
\longrightarrow Y\ast (X\ast C)$ and not 
$\delta_{X,Y,C}\colon (X\otimes Y)\ast C
\longrightarrow X\ast (Y\ast C)$,
because the former agrees with 
the convention to write composition of morphisms in the
anti-diagrammatic order, which we adopt in
this paper.}
\end{itemize} 
making the following diagrams commute for all $X,Y,Z\in\cat{M}$
and $C\in\cat{C}$:
\begin{equation*}
\begin{tikzpicture}[baseline=-\the\dimexpr\fontdimen22\textfont2\relax ]
      \node (TL) at (0,1)  {$(I\otimes X)\ast C$};
      \node (TR) at (4,1)  {$I\ast(X\ast C)$};
      \node (BR) at (4,-1) {$X\ast C$};
      \draw[->] (TL) to node[auto,labelsize](T) {$\delta_{X,I,C}$} (TR);
      \draw[->] (TR) to node[auto,labelsize](R) {$\varepsilon_{X\ast C}$} (BR);
      \draw[->]  (TL) to node[auto,labelsize,swap] (L) {$\cong$} (BR);
\end{tikzpicture} 
\quad
\begin{tikzpicture}[baseline=-\the\dimexpr\fontdimen22\textfont2\relax ]
      \node (TL) at (0,1)  {$(X\otimes I)\ast C$};
      \node (TR) at (4,1)  {$X\ast(I\ast C)$};
      \node (BR) at (4,-1) {$X\ast C$};
      \draw[->] (TL) to node[auto,labelsize](T) {$\delta_{I,X,C}$} (TR);
      \draw[->] (TR) to node[auto,labelsize](R) {$X\ast \varepsilon_{C}$} (BR);
      \draw[->]  (TL) to node[auto,labelsize,swap] (L) {$\cong$} (BR);
\end{tikzpicture} 
\end{equation*}
\begin{equation*}
\begin{tikzpicture}[baseline=-\the\dimexpr\fontdimen22\textfont2\relax ]
      \node (TL) at (0,1)  {$((Z\otimes Y)\otimes X)\ast C$};
      \node (TR) at (9,1)  {$(Z\otimes Y)\ast (X\ast C)$};
      \node (BL) at (0,-1) {$(Z\otimes (Y\otimes X))\ast C$};
      \node (BM) at (4.5,-1) {$Z\ast ((Y\otimes X)\ast C)$};
      \node (BR) at (9,-1) {$Z\ast (Y\ast (X\ast C)).$};
      \draw[->] (TL) to node[auto,labelsize](T) {$\delta_{X,Z\otimes Y,C}$} 
      (TR);
      \draw[->]  (TR) to node[auto,labelsize] {$\delta_{Y,Z,X\ast C}$} (BR);
      \draw[->]  (TL) to node[auto,swap,labelsize] {$\cong$} (BL);
      \draw[->]  (BL) to node[auto,labelsize] {$\delta_{Y\otimes X,Z,C}$} (BM);
      \draw[->] (BM) to node[auto,labelsize](B) {$Z\ast \delta_{X,Y,C}$} (BR);
\end{tikzpicture} 
\end{equation*}
We say that $(\cat{C},\ast,\varepsilon,\delta)$ is 
an oplax action of $\cat{M}$, or that $(\ast,\varepsilon,\delta)$ is 
an oplax action of $\cat{M}$ on $\cat{C}$.
\end{definition}

An oplax action $(\ast,\varepsilon,\delta)$ of $\cat{M}$
on $\cat{C}$ is 
called a \defemph{pseudo action}
(resp.~\defemph{strict action})
if both $\varepsilon$ and $\delta$ are natural isomorphisms
(resp.~identities).

The definition of model we derive from an oplax action is 
the following (cf.~\cite[Section~2.2]{Baez_Dolan_HDA3}).

\begin{definition}
\label{def:action_model}
Let $\cat{M}=(\cat{M},I,\otimes)$ be a metatheory,
$\monoid{T}=(T,e,m)$ be a theory in $\cat{M}$,
$\cat{C}$ be a large category, and $\ast=(\ast,\varepsilon,\delta)$
be an oplax action of $\cat{M}$ on $\cat{C}$.
\begin{enumerate}
\item A \defemph{model of $\monoid{T}$ in $\cat{C}$ with respect to $\ast$}
is a pair $(C,\gamma)$ consisting of an object $C\in\cat{C}$
and a morphism $\gamma\colon T\ast C\longrightarrow C$ in $\cat{C}$
making the following diagrams commute:
\begin{equation*}
\begin{tikzpicture}[baseline=-\the\dimexpr\fontdimen22\textfont2\relax ]
      \node (TL) at (0,2)  {$I\ast C$};
      \node (TR) at (2,2)  {$T\ast C$};
      \node (BR) at (2,0) {${C}$};
      \draw[->]  (TL) to node[auto,labelsize] {$e\ast C$} (TR);
      \draw[->]  (TR) to node[auto,labelsize] {$\gamma$} (BR);
      \draw[->] (TL) to node[auto,swap,labelsize](B) {$\varepsilon_C$} (BR);
\end{tikzpicture} 
\qquad
\begin{tikzpicture}[baseline=-\the\dimexpr\fontdimen22\textfont2\relax ]
      \node (TL) at (0,2)  {$(T\otimes T)\ast C$};
      \node (TR) at (5,2)  {$T\ast C$};
      \node (BL) at (0,0) {$T\ast (T\ast C)$};
      \node (BM) at (3,0) {$T\ast C$};
      \node (BR) at (5,0) {${C}.$};
      \draw[->]  (TL) to node[auto,swap,labelsize] {$\delta_{T,T,C}$} (BL);
      \draw[->]  (TL) to node[auto,labelsize] {$m\ast C$} (TR);
      \draw[->]  (BL) to node[auto,labelsize] {$T\ast \gamma$} (BM);
      \draw[->]  (BM) to node[auto,labelsize] {$\gamma$} (BR);
      \draw[->]  (TR) to node[auto,labelsize](B) {$\gamma$} (BR);
\end{tikzpicture} 
\end{equation*}
\item Let $(C,\gamma)$ and $(C',\gamma')$ be models of $\monoid{T}$ in $\cat{C}$
with respect to $\ast$. A \defemph{homomorphism from $(C,\gamma)$ to 
$(C',\gamma')$}
is a morphism $f\colon C\longrightarrow C'$ in $\cat{C}$ making the 
following diagram commute:
\[
\begin{tikzpicture}[baseline=-\the\dimexpr\fontdimen22\textfont2\relax ]
      \node (TL) at (0,2)  {$T\ast C$};
      \node (TR) at (3,2)  {$T\ast C'$};
      \node (BL) at (0,0) {$C$};
      \node (BR) at (3,0) {$C'.$};
      \draw[->]  (TL) to node[auto,swap,labelsize] {$\gamma$} (BL);
      \draw[->]  (TL) to node[auto,labelsize] {$T\ast f$} (TR);
      \draw[->]  (BL) to node[auto,labelsize] {${f}$} (BR);
      \draw[->]  (TR) to node[auto,labelsize](B) {$\gamma'$} (BR);
\end{tikzpicture} 
\] 
\end{enumerate}
We denote the (large) category of models of $\monoid{T}$ in $\cat{C}$ with 
respect to $\ast$ by $\Mod{\monoid{T}}{(\cat{C},\ast)}$.
\end{definition}

\begin{example}
\label{ex:monad_standard_action}
Let $\cat{C}$ be a large category.
Recall that the metatheory for monads on $\cat{C}$ is
$[\cat{C},\cat{C}]$.
We have a strict action 
\[
\ast\colon [\cat{C},\cat{C}]\times\cat{C}\longrightarrow\cat{C}
\]
given by evaluation: $(X,C)\longmapsto XC$.

This clearly generates the definitions of 
Eilenberg--Moore algebra and 
homomorphism (Definition~\ref{def:EM_alg}).
\end{example}

\begin{example}
\label{ex:metamodel_S-operad}
Let $\cat{C}$ be a large category with finite limits and 
$\monoid{S}=(S,\eta,\mu)$
be a cartesian monad on $\cat{C}$. 
Recall that the metatheory for $\monoid{S}$-operads is 
$\cat{C}/S1$.
Models of an $\monoid{S}$-operad and their homomorphisms
(Definition~\ref{def:S-operad_model}) were
introduced by using the pseudo action
\[
\ast\colon(\cat{C}/S1)\times \cat{C}\longrightarrow\cat{C}
\]
in the first place, and therefore are an instance of the above
general definitions.
\end{example}

\subsection{The 2-category of oplax actions of $\cat{M}$}
For a metatheory $\cat{M}$, 
we can define the 2-category of oplax actions of $\cat{M}$
(cf.~Definition~\ref{def:2-cat_of_enrichments}).

\begin{definition}
Let $\cat{M}=(\cat{M},I,\otimes)$ be a metatheory.
The (locally large) 2-category $\olAct{\cat{M}}$ of oplax actions of $\cat{M}$
is defined as follows:
\begin{itemize}
\item An object is an oplax action $(\cat{C},\ast,\varepsilon,\delta)$ of 
$\cat{M}$.
\item A 1-cell from $(\cat{C},\ast,\varepsilon,\delta)$ to 
$(\cat{C'},\ast',\varepsilon',\delta')$ is a functor $G\colon \cat{C}
\longrightarrow\cat{C'}$ together with a natural transformation
$(g_{X,C}\colon X\ast' GC\longrightarrow G(X\ast C))_{X\in\cat{M},C\in\cat{C}}$
making the following diagrams commute for all $X,Y\in\cat{M}$ and $C\in\cat{C}$:
\begin{equation*}
\begin{tikzpicture}[baseline=-\the\dimexpr\fontdimen22\textfont2\relax ]
      \node (TL) at (0,1)  {$I\ast' GC$};
      \node (TR) at (4,1)  {$G(I\ast C)$};
      \node (BR) at (4,-1) {$GC$};
      \draw[->] (TL) to node[auto,labelsize](T) {$g_{I,C}$} (TR);
      \draw[->] (TR) to node[auto,labelsize](R) {$G\varepsilon_C$} (BR);
      \draw[->]  (TL) to node[auto,labelsize,swap] (L) {$\varepsilon'_{GC}$} 
      (BR);
\end{tikzpicture} 
\end{equation*}
\begin{equation*}
\begin{tikzpicture}[baseline=-\the\dimexpr\fontdimen22\textfont2\relax ]
      \node (TL) at (0,1)  {$(Y\otimes X)\ast' GC$};
      \node (TR) at (8,1)  {$G((Y\otimes X)\ast C)$};
      \node (BL) at (0,-1) {$Y\ast'(X\ast'GC)$};
      \node (BM) at (4,-1) {$Y\ast'G(X\ast C)$};
      \node (BR) at (8,-1) {$G(Y\ast (X\ast C)).$};
      \draw[->] (TL) to node[auto,labelsize](T) {$g_{Y\otimes X,C}$} 
      (TR);
      \draw[->]  (TR) to node[auto,labelsize] {$G\delta_{X,Y,C}$} (BR);
      \draw[->]  (TL) to node[auto,swap,labelsize] {$\delta'_{X,Y,GC}$} (BL);
      \draw[->]  (BL) to node[auto,labelsize] {$Y\ast' g_{X,C}$} (BM);
      \draw[->] (BM) to node[auto,labelsize](B) {$g_{Y,X\ast C}$} (BR);
\end{tikzpicture} 
\end{equation*}
\item A 2-cell from $(G,g)$ to $(G',g')$, both from 
$(\cat{C},\ast,\varepsilon,\delta)$ to 
$(\cat{C'},\ast',\varepsilon',\delta')$,
is a natural transformation $\theta\colon G\Longrightarrow G'$
making the following diagram commute for all $X\in\cat{M}$
and $C\in\cat{C}$:
\begin{equation*}
\begin{tikzpicture}[baseline=-\the\dimexpr\fontdimen22\textfont2\relax ]
      \node (TL) at (0,2)  {$X\ast' GC$};
      \node (TR) at (4,2)  {$G(X\ast C)$};
      \node (BL) at (0,0) {$X\ast' G'C$};
      \node (BR) at (4,0) {$G'(X\ast C).$};
      \draw[->] (TL) to node[auto,labelsize](T) {$g_{X,C}$} 
      (TR);
      \draw[->]  (TR) to node[auto,labelsize] {$\theta_{X\ast C}$} (BR);
      \draw[->]  (TL) to node[auto,swap,labelsize] {$X\ast' \theta_C$} 
      (BL);
      \draw[->]  (BL) to node[auto,labelsize] {$g'_{X,C}$} (BR);
\end{tikzpicture} 
\end{equation*}
\end{itemize}
\end{definition}

Similarly to the case of enrichments, we may
extend the $\Mod{-}{-}$ construction into a 2-functor
\[
\Mod{-}{-}\colon \Th{\cat{M}}^\op\times \olAct{\cat{M}}
\longrightarrow\tCAT.
\]

On the other hand, $\olAct{-}$ extends to a 2-functor 
in an apparently different manner than $\Enrich{-}$.
Namely, it is a 2-functor of type
\[
\olAct{-}\colon (\MonCATol)^\coop\longrightarrow \twoCAT,
\]
where $\MonCATol$ is the 2-category of large monoidal categories
(= metatheories),
oplax monoidal functors and monoidal natural transformations.
The apparent discrepancy between functoriality of $\olAct{-}$
and $\Enrich{-}$
will be solved in Section~\ref{sec:morphism_metatheory}
with the introduction of morphisms of metatheories.

\subsection{The relation between enrichments and oplax actions}
\label{subsec:enrich_action_adjunction}
We have introduced two types of structures---enrichment and oplax action---to
formalise notions of model. 
The former captures the standard notions of model 
for clones, symmetric operads, non-symmetric operads, PROPs and PROs,
whereas the latter captures those for
monads and generalised operads.
We will unify enrichment and oplax action by  
the notion of \emph{metamodel} in the next section,
but before doing so we remark on the relationship 
between them.
We also explain why in some good cases we can give definition of model 
both in terms of enrichment and oplax actions;
for instances of this phenomenon in the literature,
see e.g., \cite[Section~3]{Kelly_coherence_lax_dist} and 
\cite[Section~6.4]{Leinster_book}.


\medskip

Let $\cat{M}=(\cat{M},I,\otimes)$ be a metatheory and 
$\cat{C}$ be a large category.
The relationship between enrichment and oplax action is summarised in
the adjunction
\begin{equation}
\label{eqn:adj_act_enrich}
\begin{tikzpicture}[baseline=-\the\dimexpr\fontdimen22\textfont2\relax ]
      \node (L) at (0,0)  {$\cat{M}$};
      \node (R) at (3,0)  {$\cat{C}$.};
      \draw[->,transform canvas={yshift=5pt}]  (L) to node[auto,labelsize] 
      {$-\ast C$} (R);
      \draw[<-,transform canvas={yshift=-5pt}]  (L) to 
      node[auto,swap,labelsize] 
      {$\enrich{C}{-}$} (R);
      \node[rotate=-90,labelsize] at (1.5,0)  {$\dashv$};
\end{tikzpicture} 
\end{equation}
In more detail, what we mean is the following.
Suppose that we have an enrichment $(\enrich{-}{-},j,M)$
of $\cat{C}$ over $\cat{M}$.
If, in addition, for each $C\in\cat{C}$ the functor $\enrich{C}{-}$
has a left adjoint as in (\ref{eqn:adj_act_enrich}),
then---by the parameter theorem 
for adjunctions; see \cite[Section~IV.7]{MacLane_CWM}---the left adjoints
canonically extend to a 
bifunctor $\ast\colon \cat{M}\times\cat{C}\longrightarrow\cat{C}$,
and $j$ and $M$ define appropriate natural transformations $\varepsilon$
and $\delta$, giving rise to an oplax action $(\ast,\varepsilon,\delta)$
of $\cat{M}$ on $\cat{C}$.
And vice versa, if we start from an oplax action.

To make this idea into a precise mathematical statement, let us introduce the 
following 2-categories.
\begin{definition}
\label{def:enrichr_olactl}
Let $\cat{M}$ be a metatheory.
\begin{enumerate}
\item Let $\Enrichr{\cat{M}}$ be the full sub-2-category of $\Enrich{\cat{M}}$
consisting of all enrichments $(\cat{C},\enrich{-}{-},j,M)$ such that 
for each $C\in\cat{C}$, $\enrich{C}{-}$ is a right adjoint.
\item Let $\olActl{\cat{M}}$ be the full sub-2-category of $\olAct{\cat{M}}$
consisting of all oplax actions $(\cat{C},\ast,\varepsilon,\delta)$
such that for each $C\in\cat{C}$, $-\ast C$ is a left adjoint. 
\end{enumerate}
\end{definition}
The above discussion can be summarised in the
statement that the two 2-categories $\Enrichr{\cat{M}}$ and $\olActl{\cat{M}}$
are equivalent.
A direct proof of this equivalence would be essentially routine, 
but seems to involve rather lengthy calculation.
We shall defer a proof to Corollary~\ref{cor:enrichr_olactl_equiv}.

This observation, also mentioned in \cite{Campbell_skew} in a more general setting, is a variant of well-known categorical folklore.
In the literature, it is usually stated in a slightly more restricted form than 
the above, for example as a correspondence between tensored $\cat{M}$-categories
and closed pseudo actions of 
$\cat{M}$~\cite{Kelly_coherence_lax_dist,Gordon_Power,Lindner_enrich_module,Janelidze_Kelly}.
 
\medskip

Furthermore, the above correspondence is
compatible with the definitions of model (Definitions~\ref{def:enrich_model}
and \ref{def:action_model}).
Suppose that $(\cat{C},\enrich{-}{-},j,M)$ and 
$(\cat{C},\ast,\varepsilon,\delta)$
form a pair of an enrichment over $\cat{M}$ and an oplax action of $\cat{M}$
connected by the adjunctions (\ref{eqn:adj_act_enrich})
(in a way compatible with the natural transformations
$j,M,\varepsilon$ and $\delta$).
Then for any theory $\monoid{T}=(T,e,m)$ in $\cat{M}$ and any
object $C\in\cat{C}$,
a morphism 
\[
\chi\colon T\longrightarrow \enrich{C}{C}
\]
is a model of $\monoid{T}$ in $\cat{C}$ with respect to $\enrich{-}{-}$
(Definition~\ref{def:enrich_model})
if and only if its transpose under the adjunction $-\ast C\dashv \enrich{C}{-}$
\[
\gamma\colon T\ast C\longrightarrow C
\]
is a model of $\monoid{T}$ in $\cat{C}$ with respect to $\ast$
(Definition~\ref{def:action_model}),
and similarly for homomorphism between models of $\monoid{T}$.
Hence we obtain an isomorphism of categories 
\[
\Mod{\monoid{T}}{(\cat{C},\enrich{-}{-})}\cong\Mod{\monoid{T}}{(\cat{C},\ast)}
\]
commuting with the forgetful functors into $\cat{C}$.

Some of the enrichments and oplax actions we have introduced so far
are good enough to obtain the corresponding oplax actions or enrichments,
giving rise to alternative definitions of model.

\begin{example}
\label{ex:monad_on_cat_w_powers}
Let $\cat{C}$ be a locally small category with all small powers.
Recall the strict action
\[
\ast\colon[\cat{C},\cat{C}]\times\cat{C}\longrightarrow\cat{C}
\]
of the metatheory $[\cat{C},\cat{C}]$ for monads
on $\cat{C}$, used to capture their Eilenberg--Moore algebras.
For any object $C\in\cat{C}$, write by $\name{C}\colon 1\longrightarrow\cat{C}$
the functor from the terminal category $1$ which maps the unique object of $1$
to $C\in\cat{C}$ (the \emph{name of $C$}).

By the assumptions on $\cat{C}$,
for any object $A\in\cat{C}$ the functor $-\ast A$ (which may be seen as
the precomposition by $\name{A}\colon 1\longrightarrow\cat{C}$) admits
a right adjoint $\enrich{A}{-}$, which maps any $B\in\cat{C}$ (equivalently,
$\name{B}\colon 1\longrightarrow\cat{C}$) to the right Kan extension 
$\enrich{A}{B}=\Ran_{\name{A}}\name{B}$ of $\name{B}$ along $\name{A}$.
The functor $\Ran_{\name{A}}\name{B}\colon\cat{C}\longrightarrow\cat{C}$
maps $C\in\cat{C}$ to $\Ran_{\name{A}}\name{B}(C)=B^{\cat{C}(C,A)}$.

For any object $C\in\cat{C}$, $\Ran_{\name{C}}\name{C}$
exists and acquires a canonical monad structure (the \emph{codensity 
monad of $\name{C}$}).
For any monad $\monoid{T}$ on $\cat{C}$, to give a structure of an 
Eilenberg--Moore algebra on $C\in\cat{C}$ is equivalent to 
give a monad morphism from $\monoid{T}$ to $\Ran_{\name{C}}\name{C}$.
This observation is in e.g., \cite[Section~3]{Kelly_coherence_lax_dist}.

In particular, if we take $\cat{C}=\Set$,
we see that the above enrichment agrees with the one given 
in Example~\ref{ex:relative_alg}.
Hence the notion of relative algebra \cite{Hino} agrees with
that of Eilenberg--Moore algebra in this case. 
\end{example}

\section{Metamodels and models}
\label{sec:metamodel_model}

In Sections~\ref{sec:enrichment} and \ref{sec:oplax_action},
we have seen that the standard notions of model for various notions of 
algebraic theory can be formalised either as enrichments or
as oplax actions. 
With \emph{two} definitions, however, we cannot claim to have 
formalised notions of model in a sufficiently satisfactory manner.
We now unify enrichments and oplax actions by introducing
a more general structure of 
\emph{metamodel} (of a metatheory).
We also derive definitions of models of theories and of their homomorphisms
from a metamodel, and show that they generalise the corresponding 
definitions for enrichments and oplax actions.

\medskip

We may approach 
the concept of metamodel of a metatheory $\cat{M}$
in two different ways, one by generalising
enrichments over $\cat{M}$ and the other 
by generalising oplax actions of $\cat{M}$.
Before giving a formal (and neutral) definition of metamodel, we describe
these two perspectives.

\subsection{Metamodels as generalised enrichments}
Let us first discuss how a generalisation of enrichments over $\cat{M}$
leads to the notion of metamodel.
For this, we use the construction known as the 
\emph{Day convolution}~\cite{Day_thesis}.
Given any metatheory $\cat{M}=(\cat{M},I,\otimes)$,
this construction endows the presheaf category 
$\widehat{\cat{M}}=[\cat{M}^\op,\SET]$
with a (biclosed) monoidal structure $(\DayI,\Dayo)$,
in such a way that the Yoneda embedding 
$\cat{M}\longrightarrow\widehat{\cat{M}}$ canonically becomes strong monoidal.

\begin{definition}[\cite{Day_thesis}]
\label{def:Day_conv}
Let $\cat{M}=(\cat{M},I,\otimes)$ be a metatheory.
The \defemph{convolution monoidal structure} 
$(\DayI,\Dayo)$ on the presheaf category $\widehat{\cat{M}}
=[\cat{M}^\op,\SET]$ is defined as follows.
\begin{itemize}
\item The unit object $\DayI$ is the representable functor 
$\cat{M}(-,I)\colon \cat{M}^\op\longrightarrow\SET$.
\item Given $P,Q\in\widehat{\cat{M}}$, their monoidal product
$Q\Dayo P\colon \cat{M}^\op\longrightarrow\SET$ maps $Z\in\cat{M}$
to 
\begin{equation*}
(Q\Dayo P)(Z)=\int^{X,Y\in\cat{M}} \cat{M}(Z,Y\otimes X)
\times Q(Y)\times P(X).
\end{equation*}
\end{itemize}
\end{definition}

For a metatheory $\cat{M}$, 
a \emph{metamodel of $\cat{M}$} is simply an enrichment over 
$\widehat{\cat{M}}=(\widehat{\cat{M}},\DayI,\Dayo)$.\footnote{Although 
we have defined enrichment (Definition~\ref{def:enrichment})
only for metatheories, i.e., large monoidal categories,
the definition does not depend on any size condition
and it is clear what we mean by enrichments over non-large monoidal categories, 
such as $\widehat{\cat{M}}$.}
Thanks to the strong monoidal Yoneda embedding, 
every enrichment over $\cat{M}$ induces a metamodel of $\cat{M}$.

We can find several uses of $\widehat{\cat{M}}$-categories 
(in the sense of enriched category theory) in the literature.
In particular, \cite[Section~6]{Kelly_et_al_two_sides} and
\cite{Mellies_enriched_adjunction} contain 
discussions on relationship between $\widehat{\cat{M}}$-categories
and various actions of $\cat{M}$.

\subsection{Metamodels as generalised oplax actions}
Let us move on to the second perspective on metamodels,
namely as generalised oplax actions.
First note that an oplax action $(\cat{C},\ast,\varepsilon,\delta)$ 
of a metatheory $\cat{M}$
can be equivalently given as an oplax monoidal functor 
\begin{equation*}
\cat{M}\longrightarrow[\cat{C},\cat{C}]
\end{equation*}
defined by $X\longmapsto X\ast-$,
or as a colax functor
\begin{equation}
\label{eqn:oplax_action_as_oplax_functor}
\Sigma\cat{M}\longrightarrow\tCAT,
\end{equation}
where $\Sigma\cat{M}$ denotes $\cat{M}$ seen as a one-object 
bicategory~\cite{Benabou_bicat}.

To generalise this, 
we use the bicategory $\PROF$ of {profunctors}
(Definition~\ref{def:profunctor}).
It is well-known that 
both $\tCAT$ and $\tCAT^\coop$ canonically
embed into $\PROF$.
Both embeddings are identity-on-objects and locally fully faithful 
pseudofunctors.
The embedding
\[
(-)_\ast\colon\tCAT\longrightarrow\PROF
\]
maps a functor $F\colon\cat{A}\longrightarrow\cat{B}$
to the profunctor $F_\ast\colon\cat{A}\pto\cat{B}$ defined by 
$F_\ast(B,A)=\cat{B}(B,FA)$, whereas
the embedding
\[
(-)^\ast\colon \tCAT^\coop\longrightarrow\PROF
\]
maps a functor $F\colon\cat{A}\longrightarrow\cat{B}$ to 
the profunctor $F^\ast\colon\cat{B}\pto\cat{A}$
with $F^\ast(A,B)=\cat{B}(FA,B)$.
Moreover, for any functor $F\colon\cat{A}\longrightarrow\cat{B}$
we have an adjunction $F_\ast\dashv F^\ast$ in $\PROF$.

A \emph{metamodel of $\cat{M}$} is a colax functor 
\[
\Sigma\cat{M}\longrightarrow \PROF^\coop,
\]
or equivalently a lax functor 
\begin{equation}
\label{eqn:metamodel_as_lax_functor}
(\Sigma\cat{M})^\co=\Sigma(\cat{M}^\op)\longrightarrow \PROF^\op.
\end{equation}
Clearly, oplax actions of $\cat{M}$, in the form
(\ref{eqn:oplax_action_as_oplax_functor}),
give rise to metamodels of $\cat{M}$ by postcomposing the pseudofunctor 
$(-)^\ast$.

Let us restate what a lax functor of type (\ref{eqn:metamodel_as_lax_functor})
amounts to, in monoidal categorical terms. 
\begin{definition}
Let $\cat{C}$ be a large category. 
Define the monoidal category $[\cat{C}^\op\times \cat{C},\SET]
=([\cat{C}^\op\times\cat{C},\SET],\cat{C}(-,-),\ptensorrev)$ of
\defemph{endo-profunctors on $\cat{C}$}
to be the endo-hom-category $\PROF^\op(\cat{C},\cat{C})$.
More precisely:
\begin{itemize}
\item The unit object is the hom-functor
$\cat{C}(-,-)\colon \cat{C}^\op\times\cat{C}\longrightarrow\SET$.
\item Given $H,K\colon\cat{C}^\op\times\cat{C}\longrightarrow\SET$,
their monoidal product $H\ptensorrev K$
maps $(A,C)\in\cat{C}^\op\times\cat{C}$ to
\[
(H\ptensorrev K)(A,C)=\int^{B\in\cat{C}}H(B,C)\times K(A,B). 
\]
\end{itemize}
\end{definition}
Note that $H\ptensorrev K\cong K\ptensor H$ (i.e., $\ptensorrev$ is 
``$\ptensor$ reversed'').

Using this monoidal structure on $[\cat{C}^\op\times\cat{C},\SET]$, a 
metamodel of $\cat{M}$ in a large category $\cat{C}$
may be written as a lax monoidal functor
\[
\cat{M}^\op\longrightarrow[\cat{C}^\op\times\cat{C},\SET].
\]

\medskip

We now give the neutral definition of metamodels.

\begin{definition}[{Cf.~\cite[Definition 4.1]{Campbell_skew}}]
\label{def:metamodel}
Let $\cat{M}=(\cat{M},\otimes, I)$ be a metatheory.
A \defemph{metamodel of $\cat{M}$} consists of:
\begin{itemize}
\item a large category $\cat{C}$;
\item a functor 
$\Phi\colon \cat{M}^\op\times\cat{C}^\op\times\cat{C}\longrightarrow\SET$
(whose action we write as $(X,A,B)\longmapsto \Phi_X(A,B)$);
\item a natural transformation $((\overline{\phi}_\cdot)_C\colon 
1\longrightarrow\Phi_I(C,C))_{C\in\cat{C}}$;
\item a natural transformation 
\[
((\overline{\phi}_{X,Y})_{A,B,C}\colon 
\Phi_Y(B,C)\times\Phi_X(A,B)
\longrightarrow\Phi_{Y\otimes X}(A,C))_{X,Y\in\cat{M},A,B,C\in\cat{C}},
\]
\end{itemize}
making the following diagrams commute 
for all $X,Y,Z\in\cat{M}$ and $A,B,C,D\in\cat{C}$:
\begin{equation*}
\begin{tikzpicture}[baseline=-\the\dimexpr\fontdimen22\textfont2\relax ]
      \node (TL) at (0,1)  {$1\times  \Phi_X(A,B)$};
      \node (TR) at (6,1)  {$\Phi_I(B,B)\times \Phi_X(A,B)$};
      \node (BL) at (0,-1) {$\Phi_X(A,B)$};
      \node (BR) at (6,-1) {$\Phi_{I\otimes X}(A,B)$};
      \draw[->] (TL) to node[auto,labelsize](T) 
      {$(\overline{\phi}_\cdot)_B\times \Phi_X(A,B)$} 
      (TR);
      \draw[->]  (TR) to node[auto,labelsize] 
      {$(\overline{\phi}_{X,I})_{A,B,B}$} (BR);
      \draw[->]  (TL) to node[auto,swap,labelsize] {$\cong$} 
      (BL);
      \draw[->]  (BL) to node[auto,labelsize] {$\cong$} (BR);
\end{tikzpicture} 
\end{equation*}
\begin{equation*}
\begin{tikzpicture}[baseline=-\the\dimexpr\fontdimen22\textfont2\relax ]
      \node (TL) at (0,1)  {$\Phi_X(A,B)\times 1$};
      \node (TR) at (6,1)  {$\Phi_X(A,B)\times\Phi_I(A,A)$};
      \node (BL) at (0,-1) {$\Phi_X(A,B)$};
      \node (BR) at (6,-1) {$\Phi_{X\otimes I}(A,B)$};
      \draw[->] (TL) to node[auto,labelsize](T) 
      {$\Phi_X(A,B)\times(\overline{\phi}_\cdot)_A$} 
      (TR);
      \draw[->]  (TR) to node[auto,labelsize] 
      {$(\overline{\phi}_{I,X})_{A,A,B}$} (BR);
      \draw[->]  (TL) to node[auto,swap,labelsize] {$\cong$} 
      (BL);
      \draw[->]  (BL) to node[auto,labelsize] {$\cong$} (BR);
\end{tikzpicture}
\end{equation*}
\begin{equation*}
\begin{tikzpicture}[baseline=-\the\dimexpr\fontdimen22\textfont2\relax ]
      \node (TL) at (0,2)  {$\big(\Phi_Z(C,D)\times \Phi_Y(B,C)\big)\times 
            \Phi_X(A,B)$};
      \node (TM) at (8,2)  {$\Phi_{Z\otimes Y}(B,D)\times \Phi_X(A,B)$};
      \node (TR) at (8,0)  {$\Phi_{(Z\otimes Y)\otimes X}(A,D)$};
      \node (BL) at (0,0) {$\Phi_Z(C,D)\times \big(\Phi_Y(B,C)\times 
                  \Phi_X(A,B)\big)$};
      \node (BM) at (0,-2) {$\Phi_Z(C,D)\times\Phi_{Y\otimes X}(A,C)$};
      \node (BR) at (8,-2) {$\Phi_{Z\otimes (Y\otimes X)}(A,D).$};
      \draw[->] (TL) to node[auto,labelsize](T) 
      {$(\overline{\phi}_{Y,Z})_{B,C,D}\times\Phi_X(A,B)$} 
      (TM);
      \draw[->] (TM) to node[auto,labelsize](T) 
      {$(\overline{\phi}_{X,Z\otimes Y})_{A,B,D}$} 
      (TR);
      \draw[->]  (TR) to node[auto,labelsize] {$\cong$} (BR);
      \draw[->]  (TL) to node[auto,swap,labelsize] {$\cong$} (BL);
      \draw[->]  (BL) to node[auto,swap,labelsize] {$\Phi_Z(C,D)\times 
      (\overline{\phi}_{X,Y})_{A,B,C}$} 
      (BM);
      \draw[->] (BM) to node[auto,labelsize](B) {$(\overline{\phi}_{Y\otimes 
      X,Z})_{A,C,D}$} (BR);
\end{tikzpicture} 
\end{equation*}
We say that $(\cat{C},\Phi,\overline{\phi}_\cdot,\overline{\phi})$ is a 
metamodel of $\cat{M}$,
or that $(\Phi,\overline{\phi}_\cdot,\overline{\phi})$ is a 
metamodel of $\cat{M}$ in $\cat{C}$.
\end{definition}

The above definition perfectly makes sense even if we replace 
the category $\SET$ of large sets by the category $\Set$
of small sets. 
Indeed, most of the naturally occurring notions of model 
can be captured by such ``$\Set$-valued'' metamodels.
However, for theoretical development it turns out to be more 
convenient to define metamodels as above. 

Note that
we may replace $((\overline{\phi}_\cdot)_C)_{C\in\cat{C}}$ by 
\[
((j_C)_Z\colon \DayI(Z)\longrightarrow\Phi_{Z}(C,C))_{C\in\cat{C},Z\in\cat{M}}
\] 
and $((\overline{\phi}_{X,Y})_{A,B,C})_{X,Y\in\cat{M},A,B,C\in\cat{C}}$ by 
\[
((M_{A,B,C})_Z\colon (\Phi_{(-)}(B,C)\Dayo \Phi_{(-)}(A,B))(Z)
\longrightarrow \Phi_Z(A,C))_{A,B,C\in\cat{M},Z\in\cat{M}}.
\]
The axioms for metamodel then translate to the ones for 
enrichments (over $\widehat{\cat{M}}$).

On the other hand, we may also replace 
$((\overline{\phi}_\cdot)_C)_{C\in\cat{C}}$ by
\[
((\phi_\cdot)_{A,B}\colon\cat{C}(A,B)\longrightarrow\Phi_I(A,B))_{A,B\in\cat{C}}
\]
and $((\overline{\phi}_{X,Y})_{A,B,C})_{X,Y\in\cat{M},A,B,C\in\cat{C}}$ by
\[
((\phi_{X,Y})_{A,C}\colon (\Phi_Y\ptensorrev\Phi_X)(A,C)\longrightarrow
\Phi_{Y\otimes X}(A,C))_{A,C\in\cat{C},X,Y\in\cat{M}}.
\]
The axioms for metamodel then state that 
\[
(\Phi,\phi_\cdot,\phi)\colon(\cat{M}^\op,I,\otimes)\longrightarrow 
([\cat{C}^\op\times\cat{C},\SET],\cat{C}(-,-),\ptensorrev)
\]
is an oplax monoidal functor.

Hence the attempts to generalise enrichments and oplax actions mentioned above
coincide and both give rise to Definition~\ref{def:metamodel}.

\medskip

The definitions of model and homomorphism we derive from a metamodel 
are the following.
\begin{definition}
\label{def:metamodel_model}
Let $\cat{M}=(\cat{M},I,\otimes)$ be a metatheory, 
$\monoid{T}=(T,e,m)$ be a theory in $\cat{M}$, 
$\cat{C}$ be a large category and 
$\Phi=(\Phi,\overline{\phi}_\cdot,\overline{\phi})$ 
be a metamodel of $\cat{M}$ in $\cat{C}$.
\begin{enumerate}
\item A \defemph{model of $\monoid{T}$ in $\cat{C}$ with respect to $\Phi$}
is a pair $(C,\xi)$ consisting of an object $C$ of $\cat{C}$
and an element $\xi\in \Phi_T(C,C)$ such that 
$(\Phi_e)_{C,C}(\xi)=(\overline{\phi}_\cdot)_{C}(\ast)$ (where $\ast$
is the unique element of $1$)
and $(\Phi_m)_{C,C}(\xi)=(\overline{\phi}_{T,T})_{C,C,C}(\xi,\xi)$:
\[
\begin{tikzpicture}[baseline=-\the\dimexpr\fontdimen22\textfont2\relax ]
      \node(11) at (0,1.5) {$\Phi_T(C,C)$};
      \node(12) at (2,0) {$\Phi_I(C,C)$};
      \node(21) at (4,1.5) {$1$};
      
      \draw [->]  (11) to node [auto,swap,labelsize]{$(\Phi_e)_{C,C}$} (12);
      \draw [->]  (21) to node [auto,labelsize]{$(\overline{\phi}_\cdot)_{C}$} 
      (12);
\end{tikzpicture}
\quad
\begin{tikzpicture}[baseline=-\the\dimexpr\fontdimen22\textfont2\relax ]
      \node(11) at (0,1.5) {$\Phi_T(C,C)$};
      \node(12) at (2,0) {$\Phi_{T\otimes T}(C,C).$};
      \node(21) at (4,1.5) {$\Phi_T(C,C)\times \Phi_T(C,C)$};
      
      \draw [->]  (11) to node [auto,swap,labelsize]{$(\Phi_m)_{C,C}$} (12);
      \draw [->]  (21) to node 
      [auto,labelsize]{$(\overline{\phi}_{T,T})_{C,C,C}$} (12);
\end{tikzpicture}
\] 
\item Let $(C,\xi)$ and $(C',\xi')$ be models of $\monoid{T}$ in $\cat{C}$
with respect to $\Phi$. A \defemph{homomorphism from $(C,\xi)$ to $(C',\xi')$}
is a morphism $f\colon C\longrightarrow C'$ in $\cat{C}$ such that
$\Phi_T(C,f)(\xi)=\Phi_T(f,C')(\xi')$:
\[
\begin{tikzpicture}[baseline=-\the\dimexpr\fontdimen22\textfont2\relax ]
      \node(11) at (0,1.5) {$\Phi_T(C,C)$};
      \node(12) at (2,0) {$\Phi_T(C,C').$};
      \node(21) at (4,1.5) {$\Phi_T(C',C')$};
      
      \draw [->]  (11) to node [auto,swap,labelsize]{$\Phi_T(C,f)$} (12);
      \draw [->]  (21) to node [auto,labelsize]{$\Phi_T(f,C')$} (12);
\end{tikzpicture}
\] 
\end{enumerate}
We denote the (large) category of models of $\monoid{T}$ in $\cat{C}$ with 
respect to 
$\Phi$ by $\Mod{\monoid{T}}{(\cat{C},\Phi)}$.
\end{definition}

\begin{example}
\label{ex:enrichment_as_metamodel}
Let $\cat{M}=(\cat{M},I,\otimes)$ be a metatheory,
$\cat{C}$ be a large category and 
$(\enrich{-}{-},j,M)$ be an enrichment of $\cat{C}$ over $\cat{M}$.
This induces a metamodel $(\Phi,\overline{\phi}_\cdot,\overline{\phi})$ 
of $\cat{M}$ in $\cat{C}$ as follows.
\begin{itemize}
\item The functor 
$\Phi\colon\cat{M}^\op\times\cat{C}^\op\times\cat{C}\longrightarrow\SET$
maps $(X,A,B)\in \cat{M}^\op\times\cat{C}^\op\times\cat{C}$
to 
\[
\Phi_X(A,B)=\cat{M}(X,\enrich{A}{B}).
\]
\item For each $C\in\cat{C}$, $(\overline{\phi}_\cdot)_C\colon 
1\longrightarrow\Phi_I(C,C)$
is the name of $j_C$
(i.e., $(\overline{\phi}_\cdot)_C$ maps the unique element of the singleton 
$1$  to ${j_C}$).
\item For each $A,B,C\in\cat{C}$ and $X,Y\in\cat{M}$, the function
$(\overline{\phi}_{X,Y})_{A,B,C}\colon 
\Phi_Y(B,C)\times\Phi_X(A,B)
\longrightarrow\Phi_{Y\otimes X}(A,C)$
maps $g\colon Y\longrightarrow\enrich{B}{C}$ and $f\colon 
X\longrightarrow\enrich{A}{B}$ to 
\[
\begin{tikzpicture}[baseline=-\the\dimexpr\fontdimen22\textfont2\relax ]
      \node (L) at (0,0)  {$Y\otimes X$};
      \node (M) at (3,0)  {$\enrich{B}{C}\otimes \enrich{A}{B}$};
      \node (R) at (6,0) {$\enrich{A}{C}.$};
      \draw[->] (L) to node[auto,labelsize](T) {$g\otimes f$} 
      (M);
      \draw[->]  (M) to node[auto,labelsize] {$M_{A,B,C}$} (R);
\end{tikzpicture}
\]
\end{itemize}

The definition of model and homomorphism (Definition~\ref{def:enrich_model})
we derive from an enrichment may be seen as 
a special case of the 
corresponding definition (Definition~\ref{def:metamodel_model})
for metamodel.
\end{example}

\begin{example}
\label{ex:oplax_action_as_metamodel}
Let $\cat{M}=(\cat{M},I,\otimes)$ be a metatheory,
$\cat{C}$ be a large category and 
$(\ast,\varepsilon,\delta)$ be an oplax action of $\cat{M}$ on $\cat{C}$.
This induces a metamodel $(\Phi,\overline{\phi}_\cdot,\overline{\phi})$ 
of $\cat{M}$ in $\cat{C}$ as follows.
\begin{itemize}
\item The functor 
$\Phi\colon\cat{M}^\op\times\cat{C}^\op\times\cat{C}\longrightarrow\SET$
maps $(X,A,B)\in \cat{M}^\op\times\cat{C}^\op\times\cat{C}$
to 
\[
\Phi_X(A,B)=\cat{C}(X\ast A,{B}).
\]
\item For each $C\in\cat{C}$, $(\overline{\phi}_\cdot)_C\colon 
1\longrightarrow\Phi_I(C,C)$
is the name of $\varepsilon_C$.
\item For each $A,B,C\in\cat{C}$ and $X,Y\in\cat{M}$, the function
$(\overline{\phi}_{X,Y})_{A,B,C}\colon 
\Phi_Y(B,C)\times\Phi_X(A,B)
\longrightarrow\Phi_{Y\otimes X}(A,C)$
maps $g\colon Y\ast{B}\longrightarrow{C}$ and $f\colon 
X\ast{A}\longrightarrow{B}$ to 
\[
\begin{tikzpicture}[baseline=-\the\dimexpr\fontdimen22\textfont2\relax ]
      \node (L) at (-0.5,0)  {$(Y\otimes X)\ast A$};
      \node (M) at (3,0)  {$Y\ast(X\ast A)$};
      \node (R) at (6,0) {$Y\ast B$};
      \node (RR) at (8,0) {$C.$};
      \draw[->] (L) to node[auto,labelsize](T) {$\delta_{X,Y,A}$} 
      (M);
      \draw[->]  (M) to node[auto,labelsize] {$Y\ast f$} (R);
      \draw[->]  (R) to node[auto,labelsize] {$g$} (RR);
\end{tikzpicture}
\]
\end{itemize}

The definition of model and homomorphism 
(Definition~\ref{def:action_model})
we derive from an oplax action may be seen as a special case of the
corresponding definition (Definition~\ref{def:metamodel_model})
for metamodel.
\end{example}

\subsection{The 2-category of metamodels}
Metamodels of a metatheory naturally form a 2-category,
just as enrichments and oplax actions do.

\begin{definition}[{Cf.~\cite[Definitions 4.10 and 4.11]{Campbell_skew}}]
Let $\cat{M}=(\cat{M},I,\otimes)$ be a metatheory.
We define the (locally large) 2-category $\MtMod{\cat{M}}$ of 
metamodels of $\cat{M}$ as follows.
\begin{itemize}
\item An object is a metamodel 
$(\cat{C},\Phi,\overline{\phi}_\cdot,\overline{\phi})$ of $\cat{M}$.
\item A 1-cell from $(\cat{C},\Phi,\overline{\phi}_\cdot,\overline{\phi})$ to 
$(\cat{C'},\Phi',\overline{\phi'}_\cdot,\overline{\phi'})$
is a functor $G\colon \cat{C}\longrightarrow \cat{C'}$
together with a natural transformation $(g_{X,A,B}\colon 
\Phi_X(A,B)\longrightarrow \Phi'_X(GA,GB))_{X\in\cat{M},A,B\in\cat{C}}$ 
making the following 
diagrams commute for all 
$X,Y\in\cat{M}$ and $A,B,C\in\cat{C}$:
\begin{equation*}
\begin{tikzpicture}[baseline=-\the\dimexpr\fontdimen22\textfont2\relax ]
      \node (TL) at (0,1)  {$1$};
      \node (TR) at (4,1)  {$\Phi_I(C,C)$};
      \node (BR) at (4,-1) {$\Phi'_I(GC,GC)$};
      \draw[->] (TL) to node[auto,labelsize](T) {$(\overline{\phi}_\cdot)_{C}$} 
      (TR);
      \draw[->]  (TR) to node[auto,labelsize] {$g_{I,C,C}$} (BR);
      \draw[->]  (TL) to node[auto,swap,labelsize] 
      {$(\overline{\phi'}_\cdot)_{GC}$} 
      (BR);
\end{tikzpicture} 
\end{equation*}
\begin{equation*}
\begin{tikzpicture}[baseline=-\the\dimexpr\fontdimen22\textfont2\relax ]
      \node (TL) at (0,1)  {$\Phi_Y(B,C)\times \Phi_X(A,B)$};
      \node (TR) at (8,1)  {$\Phi_{Y\otimes X}(A,C)$};
      \node (BL) at (0,-1) {$\Phi'_Y(GB,GC)\times \Phi'_X(GA,GB)$};
      \node (BR) at (8,-1) {$\Phi'_{Y\otimes X}(GA,GC).$};
      \draw[->] (TL) to node[auto,labelsize](T) 
      {$(\overline{\phi}_{X,Y})_{A,B,C}$} 
      (TR);
      \draw[->]  (TR) to node[auto,labelsize] {$g_{Y\otimes X,A,C}$} (BR);
      \draw[->]  (TL) to node[auto,swap,labelsize] {$g_{Y,B,C}\times 
      g_{X,A,B}$} 
      (BL);
      \draw[->]  (BL) to node[auto,labelsize] 
      {$(\overline{\phi'}_{X,Y})_{GA,GB,GC}$} (BR);
\end{tikzpicture} 
\end{equation*}
\item A 2-cell from $(G,g)$ to $(G',g')$, both from 
$(\cat{C},\Phi,\overline{\phi}_\cdot,\overline{\phi})$ to 
$(\cat{C'},\Phi',\overline{\phi'}_\cdot,\overline{\phi'})$,
is a natural transformation $\theta\colon G\Longrightarrow G'$
making the following diagram commute for all $X\in\cat{M}$
and $A,B\in\cat{C}$:
\begin{equation*}
\begin{tikzpicture}[baseline=-\the\dimexpr\fontdimen22\textfont2\relax ]
      \node (TL) at (0,2)  {$\Phi_X(A,B)$};
      \node (TR) at (6,2)  {$\Phi'_X(GA,GB)$};
      \node (BL) at (0,0) {$\Phi'_X(G'A,G'B)$};
      \node (BR) at (6,0) {$\Phi'_X(GA,G'B).$};
      \draw[->] (TL) to node[auto,labelsize](T) {$g_{X,A,B}$} 
      (TR);
      \draw[->]  (TR) to node[auto,labelsize] {$\Phi'_X(GA,\theta_B)$} (BR);
      \draw[->]  (TL) to node[auto,swap,labelsize] {$g'_{X,A,B}$} 
      (BL);
      \draw[->]  (BL) to node[auto,labelsize] {$\Phi'_X(\theta_A,G'B)$} (BR);
\end{tikzpicture}  
\end{equation*}
\end{itemize}
\end{definition}

Recall that for a functor (resp.~a 2-functor) 
$F\colon \cat{A}\longrightarrow\cat{B}$,
the \defemph{essential image of $F$} is the full subcategory
(resp.~full sub-2-category) of $\cat{B}$
consisting of all objects $B\in\cat{B}$ such that there exists an object 
$A\in\cat{A}$ and an isomorphism $FA\cong B$.
If $\cat{A}$ is a large category, a contravariant presheaf
$\cat{A}^\op\longrightarrow\SET$ 
(resp.~a covariant presheaf $\cat{A}\longrightarrow\SET$) 
over $\cat{A}$ is 
called \defemph{representable} if and only if it is 
in the essential image of the Yoneda embedding 
$\cat{A}\longrightarrow[\cat{A}^\op,\SET]$
(resp.~$\cat{A}\longrightarrow [\cat{A},\SET]^\op$).

\begin{proposition}[{Cf.~\cite[Proposition 4.12 (b)]{Campbell_skew}}]
\label{prop:enrichment_as_metamodel}
Let $\cat{M}$ be a metatheory.
The construction given in Example~\ref{ex:enrichment_as_metamodel}
canonically extends to a fully faithful 2-functor
\[
\Enrich{\cat{M}}\longrightarrow\MtMod{\cat{M}}.
\]
A metamodel $(\cat{C},\Phi,\overline{\phi}_\cdot,\overline{\phi})$
of $\cat{M}$ is in the essential image of this 2-functor if and only if 
for each $A,B\in\cat{C}$, the functor 
\[
\Phi_{(-)}(A,B)\colon\cat{M}^\op\longrightarrow\SET
\]
is representable.
\end{proposition}
\begin{proof}
The construction of the 
2-functor $\Enrich{\cat{M}}\longrightarrow\MtMod{\cat{M}}$
is straightforward. 
The rest can also be proved by a standard argument using the Yoneda lemma.
We sketch the argument below.

Let us focus on the characterisation of the essential image.
Suppose that $(\cat{C},\Phi,\overline{\phi}_\cdot,\overline{\phi})$ is a 
metamodel of $\cat{M}$ such that
for each $A,B\in\cat{C}$, the functor $\Phi_{(-)}(A,B)$
is representable.
From such a metamodel we obtain an enrichment $(\enrich{-}{-},j,M)$ of 
$\cat{C}$ over $\cat{M}$ as follows. 
For each $A,B\in\cat{C}$, choose an object $\enrich{A}{B}\in\cat{M}$ and an 
isomorphism 
$\alpha_{A,B}\colon\cat{M}(-,\enrich{A}{B})\longrightarrow\Phi_{(-)}(A,B)$.
By functoriality of $\Phi$, $\enrich{-}{-}$ uniquely extends to a 
functor of type 
$\cat{C}^\op\times\cat{C}\longrightarrow\cat{M}$ while making 
$(\alpha_{A,B})_{A,B\in\cat{C}}$ natural.
For each $C\in\cat{C}$, $(\overline{\phi}_\cdot)_C\colon 1\longrightarrow
\Phi_I(C,C)\cong \cat{M}(I,\enrich{C}{C})$ gives rise to a 
morphism $j_C\colon I\longrightarrow\enrich{C}{C}$ in $\cat{M}$. 
For each $A,B,C\in\cat{M}$, consider the function
\[
\begin{tikzpicture}[baseline=-\the\dimexpr\fontdimen22\textfont2\relax ]
      \node (TT) at (0,4.5)  
      {$\cat{M}(\enrich{B}{C},\enrich{B}{C})\times\cat{M}(\enrich{A}{B},
      \enrich{A}{B})$};
      \node (T) at (0,3)  
      {$\Phi_{\enrich{B}{C}}(B,C)\times\Phi_\enrich{A}{B}(A,B)$};
      \node (L) at (0,1.5)  {$\Phi_{\enrich{B}{C}\otimes\enrich{A}{B}}(A,C)$};
      \node (LL) at (0,0) {$\cat{M}(\enrich{B}{C}\otimes\enrich{A}{B}, 
      \enrich{A}{C}).$};
      \draw[->] (T) to node[auto,labelsize]
      {$(\overline{\phi}_{\enrich{A}{B},\enrich{B}{C}})_{A,B,C}$} 
      (L);
      \draw[->] (TT) to node[auto,labelsize]{$(\alpha_{B,C})_\enrich{B}{C}
      \times(\alpha_{A,B})_{\enrich{A}{B}}$} (T);
      \draw[->] (L) to 
      node[auto,labelsize]{$(\alpha_{A,C})_{\enrich{B}{C}
      \otimes\enrich{A}{B}}^{-1}$}
       (LL);
\end{tikzpicture}
\]
Let the image of $(\id{\enrich{B}{C}},\id{\enrich{A}{B}})$ 
under this function be $M_{A,B,C}\colon 
\enrich{B}{C}\otimes\enrich{A}{B}\longrightarrow\enrich{A}{C}$.
The axioms of metamodel then shows that $(\enrich{-}{-},j,M)$ is an enrichment.

Moreover, if we consider the metamodel induced from this enrichment 
(see Example~\ref{ex:enrichment_as_metamodel}), then it is 
isomorphic to our original 
$(\cat{C},\Phi,\overline{\phi}_\cdot,\overline{\phi})$.
In particular, for each $X,Y\in\cat{M}$ and $A,B,C\in\cat{C}$, 
the function $(\overline{\phi}_{X,Y})_{A,B,C}$ is completely
determined by $M_{A,B,C}$, as in 
Example~\ref{ex:enrichment_as_metamodel}.
To see this, note that for each $f\in\cat{M}(X,\enrich{A}{B})$
and $g\in\cat{M}(Y,\enrich{B}{C})$, the diagram 
\[
\begin{tikzpicture}[baseline=-\the\dimexpr\fontdimen22\textfont2\relax ]
      \node (TT1) at (7.5,4.5)  
      {$\cat{M}(Y,\enrich{B}{C})\times\cat{M}(X,\enrich{A}{B})$};
      \node (T1) at (7.5,3)  
      {$\Phi_Y(B,C)\times\Phi_X(A,B)$};
      \node (L1) at (7.5,1.5)  
      {$\Phi_{Y\otimes X}(A,C)$};
      \node (LL1) at (7.5,0) {$\cat{M}(\enrich{B}{C}\otimes\enrich{A}{B}, 
      \enrich{A}{C})$};
      \node (TT) at (0,4.5)  
      {$\cat{M}(\enrich{B}{C},\enrich{B}{C})\times\cat{M}(\enrich{A}{B},
      \enrich{A}{B})$};
      \node (T) at (0,3)  
      {$\Phi_{\enrich{B}{C}}(B,C)\times\Phi_\enrich{A}{B}(A,B)$};
      \node (L) at (0,1.5)  {$\Phi_{\enrich{B}{C}\otimes\enrich{A}{B}}(A,C)$};
      \node (LL) at (0,0) {$\cat{M}(\enrich{B}{C}\otimes\enrich{A}{B}, 
      \enrich{A}{C})$};
      \draw[->] (T) to node[auto,swap,labelsize]
      {$(\overline{\phi}_{\enrich{A}{B},\enrich{B}{C}})_{A,B,C}$} 
      (L);
      \draw[->] (TT) to 
      node[auto,swap,labelsize]{$\alpha
      \times\alpha$} (T);
      \draw[->] (L) to 
      node[auto,swap,labelsize]{$\alpha^{-1}$}
       (LL);
      \draw[->] (T1) to node[auto,labelsize]
      {$(\overline{\phi}_{X,Y})_{A,B,C}$} 
      (L1);
      \draw[->] (TT1) to 
      node[auto,labelsize]{$\alpha \times\alpha$} (T1);
      \draw[->] (L1) to 
      node[auto,labelsize]{$\alpha^{-1}$}
       (LL1);
      \draw[->] (TT) to node[auto,labelsize]
      {$\cat{M}(g,\id{})\times\cat{M}(f,\id{})$} 
      (TT1);
      \draw[->] (T) to node[auto,labelsize]
      {$\Phi_g(B,C)\times\Phi_f(A,B)$} 
      (T1);
      \draw[->] (L) to node[auto,labelsize]
      {$\Phi_{g\otimes f}(A,C)$} 
      (L1);
      \draw[->] (LL) to node[auto,labelsize]
      {$\cat{M}(g\otimes f,\id{})$} 
      (LL1);
\end{tikzpicture}
\]
commutes. Hence by chasing the element $(\id{\enrich{B}{C}},\id{\enrich{A}{B}})$
in the top left set, we observe that (modulo the isomorphisms $\alpha$)
$(g,f)$ is mapped by $(\overline{\phi}_{X,Y})_{A,B,C}$ to 
$M_{A,B,C}\circ (g\otimes f)$.
\end{proof}

\begin{proposition}[{Cf.~\cite[Proposition 4.2 (a)]{Campbell_skew}}]
\label{prop:oplax_action_as_metamodel}
Let $\cat{M}$ be a metatheory.
The construction given in Example~\ref{ex:oplax_action_as_metamodel}
canonically extends to a fully faithful 2-functor
\[
\olAct{\cat{M}}\longrightarrow\MtMod{\cat{M}}.
\]
A metamodel $(\cat{C},\Phi,\overline{\phi}_\cdot,\overline{\phi})$ of $\cat{M}$ 
is in the essential image of this 2-functor if and only if for each
$X\in\cat{M}$ and $A\in\cat{C}$, the functor 
\[
\Phi_X(A,-)\colon \cat{C}\longrightarrow\SET
\]
is representable.
\end{proposition}
\begin{proof}
Similar to the proof of Proposition~\ref{prop:enrichment_as_metamodel}.

In particular, given a metamodel 
$(\cat{C},\Phi,\overline{\phi}_\cdot,\overline{\phi})$ of $\cat{M}$ such that
for each $X\in\cat{M}$ and $A\in\cat{C}$, the functor $\Phi_X(A,-)$ is 
representable, we may construct an oplax action $(\ast,\varepsilon,\delta)$ of 
$\cat{C}$ as follows.
For each $X\in\cat{M}$ and $A\in\cat{C}$, choose an object $X\ast A\in\cat{C}$
and an isomorphism $\beta_{X,A}\colon \cat{C}(X\ast A,-) 
\longrightarrow\Phi_X(A,-)$.
We easily obtain a functor 
$\ast\colon\cat{M}\times\cat{C}\longrightarrow\cat{C}$
and a natural transformation $(\varepsilon_C)_{C\in\cat{C}}$. 
To get $\delta$, for each $X,Y\in\cat{M}$ and $A\in\cat{C}$ consider 
the function
\[
\begin{tikzpicture}[baseline=-\the\dimexpr\fontdimen22\textfont2\relax ]
      \node (TT) at (0,4.5)  
      {$\cat{C}(Y\ast (X\ast A),Y\ast (X\ast A))\times\cat{C}
      (X\ast A,X\ast A)$};
      \node (T) at (0,3)  
      {$\Phi_{Y}(X\ast A,Y\ast (X\ast A))\times\Phi_X(A,X\ast A)$};
      \node (L) at (0,1.5)  {$\Phi_{Y\otimes X}(A,Y\ast (X\ast A))$};
      \node (LL) at (0,0) {$\cat{C}((Y\otimes X)\ast A, 
      Y\ast (X\ast A)).$};
      \draw[->] (T) to node[auto,labelsize]
      {$(\overline{\phi}_{X,Y})_{A,X\ast A,Y\ast (X\ast A)}$}
      (L);
      \draw[->] (TT) to node[auto,labelsize]{$(\beta_{Y,X\ast A})_{Y\ast (X\ast 
      A)}
      \times(\beta_{X,A})_{X\ast A}$} (T);
      \draw[->] (L) to 
      node[auto,labelsize]{$(\beta_{Y\otimes X,A})_{Y\ast(X\ast A)}^{-1}$}
       (LL);
\end{tikzpicture}
\]
We define $\delta_{X,Y,A}\colon (Y\otimes X)\ast A\longrightarrow Y\ast(X\ast 
A)$ to be the image of $(\id{Y\ast (X\ast A)},\id{X\ast A})$ under this 
function.

To verify that the metamodel induced from this oplax action (see 
Example~\ref{ex:oplax_action_as_metamodel}) is isomorphic to 
$(\cat{C},\Phi,\overline{\phi}_\cdot,\overline{\phi})$,
essentially we only need to check that $(\overline{\phi}_{X,Y})_{A,B,C}$
for each $X,Y\in\cat{M}$ and $A,B,C\in\cat{C}$ is determined by
$\delta_{X,Y,A}$ as in Example~\ref{ex:oplax_action_as_metamodel}.
Suppressing the isomorphisms $\beta$ from now on,
for each $f\colon X\ast A\longrightarrow B$ and $g\colon Y\ast B\longrightarrow 
C$ consider the following digram:
\[
\begin{tikzpicture}[baseline=-\the\dimexpr\fontdimen22\textfont2\relax ]
      \node (TT1) at (7.5,4)  
      {$\cat{C}((Y\otimes X)\ast A,Y\ast (X\ast A))$};
      \node (T1) at (7.5,2)  
      {$\cat{C}((Y\otimes X)\ast A,C)$};
      \node (L1) at (7.5,0)  
      {$\cat{C}(Y\ast B,C)\times \cat{C}(X\ast A,B).$};
      \node (TT) at (0,4)  
      {$\cat{C}(Y\ast (X\ast A),Y\ast (X\ast A)) \times 
      \cat{C}(X\ast A,X\ast A)$};
      \node (T) at (0,2)  
      {$\cat{C}(Y\ast (X\ast A),C)\times \cat{C}(X\ast A,X\ast A)$};
      \node (L) at (0,0)  {$\cat{C}(Y\ast B,C)\times \cat{C}(X\ast A,X\ast A)$};
      \draw[<-] (T) to node[auto,swap,labelsize]
      {$\cat{C}(Y\ast f,\id{})\times \cat{C}(\id{},\id{})$} 
      (L);
      \draw[->] (TT) to 
      node[auto,swap,labelsize]{$\cat{C}(\id{},g\circ(Y\ast f))\times 
      \cat{C}(\id{},\id{})$} 
      (T);
      \draw[<-] (T1) to node[auto,labelsize]
      {$(\overline{\phi}_{X,Y})_{A,B,C}$} 
      (L1);
      \draw[->] (TT1) to 
      node[auto,labelsize]{$\cat{C}(\id{},g\circ(Y\ast f))$} (T1);
      \draw[->] (TT) to node[auto,labelsize]
      {$(\overline{\phi}_{X,Y})_{A,X\ast A,Y\ast(X\ast A)}$} 
      (TT1);
      \draw[->] (T) to node[auto,labelsize]
      {$(\overline{\phi}_{X,Y})_{A,X\ast A,C}$} 
      (T1);
      \draw[->] (L) to node[auto,labelsize]
      {$\cat{C}(\id{},\id{})\times \cat{C}(\id{},f)$} 
      (L1);
\end{tikzpicture}
\]
The top square commutes by naturality in $C$ of 
$((\overline{\phi}_{X,Y})_{A,B,C})$ and the bottom square commutes by 
(extra) naturality in $B$ of it.
By chasing the appropriate elements as follows
\[
\begin{tikzpicture}[baseline=-\the\dimexpr\fontdimen22\textfont2\relax ]
      \node (TT1) at (6.5,3)  
      {$\delta_{X,Y,A}$};
      \node (T1) at (6.5,1.5)  
      {$g\circ(Y\ast f)\circ \delta_{X,Y,A}$};
      \node (L1) at (6.5,0)  
      {$(g,f),$};
      \node (TT) at (0,3)  
      {$(\id{Y\ast (X\ast A)},\id{X\ast A})$};
      \node (T) at (0,1.5)  
      {$(g\circ (Y\ast f), \id{X\ast A})$};
      \node (L) at (0,0)  {$(g,\id{X\ast A})$};
      \draw[<-|] (T) to node[auto,swap,labelsize]
      {$\cat{C}(Y\ast f,\id{})\times \cat{C}(\id{},\id{})$} 
      (L);
      \draw[|->] (TT) to 
      node[auto,swap,labelsize]{$\cat{C}(\id{},g\circ(Y\ast f))\times 
      \cat{C}(\id{},\id{})$} 
      (T);
      \draw[<-|] (T1) to node[auto,labelsize]
      {$(\overline{\phi}_{X,Y})_{A,B,C}$} 
      (L1);
      \draw[|->] (TT1) to 
      node[auto,labelsize]{$\cat{C}(\id{},g\circ(Y\ast f))$} (T1);
      \draw[|->] (TT) to node[auto,labelsize]
      {$(\overline{\phi}_{X,Y})_{A,X\ast A,Y\ast(X\ast A)}$} 
      (TT1);
      \draw[|->] (T) to node[auto,labelsize]
      {$(\overline{\phi}_{X,Y})_{A,X\ast A,C}$} 
      (T1);
      \draw[|->] (L) to node[auto,labelsize]
      {$\cat{C}(\id{},\id{})\times \cat{C}(\id{},f)$} 
      (L1);
\end{tikzpicture}
\]
we conclude that $(\overline{\phi}_{X,Y})_{A,B,C}(g,f)=g\circ(Y\ast 
f)\circ\delta_{X,Y,A}$, as desired.
\end{proof}

Recall the 2-categories $\Enrichr{\cat{M}}$ and $\olActl{\cat{M}}$
defined in Definition~\ref{def:enrichr_olactl}.

\begin{corollary}[{Cf.~\cite[Corollary 4.13]{Campbell_skew}}]
\label{cor:enrichr_olactl_equiv}
Let $\cat{M}$ be a metatheory.
\begin{enumerate}
\item The 2-functors in Proposition~\ref{prop:enrichment_as_metamodel}
and Proposition~\ref{prop:oplax_action_as_metamodel} restrict to 
fully faithful 2-functors 
\[
\Enrichr{\cat{M}}\longrightarrow\MtMod{\cat{M}}\qquad
\olActl{\cat{M}}\longrightarrow\MtMod{\cat{M}}
\]
with the same essential image characterised as follows:
a metamodel $(\cat{C},\Phi,\overline{\phi}_\cdot,\overline{\phi})$ 
of $\cat{M}$ is in the essential image if and only if
for each $X\in\cat{M}$ and $A,B\in\cat{C}$,
the functors 
\[
\Phi_{(-)}(A,B)\colon\cat{M}^\op\longrightarrow\SET\qquad
\Phi_X(A,-)\colon\cat{C}\longrightarrow\SET
\] 
are representable.
\item The 2-categories $\Enrichr{\cat{M}}$ and $\olActl{\cat{M}}$ are 
equivalent.
\end{enumerate}
\end{corollary}
\begin{proof}
The first clause is immediate from the definition of adjunction.
For instance, an enrichment $(\cat{C},\enrich{-}{-},j,M)$ over $\cat{M}$
is in $\Enrichr{\cat{M}}$ if and only if
for each $A\in\cat{C}$, $\enrich{A}{-}$ is a right adjoint,
which in turn is the case if and only if
for each $X\in\cat{M}$ and $A\in\cat{C}$, the functor 
\[
\cat{M}(X,\enrich{A}{-})\colon\cat{C}\longrightarrow\SET
\]
is representable.

The second clause is a direct consequence of the first.
\end{proof}

The reader might have noticed that 
there is another representability condition not covered by 
Propositions~\ref{prop:enrichment_as_metamodel} and 
\ref{prop:oplax_action_as_metamodel}, namely
metamodels $(\cat{C},\Phi,\overline{\phi}_\cdot,\overline{\phi})$
such that for each $X\in\cat{M}$ and $B\in\cat{C}$,
the functor 
\[
\Phi_X(-,B)\colon \cat{C}^\op\longrightarrow\SET
\] 
is representable. They correspond to 
{right lax actions of $\cat{M}^\op$ on $\cat{C}$},
or equivalently, to right oplax actions of $\cat{M}$ on $\cat{C}^\op$.

Extending the definition of enrichment (Definition~\ref{def:enrichment})
and the 2-category of enrichments (Definition~\ref{def:2-cat_of_enrichments})
to huge monoidal categories, we obtain the following.
\begin{proposition}
\label{prop:metatheory_as_enrichment}
Let $\cat{M}$ be a metatheory
and $\widehat{\cat{M}}=([\cat{M}^\op,\SET],\widehat{I},\widehat{\otimes})$
(see Definition~\ref{def:Day_conv}).
The 2-categories $\MtMod{\cat{M}}$ and $\Enrich{\widehat{\cat{M}}}$
are canonically isomorphic.
\end{proposition}

\subsection{$\Mod{-}{-}$ as a 2-functor}
Let $\cat{M}$ be a metatheory.
Similarly to the cases of enrichments and oplax actions,
we can view the $\Mod{-}{-}$ construction as a 2-functor
using the 2-category $\MtMod{\cat{M}}$.
In fact, via the inclusion
\begin{equation}
\label{eqn:theory_as_metamodel}
\Th{\cat{M}}\longrightarrow
\Enrich{\cat{M}}\longrightarrow\MtMod{\cat{M}},
\end{equation}
the 2-functor $\Mod{-}{-}$ is simply given by the following composite:
\[
\begin{tikzpicture}[baseline=-\the\dimexpr\fontdimen22\textfont2\relax ]
      \node (1) at (0,0)  {$\Th{\cat{M}}^\op\times \MtMod{\cat{M}}$};
      \node (2) at (0,-1.5)  {$\MtMod{\cat{M}}^\op\times \MtMod{\cat{M}}$};
      \node (3) at (0,-3) {$\tCAT$,};
      \draw[->] (1) to node[auto,labelsize]{inclusion} (2);
      \draw[->] (2) to node[auto,labelsize] {$\MtMod{\cat{M}}(-,-)$} (3);
\end{tikzpicture}
\]
where $\MtMod{\cat{M}}(-,-)$ is the hom-2-functor for 
the locally large $\MtMod{\cat{M}}$.
The inclusion (\ref{eqn:theory_as_metamodel}) identifies
a theory $\monoid{T}=(T,m,e)$ in $\cat{M}$ with the 
metamodel $(\Phi^{(\monoid{T})}, 
\overline{\phi^{(\monoid{T})}}_\cdot,\overline{\phi^{(\monoid{T})}})$
of $\cat{M}$ in
the terminal category $1$ (whose unique object we denote by $\ast$),
defined as follows:
\begin{itemize}
\item the functor $\Phi^{(\monoid{T})}\colon\cat{M}^\op\times 1^\op\times 1
\longrightarrow\SET$ maps $(X,\ast,\ast)$ to $\cat{M}(X,T)$;
\item the function $(\overline{\phi^{(\monoid{T})}}_\cdot)_\ast\colon 
1\longrightarrow \cat{M}(I,T)$ maps the unique element of $1$ to $e$;
\item for each $X,Y\in\cat{M}$, the function 
$(\overline{\phi^{(\monoid{T})}}_{X,Y})_{\ast,\ast,\ast}\colon
\cat{M}(Y,T)\times \cat{M}(X,T)\longrightarrow\cat{M}(Y\otimes X,T)$
maps $(g,f)$ to $m\circ (g\otimes f)$.
\end{itemize}

\section{Morphisms of metatheories}
\label{sec:morphism_metatheory}
In this section, we introduce a notion of morphism between metatheories.
The main purpose of morphisms of metatheories is to provide a 
means to compare different notions of algebraic theory.
An example of such comparison
is given in Section~\ref{sec:enrichment},
where we compare clones, symmetric operads and non-symmetric operads. 
Recall that the crucial observation used there was the fact that 
the $\Enrich{-}$ construction extends to a 2-functor
\begin{equation}
\label{eqn:Enrich_as_2-functor}
\Enrich{-}\colon \MonCAT\longrightarrow\twoCAT.
\end{equation}
Therefore, we want to define morphisms of metatheories with 
respect to which $\MtMod{-}$ behaves (2-)functorially.

On the other hand, recall from Section~\ref{sec:oplax_action}
that $\olAct{-}$ is a 2-functor of type
\begin{equation}
\label{eqn:olAct_as_2-functor}
\olAct{-}\colon(\MonCATol)^\coop\longrightarrow\twoCAT.
\end{equation}
Since metamodels unify both enrichments and oplax actions, 
we would like to explain both (\ref{eqn:Enrich_as_2-functor})
and (\ref{eqn:olAct_as_2-functor}) by introducing a
sufficiently general notion of morphism of metatheories.

The requirement to unify both $\MonCAT$ and $(\MonCATol^\coop)$
suggests the possibility of using a suitable variant of 
profunctors, namely \emph{monoidal profunctors} 
introduced in Definition~\ref{def:monoidal_profunctor}.

\begin{definition}
Let $\cat{M}=(\cat{M},I_\cat{M},\otimes_\cat{M})$ and 
$\cat{N}=(\cat{N},I_\cat{N},\otimes_\cat{N})$ be metatheories. 
A \defemph{morphism of metatheories from $\cat{M}$ to $\cat{N}$} 
is a monoidal profunctor from $\cat{M}$ to $\cat{N}$.
\end{definition}



\begin{definition}
We define the bicategory $\MtTH$ of metatheories as follows.
\begin{itemize}
\item An object is a metatheory.
\item A 1-cell from $\cat{M}$ to $\cat{N}$
is a morphism of metatheories $\cat{M}\pto\cat{N}$.
The identity 1-cell on a metatheory $\cat{M}$ is the hom-functor $\cat{M}(-,-)$,
equipped with the evident structure for a morphism of metatheories.
Given morphisms of metatheories 
$(H,h_\cdot,h)\colon \cat{M}\pto\cat{N}$ and 
$(K,k_\cdot,k)\colon \cat{N}\pto\cat{L}$,
their composite is $(K\ptensor H,k_\cdot\ptensor h_\cdot, k \ptensor h)
\colon \cat{M}\pto\cat{L}$ where $K\ptensor H$ is the composition of the 
profunctors $H$ and $K$ (Definition~\ref{def:profunctor}), and 
$k_\cdot \ptensor h_\cdot$ and $k\ptensor h$ are the evident
natural transformations.
\item A 2-cell from $H$ to $H'$, both from $\cat{M}$ to $\cat{N}$,
is a monoidal natural transformation
$\alpha\colon H\Longrightarrow 
H'\colon\cat{N}^\op\times\cat{M}\longrightarrow\SET$.  
\end{itemize}
\end{definition}

Similarly to the case of profunctors,
we have identity-on-objects fully faithful pseudofunctors 
\[
(-)_\ast\colon\MonCAT\longrightarrow\MtTH
\]
and 
\[
(-)^\ast\colon (\MonCATol)^\coop\longrightarrow\MtTH.
\]
In detail, a lax monoidal functor 
\[
F=(F,f_\cdot,f)\colon \cat{M}\longrightarrow\cat{N}
\]
gives rise to a morphism of metatheories 
\[
F_\ast=(F_\ast,(f_\ast)_\cdot,f_\ast)\colon\cat{M}\pto\cat{N}
\]
with $F_\ast(N,M)=\cat{N}(N,FM)$, $(f_\ast)_\cdot\colon 1\longrightarrow 
\cat{N}(I_\cat{N},FI_\cat{M})$ mapping the unique element of $1$ to 
$f_\cdot\colon I_\cat{N}\longrightarrow FI_\cat{M}$,
and $(f_\ast)_{N,N',M,M'}\colon 
\cat{N}(N',FM')\times\cat{N}(N,FM)\longrightarrow\cat{N}(N'\otimes_\cat{N} N,
F(M'\otimes_\cat{M} M))$
mapping $g'\colon N'\longrightarrow FM'$ and $g\colon N\longrightarrow FM$
to $f_{M,M'}\circ (g'\otimes g)\colon N'\otimes_\cat{N} N\longrightarrow
F(M'\otimes_\cat{M} M)$.
Given an oplax monoidal functor 
\[
F=(F,f_\cdot,f)\colon \cat{M}\longrightarrow\cat{N},
\]
we obtain a morphism of metatheories
\[
F^\ast=(F^\ast, (f^\ast)_\cdot,f^\ast)\colon\cat{N}\pto\cat{M}
\]
analogously.

In particular, a \emph{strong} monoidal functor
\[
F\colon\cat{M}\longrightarrow\cat{N}
\]
gives rise to both $F_\ast\colon \cat{M}\pto\cat{N}$
and $F^\ast\colon\cat{N}\pto\cat{M}$, and 
it is straightforward to see that these form an adjunction $F_\ast\dashv 
F^\ast$ in $\MtTH$.
\medskip

A morphism of metatheories $H\colon\cat{M}\pto\cat{N}$
induces a 2-functor 
\[
\MtMod{H}\colon\MtMod{\cat{M}}\longrightarrow\MtMod{\cat{N}}.
\]
Its action on objects is as follows.
\begin{definition}
\label{def:action_of_morphism_on_metamodel}
Let $\cat{M}=(\cat{M},I_\cat{M},\otimes_\cat{M})$ and 
$\cat{N}=(\cat{N},I_\cat{N},\otimes_\cat{N})$ be metatheories,
$H=(H,h_\cdot,h)\colon\cat{M}\pto\cat{N}$ be a morphism of metatheories, 
$\cat{C}$ be a large category and 
$\Phi=(\Phi,\overline{\phi}_\cdot,\overline{\phi})$ be a metamodel of  
$\cat{M}$ in $\cat{C}$.
We define the metamodel 
$H(\Phi)=(\Phi',\overline{\phi'}_\cdot,\overline{\phi'})$ 
of $\cat{N}$ on $\cat{C}$ as follows:
\begin{itemize}
\item The functor 
$\Phi'\colon\cat{N}^\op\times\cat{C}^\op\times\cat{C}\longrightarrow\SET$
maps $(N,A,B)\in\cat{N}^\op\times\cat{C}^\op\times\cat{C}$ to the set 
\begin{equation}
\label{eqn:action_of_H_on_metamodels}
\Phi'_N(A,B)=
\int^{M\in\cat{M}}H(N,M)\times \Phi_M(A,B).
\end{equation}
\item The natural transformation $((\overline{\phi'}_\cdot)_C\colon 
1\longrightarrow \Phi'_{I_\cat{N}}(C,C))_{C\in\cat{C}}$
is defined by
mapping the unique element $\ast$ of $1$ to
\begin{multline*}
[I_\cat{M}\in\cat{M}, h_\cdot(\ast)\in H(I_\cat{N},I_\cat{M}), 
(\overline{\phi}_\cdot)_C(\ast)\in
\Phi_{I_\cat{M}}(C,C)]\\
\in\int^{M\in\cat{M}}H(I_\cat{N},M)\times \Phi_M(C,C)
\end{multline*}
for each $C\in\cat{C}$. 
\item The natural transformation 
\[((\overline{\phi'}_{N,N'})_{A,B,C}\colon
\Phi'_{N'}(B,C)\times\Phi'_N(A,B)\longrightarrow
\Phi'_{N'\otimes_\cat{N} N}(A,C))_{N,N'\in\cat{N},A,B,C\in\cat{C}}\]
is defined by mapping a pair consisting of
$[M',x',y']\in\Phi'_{N'}(B,C)$ and $[M,x,y]\in\Phi'_N(A,B)$
to 
\begin{equation*}
[M'\otimes_\cat{M}M, h_{N,N',M,M'}(x',x), 
(\overline{\phi}_{M,M'})_{A,B,C}(y',y)]
\end{equation*}
for each $N,N'\in\cat{N}$ and $A,B,C\in\cat{C}$.
\end{itemize}
\end{definition}

The above construction extends routinely, giving rise to a pseudofunctor
\[
\MtMod{-}\colon \MtTH\longrightarrow\twoCAT.
\]

\subsection{Comparing different notions of algebraic theory}
\label{subsec:comparing}
We now demonstrate how we can compare 
different notions of algebraic theory via morphisms of metatheories.

We start with a few remarks on simplification of the action
(Definition~\ref{def:action_of_morphism_on_metamodel}) of a morphism of 
metatheories
on metamodels, in certain special cases.
Let $\cat{M}$ and $\cat{N}$ be metatheories, 
$
H\colon\cat{M}\pto\cat{N}
$
be a morphism of metatheories,
$\cat{C}$ be a large category and 
$\Phi=(\Phi,\overline{\phi}_\cdot,\overline{\phi})$ be a metamodel of $\cat{M}$
in $\cat{C}$. 

First consider the case where for each $A,B\in\cat{C}$,
the functor $\Phi_{(-)}(A,B)\colon \cat{M}^\op\longrightarrow{\SET}$
is representable. This means that $\Phi$ is in fact (up to an isomorphism) an 
enrichment $\enrich{-}{-}$; see Proposition~\ref{prop:enrichment_as_metamodel}.
In this case, $\Phi_M(A,B)$ may be written as $\cat{M}(M,\enrich{A}{B})$
and hence the formula (\ref{eqn:action_of_H_on_metamodels})
simplifies:
\[
\Phi'_N(A,B)=
\int^{M\in\cat{M}}H(N,M)\times \cat{M}(M,\enrich{A}{B})
\cong H(N,\enrich{A}{B}).
\]
In particular, if moreover $H$ is of the form
\[
F_\ast\colon \cat{M}\pto\cat{N}
\]
for some lax monoidal functor $F\colon\cat{M}\longrightarrow\cat{N}$,
then we have 
\[
\Phi'_N(A,B)\cong \cat{N}(N,F\enrich{A}{B}),
\]
implying that $H(\Phi)=F_\ast(\Phi)$ is again isomorphic to
an enrichment; indeed, this case reduces to $F_\ast(\enrich{-}{-})$
defined in Definition~\ref{def:action_of_lax_on_enrichment}.
Note that, as a special case, for any theory $\monoid{T}$ in $\cat{M}$
(recall that such a theory is identified with a metamodel of $\cat{M}$
in the terminal category $1$),
$F_\ast(\monoid{T})$ is again isomorphic to a theory in $\cat{N}$.
Therefore the 2-functor 
\[
\MtMod{F_\ast}\colon\MtMod{\cat{M}}\longrightarrow\MtMod{\cat{N}}
\]
extends the functor 
\[
\Th{F}\colon\Th{\cat{M}}\longrightarrow\Th{\cat{N}}
\]
between the categories of theories induced by $F$, using the well-known fact 
that a lax monoidal functor preserves theories (= monoid objects).

Next consider the case where $H$ is of the form
\[
G^\ast\colon \cat{M}\pto\cat{N}
\]
for some oplax monoidal functor $G\colon\cat{N}\longrightarrow\cat{M}$.
In this case $H(N,M)=\cat{M}(GN,M)$ and the formula 
(\ref{eqn:action_of_H_on_metamodels}) simplifies as follows:
\[
\Phi'_N(A,B)=
\int^{M\in\cat{M}}\cat{M}(GN,M)\times \Phi_M(A,B)
\cong \Phi_{GN}(A,B).
\]

Suppose now that we have a strong monoidal 
functor
\[
F\colon\cat{M}\longrightarrow\cat{N}
\]
between metatheories $\cat{M}$ and $\cat{N}$.
On the one hand, $F$ induces a functor
\[
\Th{F}\colon\Th{\cat{M}}\longrightarrow\Th{\cat{N}}
\]
between the categories of theories,
which is a restriction of the 2-functor $\MtMod{F_\ast}$.
On the other hand, $F$ induces a 2-functor
\[
\MtMod{F^\ast}\colon \MtMod{\cat{N}}\longrightarrow\MtMod{\cat{M}}
\]
between the 2-categories of metamodels.
The 2-adjointness $\MtMod{F_\ast}\dashv \MtMod{F^\ast}$
yields, for each theory $\monoid{T}$ in $\cat{M}$
and each metamodel $(\cat{C},\Phi)$ of $\cat{N}$,
an isomorphism of categories
\[
\Mod{F_\ast(\monoid{T})}{(\cat{C},\Phi)}\cong
\Mod{\monoid{T}}{(\cat{C},F^\ast(\Phi))}.
\]
Observe that $F_\ast(\monoid{T})=\Th{F}(\monoid{T})$ is the standard action of 
a strong monoidal 
functor on a theory, and $F^\ast(\Phi)$ is, in essence,
simply precomposition by $F$.

\medskip
Now we apply the above argument to some concrete cases.
\begin{example}
Recall from Section~\ref{sec:enrichment}, where we have compared clones, 
symmetric operads and non-symmetric operads, that there is a chain of lax 
monoidal functors
\[
\begin{tikzpicture}[baseline=-\the\dimexpr\fontdimen22\textfont2\relax ]
      \node (L) at (0,0)  {$[\Ncat,\Set]$};
      \node (M) at (3.5,0)  {$[\Pcat,\Set]$};
      \node (R) at (7,0)  {$[\F,\Set]$.};
      \draw[->]  (L) to node[auto,labelsize] 
      {$\Lan_J$} (M);
      \draw[->]  (M) to node[auto,labelsize] 
      {$\Lan_{J'}$} (R);
\end{tikzpicture} 
\] 
These functors, being left adjoints in $\MonCAT$, are in fact strong monoidal 
\cite{Kelly:doctrinal}.
Theories are mapped as follows, as noted in Section~\ref{sec:enrichment}:
\[
\begin{tikzpicture}[baseline=-\the\dimexpr\fontdimen22\textfont2\relax ]
      \node (L) at (0,0)  {$\Th{[\Ncat,\Set]}$};
      \node (M) at (4.5,0)  {$\Th{[\Pcat,\Set]}$};
      \node (R) at (9,0)  {$\Th{[\F,\Set]}.$};
      \draw[->]  (L) to node[auto,labelsize] {$\Th{\Lan_J}$} (M);
      \draw[->]  (M) to node[auto,labelsize] {$\Th{\Lan_{J'}}$} (R);
\end{tikzpicture} 
\]
In this case, the suitable 2-functors between 2-categories of metamodels
can be given either as $\MtMod{(\Lan_J)^\ast}$ or $\MtMod{([J,\Set])_\ast}$ 
(and similarly for $J'$),
because $(\Lan_J)^\ast\cong ([J,\Set])_\ast$ in $\MtTH$.
\end{example}
\begin{example}
Let us consider the relationship between clones and monads on $\Set$.
The inclusion functor
\[
J''\colon \F\longrightarrow\Set
\]
induces a functor 
\[
\Lan_{J''}\colon[\F,\Set]\longrightarrow[\Set,\Set],
\]
which naturally acquires the structure of a strong monoidal functor.
The essential image of this functor is precisely the \emph{finitary}
endofunctors on $\Set$, i.e., those endofunctors preserving filtered colimits.
The functor $\Th{\Lan_{J''}}$ maps a clone to a finitary monad on $\Set$,
in accordance with 
the well-known correspondence between clones (= Lawvere theories) and finitary 
monads on $\Set$ \cite{Linton_equational}.
Between the 2-categories of metamodels, we have a 2-functor
\[
\MtMod{(\Lan_{J''})^\ast}\colon\MtMod{[\Set,\Set]}\longrightarrow
\MtMod{[\F,\Set]}.
\]
The standard metamodel of $[\Set,\Set]$ in $\Set$ (corresponding to the 
definition of Eilenberg--Moore algebras) is given by the strict action
described in Example~\ref{ex:monad_standard_action};
in particular, its functor part $\Phi\colon 
[\Set,\Set]^\op\times\Set^\op\times\Set\longrightarrow\SET$ maps $(F,A,B)$ to 
$\Set(FA,B)$. 
The metamodel $(\Lan_{J''})^\ast(\Phi)$ of $[\F,\Set]$ in $\Set$ has the 
functor part $(\Lan_{J''})^\ast(\Phi)\colon[\F,\Set]^\op\times\Set^\op\times
\Set\longrightarrow\SET$ mapping $(X,A,B)$ to 
\begin{align*}
\Set((\Lan_{J''}X)A,B)
&= \Set\left(\int^{[n]\in\F} A^n\times X_n,B\right)\\
&\cong \int_{[n]\in\F} \Set (X_n, \Set(A^n,B))\\
&\cong [\F,\Set] (X, \enrich{A}{B}),
\end{align*}
where $\enrich{A}{B}\in[\F,\Set]$ in the final line is the one in 
Example~\ref{ex:clone_enrichment}.
Hence $\MtMod{(\Lan_{J''})^\ast}$ preserves the standard metamodels 
and this way we recover the well-known observation 
that the classical correspondence 
of clones and finitary monads on $\Set$ preserves semantics.

Note that by combining the previous example we obtain the chain
\[
\begin{tikzpicture}[baseline=-\the\dimexpr\fontdimen22\textfont2\relax ]
      \node (L) at (0,0)  {$[\Ncat,\Set]$};
      \node (M) at (2.5,0)  {$[\Pcat,\Set]$};
      \node (R) at (5,0)  {$[\F,\Set]$};
      \node (RR) at (7.5,0) {$[\Set,\Set]$};
      \draw[->]  (L) to node[auto,labelsize] 
      {$\Lan_J$} (M);
      \draw[->]  (M) to node[auto,labelsize] 
      {$\Lan_{J'}$} (R);
      \draw[->]  (R) to node[auto,labelsize] 
      {$\Lan_{J''}$} (RR);
\end{tikzpicture} 
\] 
of strong monoidal functors, connecting non-symmetric and symmetric operads 
with monads on $\Set$.
\end{example}

\begin{example}
Let $\cat{M}$ be a metatheory, $\cat{C}$ a large category,
and $\ast$ a \emph{pseudo} action of $\cat{M}$ on $\cat{C}$.
We obtain a strong monoidal functor 
\[
F\colon\cat{M}\longrightarrow[\cat{C},\cat{C}]
\]
(where $[\cat{C},\cat{C}]$ is equipped with the composition monoidal structure)
as the transpose of $\ast\colon\cat{M}\times\cat{C}\longrightarrow\cat{C}$.
The functor $\Th{F}$ maps any theory $\monoid{T}=(T,e,m)$ in $\cat{M}$
to the monad $F(\monoid{T})=(T\ast(-), e\ast(-),m\ast(-))$ on $\cat{C}$.
The 2-functor 
$\MtMod{F^\ast}\colon\MtMod{[\cat{C},\cat{C}]}\longrightarrow\MtMod{\cat{M}}$
maps the standard metamodel $\Phi$ of $[\cat{C},\cat{C}]$ in $\cat{C}$
(Example~\ref{ex:monad_standard_action}) to the metamodel 
$F^\ast(\Phi)\colon 
\cat{M}^\op\times\cat{C}^\op\times\cat{C}\longrightarrow\SET$
mapping $(X,A,B)$ to
\[
\cat{C}((FX)A,B)=\cat{C}(X\ast A,B).
\]
Therefore it maps the standard metamodel $\Phi$ to the metamodel induced from 
$\ast$.

As a special case, for a large category $\cat{C}$ with finite limits and a 
cartesian monad $\monoid{S}$ on $\cat{C}$,
the standard metamodel for $\monoid{S}$-operads
(Example~\ref{ex:metamodel_S-operad}) may be related to 
the standard metamodel of monads on $\cat{C}$,
and models of an $\monoid{S}$-operad $\monoid{T}$ may alternatively be defined 
as Eilenberg--Moore algebras of the monad on $\cat{C}$ induced from $\monoid{T}$
(as noted in \cite{Leinster_book}).
\end{example}

\subsection{The universal theory}
We conclude this section with another application of
the aforementioned method to compare different notions of algebraic
theory via strong monoidal functors.

Let us fix a large category $\cat{C}$ for the rest of this section.
As we have already seen,
whenever we are given a metatheory $\cat{M}$,
a theory $\monoid{T}$ in $\cat{M}$ and a 
metamodel $\Phi$ of $\cat{M}$ in $\cat{C}$,
we can consider the category $\Mod{\monoid{T}}{(\cat{C},\Phi)}$ of models and 
the associated forgetful functor to $\cat{C}$.

We now know that
there are in general multiple choices of the input data 
$(\cat{M},\monoid{T},\Phi)$ generating a given category of models 
and the associated forgetful functor.
That is, given metatheories $\cat{M}$ and $\cat{N}$, a strong monoidal functor 
$F\colon\cat{M}\longrightarrow\cat{N}$, a theory $\monoid{T}$ in ${\cat{M}}$
and a metamodel $\Phi$ of $\cat{N}$ in $\cat{C}$, 
the triples $(\cat{M},\monoid{T},F^\ast (\Phi))$ and 
$(\cat{N},F_\ast(\monoid{T}),\Phi)$ define canonically isomorphic categories of 
models over $\cat{C}$.
Let us combine this observation with the well-known concept of universal monoid 
object~\cite[Section~VII.5]{MacLane_CWM}.

We define the \defemph{augumented simplex category} (also known as the 
\defemph{algebraists' simplex category}) $\Simp$ 
\cite[Section~VII.5]{MacLane_CWM} to be the category of all finite ordinals 
(including the empty ordinal $\ord{0}$) and monotone functions.
The ordinal with $n$ elements $\{0 < 1 <\dots < n-1\}$
is denoted by $\ord{n}$.
We endow $\Simp$ with the structure of a strict monoidal category
via ordinal sum as monoidal product.
The (terminal) object $\ord{1}\in\Simp$, together with the 
unique morphisms $!_\ord{0}\colon\ord{0}\longrightarrow\ord{1}$ and 
$!_\ord{2}\colon\ord{2}\longrightarrow\ord{1}$
define a monoid object, i.e., a theory, in $\Simp$.
This theory $\monoid{U}=(\ord{1},!_\ord{0},!_\ord{2})$ 
is in fact the \emph{universal} theory,
in the sense that for any metatheory $\cat{M}$ 
and a theory $\monoid{T}$ therein,
there exists a strong monoidal functor $F\colon \Simp\longrightarrow\cat{M}$,
uniquely up to monoidal natural isomorphisms,
which carries $\monoid{U}$ to $\monoid{T}$.
(The monoidal category $\Simp$ being isomorphic to the PRO of monoids,
this last assertion is just the reformulation of its models in the style of 
functorial 
semantics.)

Consequently,
given any triple $(\cat{M},\monoid{T},\Phi)$ defining a 
category of models over $\cat{C}$,
we get the strong monoidal functor $F\colon \Simp\longrightarrow\cat{M}$
corresponding to $\monoid{T}$ (so that $F_\ast(\monoid{U})\cong\monoid{T}$),
and hence obtain the triple $(\Simp,\monoid{U},F^\ast(\Phi))$ which
defines the category of models over $\cat{C}$  isomorphic to the original one.
So, in fact, whichever triple $(\cat{M},\monoid{T},\Phi)$ we choose, 
the category of models over $\cat{C}$ defined by it can be 
realised (up to an isomorphism) as the category of models of the universal 
theory $\monoid{U}$,
with a suitable choice of a metamodel of $\Simp$ in $\cat{C}$
(or equivalently, in light of Proposition~\ref{prop:metatheory_as_enrichment},
with a suitable choice of an enrichment of $\cat{C}$ over $[\Simp^\op,\SET]$).
Let us define the category $\MtModfib{\cat{C}}{\Simp}$ of 
metamodels of $\Simp$ in $\cat{C}$
as the fibre of $\MtMod{\Simp}$ (thought of as a 2-category over $\tCAT$ via 
the evident forgetful 2-functor) above $\cat{C}\in\tCAT$, namely as the pullback
\begin{equation*}
\begin{tikzpicture}[baseline=-\the\dimexpr\fontdimen22\textfont2\relax ]
      \node (TL) at (0,2)  {$\MtModfib{\cat{C}}{\Simp}$};
      \node (TR) at (4,2)  {$\MtMod{\Simp}$};
      \node (BL) at (0,0) {$1$};
      \node (BR) at (4,0) {$\tCAT$};
      \draw[->] (TL) to node[auto,labelsize](T) {} 
      (TR);
      \draw[->]  (TR) to node[auto,labelsize] {$\text{forgetful}$} (BR);
      \draw[->]  (TL) to node[auto,swap,labelsize] {} 
      (BL);
      \draw[->]  (BL) to node[auto,swap,labelsize] {$\name{\cat{C}}$} (BR);
     \draw (0.2,1.4) -- (0.6,1.4) -- (0.6,1.8);
\end{tikzpicture}  
\end{equation*}
in $\twoCAT$ (then this pullback is locally discrete).
We have a functor 
\begin{equation}
\label{eqn:cat_of_models_over_C}
\Mod{\monoid{U}}{-}\colon\MtModfib{\cat{C}}{\Simp}\longrightarrow\CAT/\cat{C}
\end{equation}
mapping a metamodel $\Phi$ of $\Simp$ over $\cat{C}$ to the category of 
models $\Mod{\monoid{U}}{(\cat{C},\Phi)}$ equipped with the forgetful functor
to $\cat{C}$.

The open problem of finding an intrinsic characterisation of 
the class of forgetful functors 
arising in our framework, mentioned in the introduction,
may be phrased as that of 
finding intrinsic properties of objects in $\CAT/\cat{C}$
which identify the essential image of 
the functor (\ref{eqn:cat_of_models_over_C}).
From the definition of category of models 
(Definition~\ref{def:metamodel_model}),
it is immediate that 
such forgetful functors are faithful and 
amnestic~\cite[Definition~3.27]{AHS:joy-of-cats}
isofibrations,\footnote{A functor $F\colon\cat{A}\longrightarrow\cat{B}$
is an \defemph{isofibration} iff for any $A\in\cat{A}$ and an isomorphism
$g\colon FA\longrightarrow B'$ in $\cat{B}$,
there exists an isomorphism $f\colon A\longrightarrow A'$ in $\cat{A}$
such that $Ff=g$.}
for example.
We note that the functor (\ref{eqn:cat_of_models_over_C})
admits a left adjoint, which maps $(V\colon\cat{X}\longrightarrow\cat{C})
\in\CAT/\cat{C}$
to the metamodel $\Phi\colon \Simp^\op\times\cat{C}^\op\times\cat{C}
\longrightarrow\SET$
defined as 
\[
\Phi_{\ord{n}}(A,B)=\int^{X_1,\dots,X_n\in\cat{X}}
\cat{C}(A,VX_1)\times\cat{C}(VX_1,VX_2)\times\dots\times
\cat{C}(VX_n,B).
\]

\section{Structure-semantics adjunctions}
\label{sec:str-sem}

\emph{Structure-semantics adjunctions} are a classical topic in categorical 
algebra.
They are a family of 
adjunctions parametrised by a metatheory $\cat{M}$
and its metamodel $(\cat{C},\Phi,\overline{\phi}_\cdot,\overline{\phi})$;
if we fix these parameters, the structure-semantics adjunction is 
\emph{ideally} of type
\begin{equation}
\label{eqn:str_sem_idealised}
\begin{tikzpicture}[baseline=-\the\dimexpr\fontdimen22\textfont2\relax ]
      \node(11) at (0,0) 
      {$\Th{\cat{M}}^\op$};
      \node(22) at (4,0) {$\CAT/\cat{C}$,};
  
      \draw [->,transform canvas={yshift=5pt}]  (22) to node 
      [auto,swap,labelsize]{$\Str$} (11);
      \draw [->,transform canvas={yshift=-5pt}]  (11) to node 
      [auto,swap,labelsize]{$\Sem$} (22);
      \path (11) to node[midway](m){} (22); 

      \node at (m) [labelsize,rotate=90] {$\vdash$};
\end{tikzpicture}
\end{equation}
and the functor $\Sem$ is essentially $\Mod{-}{(\cat{C},\Phi)}$.
The idea is that we may regard an object of $\CAT/\cat{C}$, say 
$V\colon\cat{A}\longrightarrow\cat{C}$,
as specifying an additional structure (of a very general kind) 
on objects in $\cat{C}$, by viewing $\cat{A}$ as the category of 
$\cat{C}$-objects with the additional structure and $V$ as the associated 
forgetful functor.
The functor $\Str$ then extracts a theory from $V$,  
giving the ``best approximation'' of this additional structure by 
structures expressible by theories in $\cat{M}$.
Various authors have constructed such adjunctions for a variety of 
notions of algebraic theory, most notably for 
clones~\cite{Lawvere_thesis,Linton_equational,
Isbell_general_functorial_sem} and monads~\cite{Dubuc_Kan,Street_FTM}.
There are also some attempts to unify these 
results~\cite{Linton_outline,Avery_thesis}.

If we try to work this idea out, however, there turn out to be size-issues or 
other problems, and usually we cannot obtain an adjunction of type 
(\ref{eqn:str_sem_idealised});
we cannot find a suitable functor $\Str$ of that type.
To get an adjunction, various conditions on objects in $\CAT/\cat{C}$
were introduced in the literature in order to single out well-behaved
(usually called \emph{tractable}) objects, 
yielding a restricted version of (\ref{eqn:str_sem_idealised}):
\begin{equation}
\label{eqn:str_sem_tractable}
\begin{tikzpicture}[baseline=-\the\dimexpr\fontdimen22\textfont2\relax ]
      \node(11) at (0,0) 
      {$\Th{\cat{M}}^\op$};
      \node(22) at (4,0) {$(\CAT/\cat{C})_{\mathrm{tr}}$.};
  
      \draw [->,transform canvas={yshift=5pt}]  (22) to node 
      [auto,swap,labelsize]{$\Str$} (11);
      \draw [->,transform canvas={yshift=-5pt}]  (11) to node 
      [auto,swap,labelsize]{$\Sem$} (22);
      \path (11) to node[midway](m){} (22); 

      \node at (m) [labelsize,rotate=90] {$\vdash$};
\end{tikzpicture}
\end{equation}
Here, $(\CAT/\cat{C})_\mathrm{tr}$ is the full-subcategory of 
$\CAT/\cat{C}$ consisting of all {tractable} objects.

In this section, we construct a structure-semantics adjunction for 
an arbitrary metatheory and an arbitrary metamodel of it.
Of course, we cannot obtain an adjunction of type (\ref{eqn:str_sem_idealised}),
for the same reasons that have prevented  
other authors from doing so.
However, we shall obtain a modified adjunction by a strategy different from 
theirs (and similar to \cite{Linton_outline,Avery_thesis}):
instead of restricting $\CAT/\cat{C}$,
we {extend} $\Th{\cat{M}}$ to $\Th{\widehat{\cat{M}}}$\footnote{The monoidal 
category $\widehat{\cat{M}}$ is not a metatheory because it is not large.
Extending Definition~\ref{def:theory}, by $\Th{\widehat{\cat{M}}}$ we mean
the category of monoids in $\widehat{\cat{M}}$.}
(where $\widehat{\cat{M}}=[\cat{M}^\op,\SET]$ is equipped with
the convolution monoidal structure),
and obtain an extended version of (\ref{eqn:str_sem_idealised}):
\begin{equation}
\label{eqn:str_sem_extend}
\begin{tikzpicture}[baseline=-\the\dimexpr\fontdimen22\textfont2\relax ]
      \node(11) at (0,0) 
      {$\Th{\widehat{\cat{M}}}^\op$};
      \node(22) at (4,0) {$\CAT/\cat{C}$.};
  
      \draw [->,transform canvas={yshift=5pt}]  (22) to node 
      [auto,swap,labelsize]{$\Str$} (11);
      \draw [->,transform canvas={yshift=-5pt}]  (11) to node 
      [auto,swap,labelsize]{$\Sem$} (22);
      \path (11) to node[midway](m){} (22); 

      \node at (m) [labelsize,rotate=90] {$\vdash$};
\end{tikzpicture}
\end{equation}
We may then obtain known adjunctions of the form (\ref{eqn:str_sem_tractable}) 
for clones and monads,
by suitably restricting (\ref{eqn:str_sem_extend}).

\subsection{The structure and semantics functors}
Let $\cat{M}=(\cat{M},I,\otimes)$ be a metatheory, $\cat{C}$ be a large 
category, and $\Phi=(\Phi,\overline{\phi}_\cdot,\overline{\phi})$ be a 
metamodel of $\cat{M}$ in $\cat{C}$.
The metamodel $\Phi$ enables us to define, for each theory $\monoid{T}
\in\Th{\cat{M}}$, the category of models $\Mod{\monoid{T}}{(\cat{C},\Phi)}$
together with the forgetful functor $U\colon \Mod{\monoid{T}}{(\cat{C},\Phi)}
\longrightarrow\cat{C}$.
This construction is functorial, and gives rise to a functor
\[
\Th{\cat{M}}^\op\longrightarrow\CAT/\cat{C}.
\]
However, as we have remarked in Proposition~\ref{prop:metatheory_as_enrichment},
a metamodel of $\cat{M}$ in $\cat{C}$ corresponds to 
an enrichment of $\cat{C}$ over $\widehat{\cat{M}}$;
hence using $\Phi$ we can actually give 
the definition of models for any theory (i.e., monoid object) in 
$\widehat{\cat{M}}$.
Therefore the previous functor can be extended to 
\begin{equation}
\label{eqn:semantics_functor}
\Sem\colon \Th{\widehat{\cat{M}}}^\op\longrightarrow\CAT/\cat{C}.
\end{equation}
The category $\Th{\widehat{\cat{M}}}$ is 
isomorphic to the category of lax monoidal functors of type 
$\cat{M}^\op\longrightarrow\SET$ and monoidal natural transformations between 
them.
Indeed, an object $(P,e,m)$ of $\Th{\widehat{\cat{M}}}$
consists of:
\begin{itemize}
\item a functor $P\colon\cat{M}^\op\longrightarrow\SET$;
\item a natural transformation $(e_X\colon \DayI (X)\longrightarrow 
P(X))_{X\in\cat{M}}$;
\item a natural transformation $(m_X\colon (P\Dayo P)(X)\longrightarrow 
P(X))_{X\in\cat{M}}$
\end{itemize}
satisfying the monoid axioms, and such a data is equivalent to 
\begin{itemize}
\item a functor $P\colon\cat{M}^\op\longrightarrow\SET$;
\item a function $\overline{e}\colon 1\longrightarrow P(I)$;
\item a natural transformation $(\overline{m}_{X,Y}\colon P(Y)\times P(X)
\longrightarrow P(Y\otimes X))_{X,Y\in\cat{M}}$
\end{itemize}
satisfying the axioms for $(P,\overline{e},\overline{m})$ to be a
lax monoidal functor $\cat{M}^\op\longrightarrow\SET$.
We shall use these two descriptions of objects in
$\Th{\widehat{\cat{M}}}$ 
interchangeably.

Let us describe the action of the functor $\Sem$ concretely. 
For any $\monoid{P}=(P,e,m)\in\Th{\widehat{\cat{M}}}$, we define 
the category $\Mod{\monoid{P}}{(\cat{C},\Phi)}$ as follows:
\begin{itemize}
\item An object is a pair consisting of an object $C\in\cat{C}$ and 
a natural transformation 
\[
(\xi_X\colon P(X)\longrightarrow \Phi_X(C,C))_{X\in\cat{M}}
\]
making the following diagrams commute for each $X,Y\in\cat{M}$:
\begin{equation}
\label{eqn:model_of_P}
\begin{tikzpicture}[baseline=-\the\dimexpr\fontdimen22\textfont2\relax ]
      \node (TL) at (0,1)  {$1$};
      \node (TR) at (2,1)  {$P(I)$};
      \node (BR) at (2,-1) {$\Phi_I(C,C)$};
      \draw[->] (TL) to node[auto,labelsize](T) {$\overline{e}$} (TR);
      \draw[->]  (TR) to node[auto,labelsize] {$\xi_I$} (BR);
      \draw[->]  (TL) to node[auto,swap,labelsize] 
      {$(\overline{\phi}_\cdot)_{C}$} 
      (BR);
\end{tikzpicture} 
\quad
\begin{tikzpicture}[baseline=-\the\dimexpr\fontdimen22\textfont2\relax ]
      \node (TL) at (0,1)  {$P(Y)\times P(X)$};
      \node (TR) at (4.5,1)  {$P({Y\otimes X})$};
      \node (BL) at (0,-1) {$\Phi_Y(C,C)\times \Phi_X(C,C)$};
      \node (BR) at (4.5,-1) {$\Phi_{Y\otimes X}(C,C).$};
      \draw[->] (TL) to node[auto,labelsize](T) 
      {$\overline{m}_{X,Y}$} 
      (TR);
      \draw[->]  (TR) to node[auto,labelsize] {$\xi_{Y\otimes X}$} (BR);
      \draw[->]  (TL) to node[auto,swap,labelsize] {$\xi_Y\times \xi_X$} 
      (BL);
      \draw[->]  (BL) to node[auto,labelsize] 
      {$(\overline{\phi}_{X,Y})_{C,C,C}$} (BR);
\end{tikzpicture} 
\end{equation}
\item A morphism from $(C,\xi)$ to $(C',\xi')$ is a morphism $f\colon 
C\longrightarrow C'$ in $\cat{C}$ making the following diagram commute for 
each $X\in\cat{M}$:
\begin{equation*}
\begin{tikzpicture}[baseline=-\the\dimexpr\fontdimen22\textfont2\relax ]
      \node (TL) at (0,1)  {$P(X)$};
      \node (TR) at (4,1)  {$\Phi_X(C,C)$};
      \node (BL) at (0,-1) {$\Phi_X(C',C')$};
      \node (BR) at (4,-1) {$\Phi_{X}(C,C').$};
      \draw[->] (TL) to node[auto,labelsize](T) {$\xi_X$} (TR);
      \draw[->]  (TR) to node[auto,labelsize] {$\Phi_X(C,f)$} (BR);
      \draw[->]  (TL) to node[auto,swap,labelsize] {$\xi'_X$} (BL);
      \draw[->]  (BL) to node[auto,labelsize] {$\Phi_X(f,C')$} (BR);
\end{tikzpicture} 
\end{equation*}
\end{itemize}
There exists an evident forgetful functor 
$U\colon\Mod{\monoid{P}}{(\cat{C},\Phi)}\longrightarrow\cat{C}$ mapping 
$(C,\xi)$ to $C$ and $f$ to $f$;
the functor $\Sem$ maps $\monoid{P}$ to $U$.

We have a canonical fully faithful functor
\begin{equation*}
J\colon\Th{\cat{M}}\longrightarrow\Th{\widehat{\cat{M}}}
\end{equation*}
mapping $(T,e,m)\in\Th{\cat{M}}$ to the functor $\cat{M}(-,T)$
with the evident monoid structure induced from $e$ and $m$.
An object $(P,e,m)\in\Th{\widehat{\cat{M}}}$ is in the essential 
image of $J$
if and only if $P\colon\cat{M}^\op\longrightarrow\SET$ is representable.

Let us describe the left adjoint $\Str$ to (\ref{eqn:semantics_functor}).
Given an object $V\colon\cat{A}\longrightarrow\cat{C}$ of $\CAT/\cat{C}$, 
we define $\Str(V)=(P^{(V)},e^{(V)},m^{(V)})\in\Th{\widehat{\cat{M}}}$ as 
follows:
\begin{itemize}
\item The functor $P^{(V)}\colon \cat{M}^\op\longrightarrow\SET$
maps $X\in\cat{M}$ to 
\begin{equation}
\label{eqn:structure_of_V}
P^{(V)}(X)=\int_{A\in\cat{A}}\Phi_X(VA,VA).
\end{equation}
\item The function $\overline{e^{(V)}}\colon 1\longrightarrow P^{(V)}(I)$
maps the unique element of $1$ to 
$((\overline{\phi}_\cdot)_{VA} (\ast))_{A\in\cat{A}}$ $\in P^{(V)}(I)$.
\item The $(X,Y)$-th component of the natural transformation 
\[(\overline{m^{(V)}}_{X,Y}\colon P^{(V)}(Y)\times P^{(V)}(X)\longrightarrow 
P^{(V)}(Y\otimes X))_{X,Y\in\cat{M}}\]
maps $((y_A)_{A\in\cat{A}},(x_A)_{A\in\cat{A}})$
to $((\overline{\phi}_{X,Y})_{VA,VA,VA}(y_A,x_A))_{A\in\cat{A}}$.
\end{itemize}
The monoid axioms for $(P^{(V)},{e^{(V)}},{m^{(V)}})$ 
follow easily from the axioms for metamodels,
and $\Str$ routinely extends to a functor of type 
$\CAT/\cat{C}\longrightarrow\Th{\widehat{\cat{M}}}^\op$. 

\begin{theorem}
\label{thm:str_sem_small}
Let $\cat{M}$ be a metatheory, $\cat{C}$ be a large category and 
$\Phi=(\Phi,\overline{\phi}_\cdot,\overline{\phi})$ 
be a metamodel of $\cat{M}$ in $\cat{C}$.
The functors $\Sem$ and $\Str$ defined above
form an adjunction:
\[
\begin{tikzpicture}[baseline=-\the\dimexpr\fontdimen22\textfont2\relax ]
      \node(11) at (0,0) 
      {$\Th{\widehat{\cat{M}}}^\op$};
      \node(22) at (4,0) {$\CAT/\cat{C}$.};
  
      \draw [->,transform canvas={yshift=5pt}]  (22) to node 
      [auto,swap,labelsize]{$\Str$} (11);
      \draw [->,transform canvas={yshift=-5pt}]  (11) to node 
      [auto,swap,labelsize]{$\Sem$} (22);
      \path (11) to node[midway](m){} (22); 

      \node at (m) [labelsize,rotate=90] {$\vdash$};
\end{tikzpicture}
\] 
\end{theorem}
\begin{proof}
We show that there are bijections
\[\Th{\widehat{\cat{M}}}(\monoid{P},\Str(V))\cong
(\CAT/\cat{C})(V,\Sem(\monoid{P}))\]
natural in $\monoid{P}=(P,e,m)\in\Th{\widehat{\cat{M}}}$ and $(V\colon 
\cat{A}\longrightarrow\cat{C})\in\CAT/\cat{C}$.

In fact, we show that the following three types of data naturally 
correspond to each other.
\begin{enumerate}
\item A morphism $\alpha\colon \monoid{P}\longrightarrow\Str(V)$ in 
$\Th{\widehat{\cat{M}}}$; that is, a natural transformation 
\[
(\alpha_X\colon P(X)\longrightarrow P^{(V)}(X))_{X\in\cat{M}}
\]
making the suitable diagrams commute.
\item A natural transformation \[
(\xi_{A,X}\colon P(X)\longrightarrow \Phi_X(VA,VA))_{X\in\cat{M},
A\in\cat{A}}
\]
making the following diagrams commute for each $A\in\cat{A}$ and 
$X,Y\in\cat{M}$:
\begin{equation*}
\begin{tikzpicture}[baseline=-\the\dimexpr\fontdimen22\textfont2\relax ]
      \node (TL) at (0,1)  {$1$};
      \node (TR) at (2,1)  {$P(I)$};
      \node (BR) at (2,-1) {$\Phi_I(VA,VA)$};
      \draw[->] (TL) to node[auto,labelsize](T) {$\overline{e}$} (TR);
      \draw[->]  (TR) to node[auto,labelsize] {$\xi_{A,I}$} (BR);
      \draw[->]  (TL) to node[auto,swap,labelsize] 
      {$(\overline{\phi}_\cdot)_{VA}$} 
      (BR);
\end{tikzpicture} 
\quad
\begin{tikzpicture}[baseline=-\the\dimexpr\fontdimen22\textfont2\relax ]
      \node (TL) at (0,1)  {$P(Y)\times P(X)$};
      \node (TR) at (5.5,1)  {$P({Y\otimes X})$};
      \node (BL) at (0,-1) {$\Phi_Y(VA,VA)\times \Phi_X(VA,VA)$};
      \node (BR) at (5.5,-1) {$\Phi_{Y\otimes X}(VA,VA).$};
      \draw[->] (TL) to node[auto,labelsize](T) 
      {$\overline{m}_{X,Y}$} 
      (TR);
      \draw[->]  (TR) to node[auto,labelsize] {$\xi_{A,Y\otimes X}$} (BR);
      \draw[->]  (TL) to node[auto,swap,labelsize] {$\xi_{A,Y}\times 
      \xi_{A,X}$} 
      (BL);
      \draw[->]  (BL) to node[auto,labelsize] 
      {$(\overline{\phi}_{X,Y})_{VA,VA,VA}$} (BR);
\end{tikzpicture} 
\end{equation*}
\item A morphism $F\colon V\longrightarrow\Sem(\monoid{P})$ in $\CAT/\cat{C}$;
that is, a functor 
$F\colon\cat{A}\longrightarrow\Mod{\monoid{P}}{(\cat{C},\Phi)}$
such that $U\circ F=V$ 
($U\colon\Mod{\monoid{P}}{(\cat{C},\Phi)}\longrightarrow\cat{C}$ is the 
forgetful functor).
\end{enumerate}

The correspondence between 1 and 2 is by the universality of ends
(see (\ref{eqn:structure_of_V})). 
To give $\xi$ as in 2 \emph{without} requiring naturality
in $A\in\cat{A}$, is equivalent to give a 
function $\ob{F}\colon 
\ob{\cat{A}}\longrightarrow\ob{\Mod{\monoid{P}}{(\cat{C},\Phi)}}$
such that $\ob{U}\circ \ob{F}=\ob{V}$ (see (\ref{eqn:model_of_P})). 
To say that $\xi$ is natural also in $A\in\cat{A}$
is equivalent to saying that $\ob{F}$ extends to a functor 
$F\colon \cat{A}\longrightarrow\Mod{\monoid{P}}{(\cat{C},\Phi)}$
by mapping each morphism $f$ in $\cat{A}$ to $Vf$.
\end{proof}

\subsection{The classical cases}
We now show that one can recover the known
structure-semantics adjunctions for clones and monads, by restricting
our version of structure-semantics adjunctions 
(Theorem~\ref{thm:str_sem_small}).

In both cases of clones and monads, we shall consider the diagram
\[
\begin{tikzpicture}[baseline=-\the\dimexpr\fontdimen22\textfont2\relax ]
      \node(11T) at (0,2) 
      {$\Th{\cat{\widehat{M}}}^\op$};
      \node(22T) at (4,2) {$\CAT/\cat{C}$};
      \draw [->,transform canvas={yshift=5pt}]  (22T) to node 
      [auto,swap,labelsize]{$\Str$} (11T);
      \draw [->,transform canvas={yshift=-5pt}]  (11T) to node 
      [auto,swap,labelsize]{$\Sem$} (22T);
      \path (11T) to node[midway](mT){} (22T); 
      \node at (mT) [labelsize,rotate=90] {$\vdash$};
      \node(11) at (0,0) 
      {$\Th{\cat{M}}^\op$};
      \node(22) at (4,0) {$(\CAT/\cat{C})_{\mathrm{tr}}$};
      \draw [->,transform canvas={yshift=5pt}]  (22) to node 
      [auto,swap,labelsize]{$\Str'$} (11);
      \draw [->,transform canvas={yshift=-5pt}]  (11) to node 
      [auto,swap,labelsize]{$\Sem'$} (22);
      \path (11) to node[midway](m){} (22); 
      \node at (m) [labelsize,rotate=90] {$\vdash$};
      \draw[->] (11) to node[auto,labelsize] {$J$} (11T);
      \draw[->] (22) to node[auto,swap,labelsize] {$K$} (22T);
\end{tikzpicture}
\]
in which the top adjunction is the one we have constructed above,
the bottom adjunction is a classical structure-semantics adjunction,
and $J$ and $K$ are the canonical fully faithful functors
(the precise definition of $(\CAT/\cat{C})_\mathrm{tr}$ will be given below).
We shall prove that the two squares, one involving $\Str$ and $\Str'$,
the other involving $\Sem$ and $\Sem'$, commute,
showing that $\Str'$ (resp.~$\Sem'$) arises as a restriction of 
$\Str$ (resp.~$\Sem$).

First, that $K\circ \Sem'\cong \Sem\circ J$ holds is straightforward,
and this is true as soon as $\Sem'$ maps any $\monoid{T}\in\Th{\cat{M}}$ 
to the forgetful functor 
$U\colon\Mod{\monoid{T}}{(\cat{C},\Phi)}\longrightarrow\cat{C}$.
Indeed, for any theory $\monoid{T}=(T,e,m)$ in $\cat{M}$,
$J(\monoid{T})\in\Th{\widehat{\cat{M}}}$ has the underlying object 
$\cat{M}(-,T)\in\widehat{\cat{M}}$, and the description of 
$\Mod{J(\monoid{T})}{(\cat{C},\Phi)}$ in the previous section coincides with
$\Mod{\monoid{T}}{(\cat{C},\Phi)}$ by the Yoneda lemma.

Let us check that $J\circ \Str'\cong \Str\circ K$ holds.\footnote{This does not 
seem to follow formally from $K\circ \Sem'\cong \Sem\circ J$,
even if we take into consideration the fact that $J$ and $K$ are fully 
faithful.}
For this, we have to review the classical structure functors
and the tractability conditions.

We begin with the case of clones as treated in \cite{Linton_equational}.
Let $\cat{C}$ be a locally small category with finite powers and 
consider the standard metamodel $\Phi$ of $[\F,\Set]$ in $\cat{C}$
(derived from the enrichment $\enrich{-}{-}$ in 
Example~\ref{ex:clone_enrichment}). 
An object $V\colon\cat{A}\longrightarrow\cat{C}\in\CAT/\cat{C}$ is called 
\defemph{tractable} if and only if for any natural number $n$,
the set $[\cat{A},\cat{C}]((-)^n\circ V,V)$ is small.
Given a tractable $V$, $\Str'(V)\in\Th{[\F,\Set]}$ has the underlying 
functor $|\Str'(V)|$ mapping $[n]\in\F$ to $[\cat{A},\cat{C}]((-)^n\circ V,V)$.
On the other hand, our formula (\ref{eqn:structure_of_V})
reduces as follows:
\begin{align*}
P^{(V)}(X)&=\int_{A\in\cat{A}}\Phi_X(VA,VA)\\
&= \int_{A\in\cat{A}} [\F,\Set](X, \enrich{VA}{VA})\\
&\cong \int_{A\in\cat{A},[n]\in\F} \Set(X_n, \cat{C}((VA)^n,VA))\\
&\cong \int_{[n]\in\F} \Set\bigg(X_n, 
\int_{A\in\cat{A}}\cat{C}((VA)^n,VA)\bigg)\\
&\cong \int_{[n]\in\F} \Set(X_n, [\cat{A},\cat{C}]((-)^n\circ V,V))\\
&\cong [\F,\Set](X, |\Str'(V)|). 
\end{align*}
It is routine from this to see that $J\circ \Str'\cong \Str\circ K$ holds.

As for monads, we take as a classical structure-semantics adjunction 
the one in \cite[Section~II.~1]{Dubuc_Kan}.
Let $\cat{C}$ be a large category and consider the standard metamodel $\Phi$
of $[\cat{C},\cat{C}]$ in $\cat{C}$ (derived from the standard strict action 
$\ast$ in Example~\ref{ex:monad_standard_action}).
An object $V\colon \cat{A}\longrightarrow\cat{C}\in\CAT/\cat{C}$
is called \defemph{tractable} if and only if the right Kan extension  
$\Ran_V V$ of $V$ along itself exists.\footnote{In fact, in 
\cite[p.~68]{Dubuc_Kan} 
Dubuc defines tractability as a slightly stronger condition. 
However, the condition we have introduced above is the one which is used for 
the construction of structure-semantics adjunctions in \cite{Dubuc_Kan}.}
A functor of the form
$\Ran_V V$ acquires a canonical monad structure,
giving rise to the so-called {codensity monad of $V$}. 
For a tractable $V$, $\Str'(V)$ is defined to be the codensity monad of $V$.
Now let us return to our formula (\ref{eqn:structure_of_V}):
\begin{align*}
P^{(V)}(X)&=\int_{A\in\cat{A}}\Phi_X(VA,VA)\\
&= \int_{A\in\cat{A}} \cat{C}(XVA,{VA})\\
&\cong [\cat{A},\cat{C}](X\circ V,V)\\
&\cong [\cat{C},\cat{C}](X,\Ran_V V). 
\end{align*}
Again we see that $J\circ \Str'\cong \Str\circ K$ holds.

\section{Categories of models as double limits}
\label{sec:double_lim}
Let $\cat{M}$ be a metatheory, $\monoid{T}$ be a theory in $\cat{M}$,
$\cat{C}$ be a large category and $\Phi$ be a metamodel of $\cat{M}$ in 
$\cat{C}$.
Given these data, in Definition~\ref{def:metamodel_model} we have defined---in
a concrete manner---the category $\Mod{\monoid{T}}{(\cat{C},\Phi)}$ of models 
of $\monoid{T}$ in $\cat{C}$ with respect to $\Phi$,
equipped with the forgetful functor 
$U\colon\Mod{\monoid{T}}{(\cat{C},\Phi)}\longrightarrow\cat{C}$.

In this section, we give an abstract characterisation of the categories of 
models.
A similar result is known for the {Eilenberg--Moore category}
of a monad; Street \cite{Street_FTM}
has proved that it can be abstractly characterised as the lax limit in 
$\tCAT$ of a certain diagram canonically constructed from the original monad.
We prove that the categories of models in our framework can also be 
characterised by a certain universal property.
A suitable language to express this universal property is
that of \emph{pseudo double categories}~\cite{GP1}.
We show that the category $\Mod{\monoid{T}}{(\cat{C},\Phi)}$, together with 
the forgetful functor $U$ and some other natural data, form a double limit 
in the pseudo double category $\PPROF$ of large categories,
profunctors, functors and natural transformations.

\subsection{The universality of Eilenberg--Moore categories}
\label{subsec:EM_cat_as_2-lim}
Let us begin with quickly reviewing the 2-categorical characterisation in 
\cite{Street_FTM} of the Eilenberg--Moore category of a monad on a large 
category, in elementary terms.\footnote{The main point of the paper  
\cite{Street_FTM} is the introduction of the general 
notion of Eilenberg--Moore object in a 2-category $\tcat{B}$ via a universal 
property and show 
that, if exists, it satisfies certain formal properties of Eilenberg--Moore 
categories.
However, for our purpose, it suffices to consider the simple case 
$\tcat{B}=\tCAT$ only. It is left as future work to investigate whether 
we can develop a similar ``formal theory'' from the double-categorical 
universal property of categories of models.}
Let $\cat{C}$ be a large category and $\monoid{T}=(T,\eta,\mu)$ be a monad on 
$\cat{C}$.
The Eilenberg--Moore category $\cat{C}^\monoid{T}$ of $\monoid{T}$ 
is equipped with a canonical forgetful functor $U\colon 
\cat{C}^\monoid{T}\longrightarrow\cat{C}$, mapping an Eilenberg--Moore algebra 
$(C,\gamma)$ of $\monoid{T}$ to its underlying object $C$.
Moreover, there exists a canonical natural transformation $u\colon T\circ 
U\Longrightarrow U$, i.e., of type 
\[
\begin{tikzpicture}[baseline=-\the\dimexpr\fontdimen22\textfont2\relax ]
      \node (TL) at (0,0.75)  {$\cat{C}^\monoid{T}$};
      \node (TR) at (2,0.75)  {$\cat{C}^\monoid{T}$};
      \node (BL) at (0,-0.75) {$\cat{C}$};
      \node (BR) at (2,-0.75) {$\cat{C}.$};
      \draw[->]  (TL) to node[auto,swap,labelsize] {$U$} (BL);
      \draw[->]  (TL) to node[auto,labelsize] {$\id{\cat{C}^\monoid{T}}$} (TR);
      \draw[->]  (BL) to node[auto,swap,labelsize] {$T$} (BR);
      \draw[->]  (TR) to node[auto,labelsize](B) {$U$} (BR);
      \tzsquareup{1}{0}{$u$}
\end{tikzpicture} 
\] 
We are depicting $u$ in a square rather than in a triangle for later comparison 
with similar diagrams in a pseudo double category.
For each $(C,\gamma)\in\cat{C}^\monoid{T}$, 
the $(C,\gamma)$-th component of $u$ is simply $\gamma\colon TC\longrightarrow 
C$.
The claim is that the data $(\cat{C}^\monoid{T}, U,u)$ is characterised by a
certain universal property.

To state this universal property, let us define 
a \defemph{left $\monoid{T}$-module} to be 
a triple $(\cat{A},V,v)$ consisting of a large category $\cat{A}$,
a functor $V\colon \cat{A}\longrightarrow\cat{C}$ and a natural transformation 
$v\colon T\circ V\Longrightarrow V$,
such that the following equations hold:
\[
\begin{tikzpicture}[baseline=-\the\dimexpr\fontdimen22\textfont2\relax ]
      \node (TL) at (0,0.75)  {$\cat{A}$};
      \node (TR) at (2,0.75)  {$\cat{A}$};
      \node (BL) at (0,-0.75) {$\cat{C}$};
      \node (BR) at (2,-0.75) {$\cat{C}$};
      \draw[->] (TL) to node[auto,labelsize](T) {$\id{\cat{A}}$} 
      (TR);
      \draw[->]  (TR) to node[auto,labelsize] {$V$} (BR);
      \draw[->]  (TL) to node[auto,swap,labelsize] {$V$} 
      (BL);
      \draw[->, bend left=30] (BL) to node[auto,labelsize](B) 
      {$T$} 
      (BR);
      \draw[->, bend right=30] (BL) to node[auto,swap,labelsize](B) 
      {$\id{\cat{C}}$} 
      (BR);
      \tzsquareup{1}{0.25}{$v$}
      \tzsquareup{1}{-0.75}{$\eta$}
\end{tikzpicture}
\quad=\quad
\begin{tikzpicture}[baseline=-\the\dimexpr\fontdimen22\textfont2\relax ]
      \node (TL) at (0,0.75)  {$\cat{A}$};
      \node (TR) at (2,0.75)  {$\cat{A}$};
      \node (BL) at (0,-0.75) {$\cat{C}$};
      \node (BR) at (2,-0.75) {$\cat{C}$};
      \draw[->] (TL) to node[auto,labelsize](T) {$\id{\cat{A}}$} 
      (TR);
      \draw[->]  (TR) to node[auto,labelsize] {$V$} (BR);
      \draw[->]  (TL) to node[auto,swap,labelsize] {$V$} 
      (BL);
      \draw[->] (BL) to node[auto,swap,labelsize](B) {$\id{\cat{C}}$} 
      (BR);
      \tzsquareup{1}{0}{$\id{V}$}
\end{tikzpicture}
\]
\[
\begin{tikzpicture}[baseline=-\the\dimexpr\fontdimen22\textfont2\relax ]
      \node (TL) at (0,0.75)  {$\cat{A}$};
      \node (TR) at (2,0.75)  {$\cat{A}$};
      \node (BL) at (0,-0.75) {$\cat{C}$};
      \node (BR) at (2,-0.75) {$\cat{C}$};
      \draw[->] (TL) to node[auto,labelsize](T) {$\id{\cat{A}}$} 
      (TR);
      \draw[->]  (TR) to node[auto,labelsize] {$V$} (BR);
      \draw[->]  (TL) to node[auto,swap,labelsize] {$V$} 
      (BL);
      \draw[->, bend left=30] (BL) to node[auto,labelsize](B) 
      {$T$} 
      (BR);
      \draw[->, bend right=30] (BL) to node[auto,swap,labelsize](B) 
      {$T\circ T$} 
      (BR);
      \tzsquareup{1}{0.25}{$v$}
      \tzsquareup{1}{-0.75}{$\mu$}
\end{tikzpicture}
\quad=\quad
\begin{tikzpicture}[baseline=-\the\dimexpr\fontdimen22\textfont2\relax ]
      \node (TL) at (0,0.75)  {$\cat{A}$};
      \node (TR) at (2,0.75)  {$\cat{A}$};
      \node (BL) at (0,-0.75) {$\cat{C}$};
      \node (BR) at (2,-0.75) {$\cat{C}$};
      \node (TRR) at (4,0.75)  {$\cat{A}$};
      \node (BRR) at (4,-0.75) {$\cat{C}.$};
      \draw[->] (TL) to node[auto,labelsize](T) {$\id{\cat{A}}$} 
      (TR);
      \draw[->] (TR) to node[auto,labelsize](T2) {$\id{\cat{A}}$} 
      (TRR);
      \draw[->]  (TR) to node[auto,labelsize] {$V$} (BR);
      \draw[->]  (TRR) to node[auto,labelsize] {$V$} (BRR);
      \draw[->]  (TL) to node[auto,swap,labelsize] {$V$} 
      (BL);
      \draw[->] (BL) to node[auto,swap,labelsize](B) {$T$} 
      (BR);
      \draw[->] (BR) to node[auto,swap,labelsize](B) {$T$} 
      (BRR);
      \tzsquareup{1}{0}{$v$}
      \tzsquareup{3}{0}{$v$}
\end{tikzpicture}
\]

The triple $(\cat{C}^\monoid{T}, U,u)$ is then a \defemph{universal left 
$\monoid{T}$-module}, meaning that it satisfies the following:
\begin{enumerate}
\item it is a left $\monoid{T}$-module;
\item for any left $\monoid{T}$-module $(\cat{A},V,v)$,
there exists a unique functor $K\colon \cat{A}\longrightarrow\cat{C}^\monoid{T}$
such that 
\[
\begin{tikzpicture}[baseline=-\the\dimexpr\fontdimen22\textfont2\relax ]
      \node (TL) at (0,0.75)  {$\cat{A}$};
      \node (TR) at (2,0.75)  {$\cat{A}$};
      \node (BL) at (0,-0.75) {$\cat{C}$};
      \node (BR) at (2,-0.75) {$\cat{C}$};
      \draw[->] (TL) to node[auto,labelsize](T) {$\id{\cat{A}}$} 
      (TR);
      \draw[->]  (TR) to node[auto,labelsize] {$V$} (BR);
      \draw[->]  (TL) to node[auto,swap,labelsize] {$V$} 
      (BL);
      \draw[->] (BL) to node[auto,swap,labelsize](B) {$T$} (BR);
      \tzsquareup{1}{0}{$v$}
\end{tikzpicture}  
\quad=\quad
\begin{tikzpicture}[baseline=-\the\dimexpr\fontdimen22\textfont2\relax ]
      \node (TL) at (0,1.5)  {$\cat{A}$};
      \node (TR) at (2,1.5)  {$\cat{A}$};
      \node (BL) at (0,0) {$\cat{C}^\monoid{T}$};
      \node (BR) at (2,0) {$\cat{C}^\monoid{T}$};
      \node (BBL) at (0,-1.5) {$\cat{C}$};
      \node (BBR) at (2,-1.5) {$\cat{C}$};
      \draw[->] (TL) to node[auto,labelsize](T) {$\id{\cat{A}}$} 
      (TR);
      \draw[->]  (TR) to node[auto,labelsize] {$K$} (BR);
      \draw[->]  (TL) to node[auto,swap,labelsize] {$K$} 
      (BL);
      \draw[->] (BL) to node[auto,labelsize](B) {$\id{\cat{C}^\monoid{T}}$} 
      (BR);
      \draw[->]  (BR) to node[auto,labelsize] {$U$} (BBR);
      \draw[->]  (BL) to node[auto,swap,labelsize] {$U$} (BBL);
      \draw[->] (BBL) to node[auto,swap,labelsize](B) {$T$} (BBR);
      \tzsquareup{1}{0.75}{$\id{K}$}
      \tzsquareup{1}{-0.75}{$u$} 
\end{tikzpicture} 
\]
holds;
\item for any pair of left $\monoid{T}$-modules $(\cat{A},V,v)$ and 
$(\cat{A},V',v')$ on a common large category $\cat{A}$
and any natural transformation $\theta\colon V\Longrightarrow V'$ such that 
\[
\begin{tikzpicture}[baseline=-\the\dimexpr\fontdimen22\textfont2\relax ]
      \node (TL) at (0,0.75)  {$\cat{A}$};
      \node (TR) at (2,0.75)  {$\cat{A}$};
      \node (BL) at (0,-0.75) {$\cat{C}$};
      \node (BR) at (2,-0.75) {$\cat{C}$};
      \draw[->] (TL) to node[auto,labelsize](T) {$\id{\cat{A}}$} 
      (TR);
      \draw[->,bend right=30]  (TR) to node[auto,swap,labelsize] {$V$} (BR);
      \draw[->,bend left=30]  (TR) to node[auto,labelsize] {$V'$} (BR);
      \draw[->]  (TL) to node[auto,swap,labelsize] {$V$} 
      (BL);
      \draw[->] (BL) to node[auto,swap,labelsize](B) {$T$} 
      (BR);
      \draw[2cell] (1.75,0) to node[auto,labelsize] {$\theta$} (2.25,0);
      \tzsquareupswap{1}{0}{$v$}
\end{tikzpicture}
\quad=\quad
\begin{tikzpicture}[baseline=-\the\dimexpr\fontdimen22\textfont2\relax ]
      \node (TL) at (0,0.75)  {$\cat{A}$};
      \node (TR) at (2,0.75)  {$\cat{A}$};
      \node (BL) at (0,-0.75) {$\cat{C}$};
      \node (BR) at (2,-0.75) {$\cat{C}$};
      \draw[->] (TL) to node[auto,labelsize](T) {$\id{\cat{A}}$} 
      (TR);
      \draw[->]  (TR) to node[auto,labelsize] {$V'$} (BR);
      \draw[->,bend right=30]  (TL) to node[auto,swap,labelsize] {$V$} (BL);
      \draw[->,bend left=30]  (TL) to node[auto,labelsize] {$V'$} 
      (BL);
      \draw[->] (BL) to node[auto,swap,labelsize](B) {$T$} 
      (BR);
      \draw[2cell] (-0.25,0) to node[auto,labelsize] {$\theta$} (0.25,0);
      \tzsquareup{1}{0}{$v'$}
\end{tikzpicture}
\]
holds,
there exists a unique natural transformation $\sigma\colon K\Longrightarrow K'$
such that $\theta=U\circ \sigma$, where $K\colon 
\cat{A}\longrightarrow\cat{C}^\monoid{T}$
and $K'\colon\cat{A}\longrightarrow\cat{C}^\monoid{T}$ are the functors 
corresponding to $(\cat{A},V,v)$ and 
$(\cat{A},V',v')$ respectively.
\end{enumerate}
In more conceptual terms, what is asserted is that we have a family of 
isomorphisms of categories 
\[
\tCAT(\cat{A},\cat{C}^\monoid{T})\cong
\tCAT(\cat{A},\cat{C})^{\tCAT(\cat{A},\monoid{T})}
\]
natural in $\cat{A}\in\tCAT$,
where the right hand side denotes the Eilenberg--Moore category of the 
monad $\tCAT(\cat{A},\monoid{T})$; note that 
$\tCAT(\cat{A},-)$ is a 2-functor and therefore preserves monads.

As with any universal characterisation, the above property characterises 
the triple $(\cat{C}^\monoid{T},U,u)$ uniquely up to
unique isomorphisms.
One can also express this universal property in terms of the standard 
2-categorical limit notions, such as lax limit or weighted 2-limit 
\cite{Street_limits}.

\subsection{Pseudo double categories}
\label{subsec:pseudo_double_cat}
We shall see that our category of models admit a similar characterisation,
in a different setting:
instead of the 2-category $\tCAT$, we will work within
the \emph{pseudo double category} $\PPROF$.
The notion of pseudo double category is due to Grandis and Par{\'e}~\cite{GP1},
and it generalises the classical notion of double category \cite{Ehresmann}
in a way similar to the generalisation of 2-categories to bicategories.

Recall that a double category consists of \defemph{objects} $A$, 
\defemph{vertical 
morphisms}
$f\colon A\longrightarrow A^\prime$,
\defemph{horizontal morphisms} $X\colon A\pto B$
and \defemph{squares}
\[
\begin{tikzpicture}[baseline=-\the\dimexpr\fontdimen22\textfont2\relax ]
      \node (TL) at (0,1.5)  {$A$};
      \node (TR) at (2,1.5)  {$B$};
      \node (BL) at (0,0) {$A^\prime$};
      \node (BR) at (2,0) {$B^\prime$,};
      \draw[pto] (TL) to node[auto,labelsize](T) {$X$} (TR);
      \draw[->]  (TR) to node[auto,labelsize] {$g$} (BR);
      \draw[->]  (TL) to node[auto,swap,labelsize] {$f$} (BL);
      \draw[pto] (BL) to node[auto,swap,labelsize](B) {$X^\prime$} (BR);
      \tzsquare{1}{0.75}{$\alpha$}
\end{tikzpicture}
\]
together with several identity and composition operations,
namely: 
\begin{itemize}
\item for each object $A$ we have the \defemph{vertical identity morphism} 
$\id{A}\colon A\longrightarrow A$;
\item for each composable pair of vertical morphisms $f\colon A\longrightarrow 
A'$ and $f'\colon A'\longrightarrow A''$ we have the \defemph{vertical 
composition}
$f'\circ f\colon A\longrightarrow A''$;
\item for each horizontal morphism $X\colon A\pto B$ we have the 
\defemph{vertical identity square}
\[
\begin{tikzpicture}[baseline=-\the\dimexpr\fontdimen22\textfont2\relax ]
      \node (TL) at (0,1.5)  {$A$};
      \node (TR) at (2,1.5)  {$B$};
      \node (BL) at (0,0) {$A$};
      \node (BR) at (2,0) {$B$;};
      \draw[pto] (TL) to node[auto,labelsize](T) {$X$} (TR);
      \draw[->]  (TR) to node[auto,labelsize] {$\id{B}$} (BR);
      \draw[->]  (TL) to node[auto,swap,labelsize] {$\id{A}$} (BL);
      \draw[pto] (BL) to node[auto,swap,labelsize](B) {$X$} (BR);
      \tzsquare{1}{0.75}{$\id{X}$}
\end{tikzpicture}
\]
\item for each vertically composable pair of squares 
\[
\begin{tikzpicture}[baseline=-\the\dimexpr\fontdimen22\textfont2\relax ]
      \node (TL) at (0,0.75)  {$A$};
      \node (TR) at (2,0.75)  {$B$};
      \node (BL) at (0,-0.75) {$A^\prime$};
      \node (BR) at (2,-0.75) {$B^\prime$};
      \draw[pto] (TL) to node[auto,labelsize](T) {$X$} (TR);
      \draw[->]  (TR) to node[auto,labelsize] {$g$} (BR);
      \draw[->]  (TL) to node[auto,swap,labelsize] {$f$} (BL);
      \draw[pto] (BL) to node[auto,swap,labelsize](B) {$X^\prime$} (BR);
      \tzsquare{1}{0}{$\alpha$}
\end{tikzpicture}
\text{ and }
\begin{tikzpicture}[baseline=-\the\dimexpr\fontdimen22\textfont2\relax ]
      \node (TL) at (0,0.75)  {$A'$};
      \node (TR) at (2,0.75)  {$B'$};
      \node (BL) at (0,-0.75) {$A''$};
      \node (BR) at (2,-0.75) {$B''$};
      \draw[pto] (TL) to node[auto,labelsize](T) {$X'$} (TR);
      \draw[->]  (TR) to node[auto,labelsize] {$g'$} (BR);
      \draw[->]  (TL) to node[auto,swap,labelsize] {$f'$} (BL);
      \draw[pto] (BL) to node[auto,swap,labelsize](B) {$X''$} (BR);
      \tzsquare{1}{0}{$\alpha'$}
\end{tikzpicture}
\]
we have the \defemph{vertical composition}
\[
\begin{tikzpicture}[baseline=-\the\dimexpr\fontdimen22\textfont2\relax ]
      \node (TL) at (0,1.5)  {$A$};
      \node (TR) at (2,1.5)  {$B$};
      \node (BL) at (0,0) {$A''$};
      \node (BR) at (2,0) {$B''$;};
      \draw[pto] (TL) to node[auto,labelsize](T) {$X$} (TR);
      \draw[->]  (TR) to node[auto,labelsize] {$g'\circ g$} (BR);
      \draw[->]  (TL) to node[auto,swap,labelsize] {$f'\circ f$} (BL);
      \draw[pto] (BL) to node[auto,swap,labelsize](B) {$X''$} (BR);
      \tzsquare{1}{0.75}{$\alpha'\circ \alpha$}
\end{tikzpicture}
\]
\end{itemize}
and symmetrically:
\begin{itemize}
\item for each object $A$ we have the \defemph{horizontal identity morphism} 
$I_{A}\colon A\pto A$;
\item for each composable pair of horizontal morphisms $X\colon A\pto B$ and 
$Y\colon B\pto C$ we have the \defemph{horizontal composition}
$Y\otimes X\colon A\pto  C$;
\item for each vertical morphism $f\colon A\longrightarrow A'$ we have the 
\defemph{horizontal identity square}
\[
\begin{tikzpicture}[baseline=-\the\dimexpr\fontdimen22\textfont2\relax ]
      \node (TL) at (0,1.5)  {$A$};
      \node (TR) at (2,1.5)  {$A$};
      \node (BL) at (0,0) {$A'$};
      \node (BR) at (2,0) {$A'$;};
      \draw[pto] (TL) to node[auto,labelsize](T) {$I_A$} (TR);
      \draw[->]  (TR) to node[auto,labelsize] {$f$} (BR);
      \draw[->]  (TL) to node[auto,swap,labelsize] {$f$} (BL);
      \draw[pto] (BL) to node[auto,swap,labelsize](B) {$I_{A'}$} (BR);
      \tzsquare{1}{0.75}{$I_f$}
\end{tikzpicture}
\]
\item for each horizontally composable pair of squares 
\[
\begin{tikzpicture}[baseline=-\the\dimexpr\fontdimen22\textfont2\relax ]
      \node (TL) at (0,0.75)  {$A$};
      \node (TR) at (2,0.75)  {$B$};
      \node (BL) at (0,-0.75) {$A^\prime$};
      \node (BR) at (2,-0.75) {$B^\prime$};
      \draw[pto] (TL) to node[auto,labelsize](T) {$X$} (TR);
      \draw[->]  (TR) to node[auto,labelsize] {$g$} (BR);
      \draw[->]  (TL) to node[auto,swap,labelsize] {$f$} (BL);
      \draw[pto] (BL) to node[auto,swap,labelsize](B) {$X^\prime$} (BR);
      \tzsquare{1}{0}{$\alpha$}
\end{tikzpicture}
\text{ and }
\begin{tikzpicture}[baseline=-\the\dimexpr\fontdimen22\textfont2\relax ]
      \node (TL) at (0,0.75)  {$B$};
      \node (TR) at (2,0.75)  {$C$};
      \node (BL) at (0,-0.75) {$B'$};
      \node (BR) at (2,-0.75) {$C'$};
      \draw[pto] (TL) to node[auto,labelsize](T) {$Y$} (TR);
      \draw[->]  (TR) to node[auto,labelsize] {$h$} (BR);
      \draw[->]  (TL) to node[auto,swap,labelsize] {$g$} (BL);
      \draw[pto] (BL) to node[auto,swap,labelsize](B) {$Y'$} (BR);
      \tzsquare{1}{0}{$\beta$}
\end{tikzpicture}
\]
we have the \defemph{horizontal composition}
\[
\begin{tikzpicture}[baseline=-\the\dimexpr\fontdimen22\textfont2\relax ]
      \node (TL) at (0,1.5)  {$A$};
      \node (TR) at (2,1.5)  {$C$};
      \node (BL) at (0,0) {$A'$};
      \node (BR) at (2,0) {$C'$.};
      \draw[pto] (TL) to node[auto,labelsize](T) {$Y\otimes X$} (TR);
      \draw[->]  (TR) to node[auto,labelsize] {$h$} (BR);
      \draw[->]  (TL) to node[auto,swap,labelsize] {$f$} (BL);
      \draw[pto] (BL) to node[auto,swap,labelsize](B) {$Y'\otimes X'$} (BR);
      \tzsquare{1}{0.75}{$\beta\otimes \alpha$}
\end{tikzpicture}
\]
\end{itemize}
These identity and composition operations are required to satisfy several 
axioms, such as the unit and associativity axioms for vertical 
(resp.~horizontal) 
identity and composition, as well as the axiom 
$\id{I_A}=I_\id{A}$ for each object $A$ and 
the \emph{interchange law}, saying that whenever we have a configuration 
of squares as in 
\[
\begin{tikzpicture}[baseline=-\the\dimexpr\fontdimen22\textfont2\relax ]
      \node (TL) at (0,3)  {$\bullet$};
      \node (TM) at (2,3)  {$\bullet$};
      \node (TR) at (4,3)  {$\bullet$};
      \node (ML) at (0,1.5)  {$\bullet$};
      \node (MM) at (2,1.5)  {$\bullet$};
      \node (MR) at (4,1.5)  {$\bullet$};
      \node (BL) at (0,0)  {$\bullet$};
      \node (BM) at (2,0)  {$\bullet$};
      \node (BR) at (4,0)  {$\bullet$,};
      \draw[pto] (TL) to (TM);
      \draw[pto] (TM) to (TR);
      \draw[pto] (ML) to (MM);
      \draw[pto] (MM) to (MR);
      \draw[pto] (BL) to (BM);
      \draw[pto] (BM) to (BR);
      \draw[->]  (TR) to (MR);
      \draw[->]  (MR) to (BR);
      \draw[->]  (TM) to (MM);
      \draw[->]  (MM) to (BM);
      \draw[->]  (TL) to (ML);
      \draw[->]  (ML) to (BL);
      \tzsquare{1}{2.25}{$\alpha$}
      \tzsquare{1}{0.75}{$\alpha'$}
      \tzsquare{3}{2.25}{$\beta$}
      \tzsquare{3}{0.75}{$\beta'$}
\end{tikzpicture}
\]
$(\beta'\otimes \alpha')\circ(\beta\otimes \alpha)=(\beta'\circ 
\beta)\otimes(\alpha'\circ\alpha)$ holds.

\medskip

Some naturally arising double-category-like structure, 
including $\PPROF$, are 
such that whose vertical morphisms are homomorphism-like (e.g., functors)
and whose horizontal morphisms are bimodule-like (e.g., profunctors);
cf.~\cite[Section~1]{Shulman_framed}.
However, a problem crops up from the bimodule-like horizontal
morphisms: 
in general, their composition is not unital nor associative on the nose.
Therefore such structures fail to form  
(strict) double categories, but instead form 
\emph{pseudo} (or \emph{weak}) \emph{double 
categories}~\cite{GP1,Leinster_book,Garner_thesis,Shulman_framed},
in which horizontal composition is allowed to 
be unital and associative up to suitable isomorphism.\footnote{In the 
literature, definitions of pseudo double category differ as to 
whether to weaken horizontal compositions or vertical compositions.
We follow  \cite{Garner_thesis,Shulman_framed} and weaken horizontal 
compositions, but note that 
the original paper~\cite{GP1} weakens vertical compositions.}
We refer the reader to e.g.~\cite[Section~2.1]{Garner_thesis}
for a detailed definition of pseudo double category.

\begin{example}[\cite{GP1}]
\label{ex:H_construction}
Let $\tcat{B}$ be a bicategory.
This induces a pseudo double category $\dcat{H}\tcat{B}$, given as follows:
\begin{itemize}
\item an object of $\dcat{H}\tcat{B}$ is an object of $\tcat{B}$;
\item all vertical morphisms of $\dcat{H}\tcat{B}$ are vertical identity 
morphisms;
\item a horizontal morphism of $\dcat{H}\tcat{B}$ is a 1-cell of $\tcat{B}$;
\item a square of $\dcat{H}\tcat{B}$ is a 2-cell of $\tcat{B}$.
\end{itemize}

Conversely, for any pseudo double category $\dcat{D}$,
we obtain a bicategory $\tcat{H}\dcat{D}$ given as follows:
\begin{itemize}
\item an object of $\tcat{H}\dcat{D}$ is an object of $\dcat{D}$;
\item a 1-cell of $\tcat{H}\dcat{D}$ is a horizontal morphism of $\dcat{D}$;
\item a 2-cell of $\tcat{H}\dcat{D}$ is a square in $\dcat{D}$ whose 
horizontal source and target are both vertical identity morphisms.
\end{itemize}
\end{example}

Let us introduce the pseudo double category $\PPROF$.
\begin{definition}[{\cite[Section~3.1]{GP1}}]
We define the pseudo double category $\PPROF$ as follows.
\begin{itemize}
\item An object is a large category.
\item A vertical morphism from $\cat{A}$ to $\cat{A'}$ is a functor 
$F\colon \cat{A}\longrightarrow\cat{A'}$.
\item A horizontal morphism from $\cat{A}$ to $\cat{B}$ is 
a profunctor $H\colon\cat{A}\pto \cat{B}$, i.e.,
a functor $H\colon \cat{B}^\op\times\cat{A}\longrightarrow\SET$. 
Horizontal identities and horizontal compositions are the same as 
in Definition~\ref{def:profunctor}.
\item A square as in 
\[
\begin{tikzpicture}[baseline=-\the\dimexpr\fontdimen22\textfont2\relax ]
      \node (TL) at (0,1.5)  {$\cat{A}$};
      \node (TR) at (2,1.5)  {$\cat{B}$};
      \node (BL) at (0,0) {$\cat{A^\prime}$};
      \node (BR) at (2,0) {$\cat{B^\prime}$};
      \draw[pto] (TL) to node[auto,labelsize](T) {$H$} (TR);
      \draw[->]  (TR) to node[auto,labelsize] {$G$} (BR);
      \draw[->]  (TL) to node[auto,swap,labelsize] {$F$} (BL);
      \draw[pto] (BL) to node[auto,swap,labelsize](B) {$H^\prime$} (BR);
      \tzsquare{1}{0.75}{$\alpha$}
\end{tikzpicture}
\] 
is a natural transformation of type
\[
\begin{tikzpicture}[baseline=-\the\dimexpr\fontdimen22\textfont2\relax ]
      \node (L) at (0,0)  {$\cat{B}^\op\times\cat{A}$};
      \node (T) at (2.5,-0.7)  {$\cat{B'}^\op\times\cat{A'}$};
      \node (R) at (5,0) {$\SET$.};
      \draw[->, bend left=-5] (L) to node[auto,swap,labelsize]{$G^\op\times F$} 
      (T);
      \draw[->, bend left=-5] (T) to node[auto,swap,labelsize]{$H'$} (R);
      \draw[->, bend left=20] (L) to node[auto,labelsize]{$H$} 
      (R);
      \tzsquare{2.5}{0}{$\alpha$}
\end{tikzpicture}
\]
\end{itemize}
\end{definition}

Observe that $\PROF=\tcat{H}\PPROF$,
using the construction introduced in Example~\ref{ex:H_construction}.

Given a pseudo double category $\dcat{D}$,
denote by $\dcat{D}^\op$, $\dcat{D}^\co$, and $\dcat{D}^\coop$
the pseudo double categories obtained from $\dcat{D}$ by reversing
the horizontal direction, reversing the vertical direction 
and reversing both the horizontal and vertical directions,
respectively.
In the following we shall mainly work within $\PPROF^\op$,
though most of the diagrams are symmetric in the horizontal direction and 
this makes little difference.
(In fact, the pseudo double category defined in \cite[Section~3.1]{GP1} 
amounts to our $\PPROF^\op$, because our convention on the direction of 
profunctors differs from theirs.)

\subsection{The universality of categories of models}

Let $\cat{M}$ be a metatheory, $\monoid{T}=(T,e,m)$ be a theory in $\cat{M}$,
$\cat{C}$ be a large category, and $\Phi=(\Phi, 
\overline{\phi}_\cdot,\overline{\phi})$ be a metamodel of $\cat{M}$ in 
$\cat{C}$.
Recall from Section~\ref{sec:metamodel_model} that in the data of 
the metamodel 
$(\Phi,\overline{\phi}_\cdot,\overline{\phi})$,
the natural transformations 
\[
((\overline{\phi}_\cdot)_C\colon 1\longrightarrow \Phi_I(C,C))_{C\in\cat{C}}
\]
and 
\[
((\overline{\phi}_{X,Y})_{A,B,C}\colon \Phi_Y(B,C)\times 
\Phi_X(A,B)\longrightarrow \Phi_{Y\otimes 
X}(A,C))_{X,Y\in\cat{M},A,B,C\in\cat{C}},
\]
may be replaced by
the natural transformations 
\[
((\phi_\cdot)_{A,B}\colon\cat{C}(A,B)\longrightarrow 
\Phi_I(A,B))_{A,B\in\cat{C}}
\]
and
\[
((\phi_{X,Y})_{A,B}\colon (\Phi_Y \ptensorrev 
\Phi_X)(A,B)\longrightarrow\Phi_{Y\otimes 
X}(A,B))_{X,Y\in\cat{M},A,B\in\cat{C}},
\]
respectively.
In this section we shall mainly use the expression of 
metamodel via the data $(\Phi,\phi_\cdot,\phi)$.
The category of models $\Mod{\monoid{T}}{(\cat{C},\Phi)}$,
henceforth abbreviated as $\Mod{\monoid{T}}{\cat{C}}$,
defined in Definition~\ref{def:metamodel_model} admits a canonical 
forgetful functor $U\colon 
\Mod{\monoid{T}}{\cat{C}}\longrightarrow\cat{C}$
and a natural transformation (a square in $\PPROF^\op$) $u$ as in
\[
\begin{tikzpicture}[baseline=-\the\dimexpr\fontdimen22\textfont2\relax ]
      \node (TL) at (0,0.75)  {$\Mod{\monoid{T}}{\cat{C}}$};
      \node (TR) at (5,0.75)  {$\Mod{\monoid{T}}{\cat{C}}$};
      \node (BL) at (0,-0.75) {$\cat{C}$};
      \node (BR) at (5,-0.75) {$\cat{C}.$};
      \draw[->]  (TL) to node[auto,swap,labelsize] {$U$} (BL);
      \draw[pto]  (TL) to node[auto,labelsize] 
      {$\Mod{\monoid{T}}{\cat{C}}(-,-)$} (TR);
      \draw[pto]  (BL) to node[auto,swap,labelsize] {$\Phi_T$} (BR);
      \draw[->]  (TR) to node[auto,labelsize](B) {$U$} (BR);
      \tzsquare{2.5}{0}{$u$}
\end{tikzpicture} 
\] 
Concretely, $u$ is a natural transformation 
\[
(u_{(C,\xi),(C',\xi')}\colon 
\Mod{\monoid{T}}{\cat{C}}((C,\xi),(C',\xi'))\longrightarrow
\Phi_T(C,C'))_{(C,\xi),(C',\xi')\in\Mod{\monoid{T}}{\cat{C}}}
\]
whose $((C,\xi),(C',\xi'))$-th component 
maps each morphism $f\colon (C,\xi)\longrightarrow (C',\xi')$ in 
$\Mod{\monoid{T}}{\cat{C}}$ to the element 
$\Phi_T(C,f)(\xi)=\Phi_T(f,C')(\xi')\in \Phi_T(C,C')$.
Alternatively, by the Yoneda lemma, $u$ may be equivalently given as 
a natural transformation 
\[
(\overline{u}_{(C,\xi)}\colon 
1\longrightarrow
\Phi_T(C,C))_{(C,\xi)\in\Mod{\monoid{T}}{\cat{C}}}
\]
whose $(C,\xi)$-th component maps the unique element of $1$ to 
$\xi\in\Phi_T(C,C)$.

We claim that the triple $(\Mod{\monoid{T}}{\cat{C}},U,u)$ 
has a certain universal property.
\begin{definition}
\label{def:double_cone}
We define a \defemph{vertical double cone over $\Phi(\monoid{T})$}
to be a triple $(\cat{A},V,v)$ consisting of a large category $\cat{A}$,
a functor $V\colon\cat{A}\longrightarrow\cat{C}$, and a square $v$ in 
$\PPROF^\op$ of type 
\[
\begin{tikzpicture}[baseline=-\the\dimexpr\fontdimen22\textfont2\relax ]
      \node (TL) at (0,0.75)  {$\cat{A}$};
      \node (TR) at (2,0.75)  {$\cat{A}$};
      \node (BL) at (0,-0.75) {$\cat{C}$};
      \node (BR) at (2,-0.75) {$\cat{C},$};
      \draw[->]  (TL) to node[auto,swap,labelsize] {$V$} (BL);
      \draw[pto]  (TL) to node[auto,labelsize] 
      {$\cat{A}(-,-)$} (TR);
      \draw[pto]  (BL) to node[auto,swap,labelsize] {$\Phi_T$} (BR);
      \draw[->]  (TR) to node[auto,labelsize](B) {$V$} (BR);
      \tzsquare{1}{0}{$v$}
\end{tikzpicture} 
\]
satisfying the following equations:
\begin{equation*}
\begin{tikzpicture}[baseline=-\the\dimexpr\fontdimen22\textfont2\relax ]
      \node (TL) at (0,0.75)  {$\cat{A}$};
      \node (TR) at (2,0.75)  {$\cat{A}$};
      \node (BL) at (0,-0.75) {$\cat{C}$};
      \node (BR) at (2,-0.75) {$\cat{C}$};
      \draw[pto] (TL) to node[auto,labelsize](T) {$\cat{A}(-,-)$} 
      (TR);
      \draw[->]  (TR) to node[auto,labelsize] {$V$} (BR);
      \draw[->]  (TL) to node[auto,swap,labelsize] {$V$} 
      (BL);
      \draw[pto, bend left=30] (BL) to node[auto,labelsize](B) 
      {$\Phi_T$} 
      (BR);
      \draw[pto, bend right=30] (BL) to node[auto,swap,labelsize](B) 
      {$\Phi_I$} 
      (BR);
      \tzsquare{1}{0.25}{$v$}
      \tzsquare{1}{-0.75}{$\Phi_e$}
\end{tikzpicture}
\quad=\quad
\begin{tikzpicture}[baseline=-\the\dimexpr\fontdimen22\textfont2\relax ]
      \node (TL) at (0,0.75)  {$\cat{A}$};
      \node (TR) at (2,0.75)  {$\cat{A}$};
      \node (BL) at (0,-0.75) {$\cat{C}$};
      \node (BR) at (2,-0.75) {$\cat{C}$};
      \draw[pto] (TL) to node[auto,labelsize](T) {$\cat{A}(-,-)$} 
      (TR);
      \draw[->]  (TR) to node[auto,labelsize] {$V$} (BR);
      \draw[->]  (TL) to node[auto,swap,labelsize] {$V$} 
      (BL);
      \draw[pto, bend left=30] (BL) to node[auto,labelsize](B) {$\cat{C}(-,-)$} 
      (BR);
      \draw[pto, bend right=30] (BL) to node[auto,swap,labelsize](B) 
      {$\Phi_I$} 
      (BR);
      \tzsquare{1}{0.25}{$I_{V}$}
      \tzsquare{1}{-0.75}{$\phi_\cdot$}
\end{tikzpicture}
\end{equation*}
\begin{equation*}
\begin{tikzpicture}[baseline=-\the\dimexpr\fontdimen22\textfont2\relax ]
      \node (TL) at (0,0.75)  {$\cat{A}$};
      \node (TR) at (2,0.75)  {$\cat{A}$};
      \node (BL) at (0,-0.75) {$\cat{C}$};
      \node (BR) at (2,-0.75) {$\cat{C}$};
      \draw[pto] (TL) to node[auto,labelsize](T) {$\cat{A}(-,-)$} 
      (TR);
      \draw[->]  (TR) to node[auto,labelsize] {$V$} (BR);
      \draw[->]  (TL) to node[auto,swap,labelsize] {$V$} 
      (BL);
      \draw[pto, bend left=30] (BL) to node[auto,labelsize](B) 
      {$\Phi_T$} 
      (BR);
      \draw[pto, bend right=30] (BL) to node[auto,swap,labelsize](B) 
      {$\Phi_{T\otimes T}$} 
      (BR);
      \tzsquare{1}{0.25}{$v$}
      \tzsquare{1}{-0.75}{$\Phi_m$}
\end{tikzpicture}
\quad=\quad
\begin{tikzpicture}[baseline=-\the\dimexpr\fontdimen22\textfont2\relax ]
      \node (TL) at (0,0.75)  {$\cat{A}$};
      \node (TR) at (2,0.75)  {$\cat{A}$};
      \node (BL) at (0,-0.75) {$\cat{C}$};
      \node (BR) at (2,-0.75) {$\cat{C}$};
      \node (TRR) at (4,0.75)  {$\cat{A}$};
      \node (BRR) at (4,-0.75) {$\cat{C}$};
      \node at (4.05,-0.75) {$\phantom{\cat{C}}$.};
      \draw[pto] (TL) to node[auto,labelsize](T) {$\cat{A}(-,-)$} 
      (TR);
      \draw[pto] (TR) to node[auto,labelsize](T2) {$\cat{A}(-,-)$} 
      (TRR);
      \draw[->]  (TR) to node[auto,labelsize] {$V$} (BR);
      \draw[->]  (TRR) to node[auto,labelsize] {$V$} (BRR);
      \draw[->]  (TL) to node[auto,swap,labelsize] {$V$} 
      (BL);
      \draw[pto] (BL) to node[auto,labelsize](B) {$\Phi_T$} 
      (BR);
      \draw[pto] (BR) to node[auto,labelsize](B) {$\Phi_T$} 
      (BRR);
      \draw[pto, bend right=30] (BL) to node[auto,swap,labelsize](B) 
      {$\Phi_{T\otimes T}$} 
      (BRR);
      \draw[pto, bend left=30] (TL) to node[auto,labelsize](B) 
      {$\cat{A}(-,-)$} 
      (TRR);
      \tzsquare{1}{0}{$v$}
      \tzsquare{3}{0}{$v$}
      \tzsquare{2}{-1.1}{$\phi_{T,T}$}
      \tzsquare{2}{1.1}{$\cong$}
\end{tikzpicture}
\end{equation*}
\end{definition}
Using this notion, we can state the universal property of 
the triple $(\Mod{\monoid{T}}{\cat{C}},$ $U,u)$, just as in the case of 
Eilenberg--Moore categories.

\begin{theorem}
\label{thm:double_categorical_univ_property}
Let $\cat{M}=(\cat{M},I,\otimes)$ be a metatheory, $\monoid{T}=(T,e,m)$ be a 
theory in $\cat{M}$, $\cat{C}$ be a large category, and 
$\Phi=(\Phi,\phi_\cdot,\phi)$ be a metamodel of $\cat{M}$ in $\cat{C}$.
The triple $(\Mod{\monoid{T}}{\cat{C}},U,u)$ defined above is 
a universal vertical double cone over $\Phi(\monoid{T})$,
namely:\begin{enumerate}
\item it is a vertical double cone over $\Phi(\monoid{T})$;
\item for any vertical double cone $(\cat{A},V,v)$ over $\Phi(\monoid{T})$,
there exists a unique functor $K\colon\cat{A}\longrightarrow
\Mod{\monoid{T}}{\cat{C}}$ such that 
\[
\begin{tikzpicture}[baseline=-\the\dimexpr\fontdimen22\textfont2\relax ]
      \node (TL) at (0,0.75)  {$\cat{A}$};
      \node (TR) at (2,0.75)  {$\cat{A}$};
      \node (BL) at (0,-0.75) {$\cat{C}$};
      \node (BR) at (2,-0.75) {$\cat{C}$};
      \draw[pto] (TL) to node[auto,labelsize](T) {$\cat{A}(-,-)$} 
      (TR);
      \draw[->]  (TR) to node[auto,labelsize] {$V$} (BR);
      \draw[->]  (TL) to node[auto,swap,labelsize] {$V$} 
      (BL);
      \draw[pto] (BL) to node[auto,swap,labelsize](B) {$\Phi_T$} (BR);
      \tzsquare{1}{0}{$v$}
\end{tikzpicture}  
\quad=\quad
\begin{tikzpicture}[baseline=-\the\dimexpr\fontdimen22\textfont2\relax ]
      \node (TL) at (0,1.5)  {$\cat{A}$};
      \node (TR) at (5,1.5)  {$\cat{A}$};
      \node (BL) at (0,0) {$\Mod{\monoid{T}}{\cat{C}}$};
      \node (BR) at (5,0) {$\Mod{\monoid{T}}{\cat{C}}$};
      \node (BBL) at (0,-1.5) {$\cat{C}$};
      \node (BBR) at (5,-1.5) {$\cat{C}$};
      \draw[pto] (TL) to node[auto,labelsize](T) {${\cat{A}}(-,-)$} 
      (TR);
      \draw[->]  (TR) to node[auto,labelsize] {$K$} (BR);
      \draw[->]  (TL) to node[auto,swap,labelsize] {$K$} 
      (BL);
      \draw[pto] (BL) to node[auto,labelsize](B) 
      {$\Mod{\monoid{T}}{\cat{C}}(-,-)$} 
      (BR);
      \draw[->]  (BR) to node[auto,labelsize] {$U$} (BBR);
      \draw[->]  (BL) to node[auto,swap,labelsize] {$U$} (BBL);
      \draw[pto] (BBL) to node[auto,swap,labelsize](B) {$\Phi_T$} (BBR);
      \tzsquare{2.5}{0.75}{$I_{K}$}
      \tzsquare{2.5}{-0.75}{$u$} 
\end{tikzpicture} 
\]
holds;
\item for any pair of vertical cones $(\cat{A},V,v)$ and $(\cat{A'},V',v')$
over $\Phi(\monoid{T})$, any horizontal morphism $H\colon \cat{A}\pto \cat{A'}$
in $\PPROF^\op$ and any square 
\[
\begin{tikzpicture}[baseline=-\the\dimexpr\fontdimen22\textfont2\relax ]
      \node (TL) at (0,0.75)  {$\cat{A}$};
      \node (TR) at (2,0.75)  {$\cat{A'}$};
      \node (BL) at (0,-0.75) {$\cat{C}$};
      \node (BR) at (2,-0.75) {$\cat{C}$};
      \draw[pto] (TL) to node[auto,labelsize](T) {$H$} (TR);
      \draw[->]  (TR) to node[auto,labelsize] {$V'$} (BR);
      \draw[->]  (TL) to node[auto,swap,labelsize] {$V$} (BL);
      \draw[pto] (BL) to node[auto,swap,labelsize](B) {$\cat{C}(-,-)$} (BR);
      \tzsquare{1}{0}{$\theta$}
\end{tikzpicture}  
\]
in $\PPROF^\op$ such that 
\begin{equation}
\label{eqn:H_theta}
\begin{tikzpicture}[baseline=-\the\dimexpr\fontdimen22\textfont2\relax ]
      \node (TL) at (0,0.75)  {$\cat{A}$};
      \node (TM) at (2,0.75)  {$\cat{A}$};
      \node (TR) at (4,0.75)  {$\cat{A'}$};
      \node (BL) at (0,-0.75) {$\cat{C}$};
      \node (BM) at (2,-0.75) {$\cat{C}$};
      \node (BR) at (4,-0.75) {$\cat{C}$};
      \draw[pto, bend left=30] (TL) to node[auto,labelsize](T) {$H$} (TR);
      \draw[pto] (TL) to node[auto,labelsize](T) {$\cat{A}(-,-)$} (TM);
      \draw[pto] (TM) to node[auto,labelsize](T) {$H$} (TR);
      \draw[->]  (TL) to node[auto,swap,labelsize] {$V$} (BL);
      \draw[->]  (TM) to node[auto,labelsize] {$V$} (BM);
      \draw[->]  (TR) to node[auto,labelsize] {$V'$} (BR);
      \draw[pto] (BL) to node[auto,swap,labelsize](B) {$\Phi_T$} (BM);
      \draw[pto] (BM) to node[auto,swap,labelsize](B) {$\cat{C}(-,-)$} (BR);
      \draw[pto, bend left=-30] (BL) to node[auto,swap,labelsize](T) {$\Phi_T$} 
      (BR);
      \tzsquare{2}{1.1}{$\cong$}
      \tzsquare{1}{0}{$v$}
      \tzsquare{3}{0}{$\theta$}
      \tzsquare{2}{-1.1}{$\cong$}
\end{tikzpicture}
\ =\ 
\begin{tikzpicture}[baseline=-\the\dimexpr\fontdimen22\textfont2\relax ]
      \node (TL) at (0,0.75)  {$\cat{A}$};
      \node (TM) at (2,0.75)  {$\cat{A'}$};
      \node (TR) at (4,0.75)  {$\cat{A'}$};
      \node (BL) at (0,-0.75) {$\cat{C}$};
      \node (BM) at (2,-0.75) {$\cat{C}$};
      \node (BR) at (4,-0.75) {$\cat{C}$};
      \draw[pto, bend left=30] (TL) to node[auto,labelsize](T) {$H$} (TR);
      \draw[pto] (TL) to node[auto,labelsize](T) {$H$} (TM);
      \draw[pto] (TM) to node[auto,labelsize](T) {$\cat{A'}(-,-)$} (TR);
      \draw[->]  (TL) to node[auto,swap,labelsize] {$V$} (BL);
      \draw[->]  (TM) to node[auto,labelsize] {$V'$} (BM);
      \draw[->]  (TR) to node[auto,labelsize] {$V'$} (BR);
      \draw[pto] (BL) to node[auto,swap,labelsize](B) {$\cat{C}(-,-)$} (BM);
      \draw[pto] (BM) to node[auto,swap,labelsize](B) {$\Phi_T$} (BR);
      \draw[pto, bend left=-30] (BL) to node[auto,swap,labelsize](T) {$\Phi_T$} 
      (BR);
      \tzsquare{2}{1.1}{$\cong$}
      \tzsquare{1}{0}{$\theta$}
      \tzsquare{3}{0}{$v'$}
      \tzsquare{2}{-1.1}{$\cong$}
\end{tikzpicture}  
\end{equation}
holds, there exists a unique square 
\[
\begin{tikzpicture}[baseline=-\the\dimexpr\fontdimen22\textfont2\relax ]
      \node (TL) at (0,0.75)  {$\cat{A}$};
      \node (TR) at (5,0.75)  {$\cat{A'}$};
      \node (BL) at (0,-0.75) {$\Mod{\monoid{T}}{\cat{C}}$};
      \node (BR) at (5,-0.75) {$\Mod{\monoid{T}}{\cat{C}}$};
      \draw[pto] (TL) to node[auto,labelsize](T) {$H$} (TR);
      \draw[->]  (TR) to node[auto,labelsize] {$K'$} (BR);
      \draw[->]  (TL) to node[auto,swap,labelsize] {$K$} (BL);
      \draw[pto] (BL) to node[auto,swap,labelsize](B) 
      {$\Mod{\monoid{T}}{\cat{C}}(-,-)$} (BR);
      \tzsquare{2.5}{0}{$\sigma$}
\end{tikzpicture}  
\]
in $\PPROF^\op$ such that 
\[
\begin{tikzpicture}[baseline=-\the\dimexpr\fontdimen22\textfont2\relax ]
      \node (TL) at (0,0.75)  {$\cat{A}$};
      \node (TR) at (2,0.75)  {$\cat{A'}$};
      \node (BL) at (0,-0.75) {$\cat{C}$};
      \node (BR) at (2,-0.75) {$\cat{C}$};
      \draw[pto] (TL) to node[auto,labelsize](T) {$H$} (TR);
      \draw[->]  (TR) to node[auto,labelsize] {$V'$} (BR);
      \draw[->]  (TL) to node[auto,swap,labelsize] {$V$} (BL);
      \draw[pto] (BL) to node[auto,swap,labelsize](B) {$\cat{C}(-,-)$} (BR);
      \tzsquare{1}{0}{$\theta$}
\end{tikzpicture}  
\quad =\quad
\begin{tikzpicture}[baseline=-\the\dimexpr\fontdimen22\textfont2\relax ]
      \node (TL) at (0,1.5)  {$\cat{A}$};
      \node (TR) at (5,1.5)  {$\cat{A'}$};
      \node (ML) at (0,0) {$\Mod{\monoid{T}}{\cat{C}}$};
      \node (MR) at (5,0) {$\Mod{\monoid{T}}{\cat{C}}$};
      \node (BL) at (0,-1.5) {${\cat{C}}$};
      \node (BR) at (5,-1.5) {${\cat{C}}$};
      \draw[pto] (TL) to node[auto,labelsize](T) {$H$} (TR);
      \draw[->]  (TR) to node[auto,labelsize] {$K'$} (MR);
      \draw[->]  (MR) to node[auto,labelsize] {$U$} (BR);
      \draw[->]  (ML) to node[auto,swap,labelsize] {$U$} (BL);
      \draw[->]  (TL) to node[auto,swap,labelsize] {$K$} (ML);
      \draw[pto] (ML) to node[auto,swap,labelsize](B) 
      {$\Mod{\monoid{T}}{\cat{C}}(-,-)$} (MR);
      \draw[pto] (BL) to node[auto,swap,labelsize]{${\cat{C}}(-,-)$} (BR);
      \tzsquare{2.5}{0.75}{$\sigma$}
      \tzsquare{2.5}{-0.75}{$I_{U}$}
\end{tikzpicture}  
\]
holds, where $K$ and $K'$ are the functors corresponding to $(\cat{A},V,v)$ and 
$(\cat{A'},V',v')$ respectively.
\end{enumerate}
\end{theorem}
The above statements are taken from the definition of double limit
\cite[Section~4.2]{GP1}.
\begin{proof}[of Theorem~\ref{thm:double_categorical_univ_property}]
First, that $(\Mod{\monoid{T}}{\cat{C}},U,u)$ is a vertical double cone over 
$\Phi(\monoid{T})$ follows directly from the definition of model of $\monoid{T}$
in $\cat{C}$ with respect to $\Phi$ (Definition~\ref{def:metamodel_model}).

Given a vertical double cone $(\cat{A},V,v)$ over $\Phi(\monoid{T})$,
for each object $A\in\cat{A}$, the pair $(VA,v_{A,A}(\id{A}))$ is a
$\monoid{T}$-model in $\cat{C}$ with respect to $\Phi$,
and for each morphism $f\colon A\longrightarrow A'$ in $\cat{A}$,
the morphism $Vf$ is a $\monoid{T}$-model homomorphism from 
$(VA,v_{A,A}(\id{A}))$ to $(VA',v_{A',A'}(\id{A'}))$.
The functor $K\colon \cat{A}\longrightarrow \Mod{\monoid{T}}{\cat{C}}$
can therefore be given as $KA=(VA,v_{A,A}(\id{A}))$ and $Kf=Vf$.
The uniqueness is clear. 

Finally, given $H$ and $\theta$ as in the third clause, the equation 
(\ref{eqn:H_theta}) says that for each $A\in\cat{A}$, $A'\in\cat{A'}$ 
and $x\in H(A,A')$, the 
morphism $\theta_{A,A'}(x)\colon VA\longrightarrow V'A'$
in $\cat{C}$ satisfies $\Phi_T(VA,\theta_{A,A'}(x)) (v_{A,A}(\id{A}))= 
\Phi_T(\theta_{A,A'}(x),VA')(v'_{A',A'}(\id{A'}))$;
in other words, that $\theta_{A,A'}(x)$ is a $\monoid{T}$-model homomorphism
from $KA$ to $K'A'$.
The square $\sigma$ can then be given as the natural transformation 
with $\sigma_{A,A'}(x)=\theta_{A,A'}(x)$.
\end{proof}

\subsection{Relation to double limits}
We conclude this paper by a brief sketch of how the double categorical 
universal property (Theorem~\ref{thm:double_categorical_univ_property}) of 
categories of models in our framework can be expressed 
via the notion of double limit \cite{GP1}, thus connecting our characterisation 
to a well-established notion.
An outline of this reduction is as follows.
\begin{enumerate}
\item A theory $\monoid{T}$ in a metatheory $\cat{M}$ may be equivalently given 
as a strong monoidal functor $\monoid{T}\colon\Simp\longrightarrow\cat{M}$,
where $\Simp$ is the augmented simplex category 
(cf.~Section~\ref{sec:morphism_metatheory}).
\item A metamodel $\Phi$ of a metatheory $\cat{M}$
may be identified with a lax double functor 
$\Phi\colon\dcat{H}\Sigma(\cat{M}^\op)\longrightarrow\PPROF^\op$,
where $\Sigma$ turns a monoidal category to the corresponding one-object 
bicategory and $\dcat{H}$ turns a bicategory to the corresponding vertically 
discrete pseudo double category (see Example~\ref{ex:H_construction}).
\item Therefore given a theory $\monoid{T}$ and a metamodel $\Phi$ (in 
$\cat{C}$) of a metatheory $\cat{M}$, we obtain a lax double functor 
$\Phi(\monoid{T})\colon \dcat{H}\Sigma(\Simp^\op)\longrightarrow\PPROF^\op$
as the following composition:
\[
\begin{tikzpicture}[baseline=-\the\dimexpr\fontdimen22\textfont2\relax ]
      \node (L) at (0,0) {$\dcat{H}\Sigma(\Simp^\op)$};
      \node (M) at (3,0) {$\dcat{H}\Sigma(\cat{M}^\op)$};
      \node (R) at (6,0) {$\PPROF^\op$.};
      \draw[->] (L) to node[auto,labelsize]{$\dcat{H}\Sigma(\monoid{T}^\op)$} 
      (M);
      \draw[->] (M) to node[auto,labelsize] {$\Phi$} (R);
\end{tikzpicture}  
\]
Theorem~\ref{thm:double_categorical_univ_property} may then be interpreted as 
establishing that $\Mod{\monoid{T}}{\cat{C}}$ is (the apex of) the double limit 
of $\Phi(\monoid{T})$ in the sense of \cite{GP1}.
\end{enumerate}
We refer the reader to \cite[Section~5.4]{Fujii_thesis}
for a somewhat more detailed sketch.

\begin{corollary}
\label{cor:Mod_as_dbl_lim}
Let $\cat{M}$ be a metatheory, $\monoid{T}$ be a theory in $\cat{M}$,
$\cat{C}$ be a large category and $\Phi$ be a metamodel of $\cat{M}$ in 
$\cat{C}$.
The category $\Mod{\monoid{T}}{(\cat{C},\Phi)}$ of models of $\monoid{T}$
in $\cat{C}$ with respect to $\Phi$is the apex of 
the double limit of the lax double functor $\Phi(\monoid{T})$.
\end{corollary}


\end{document}